\documentclass[10pt, reqno]{amsart}
\usepackage{fullpage}
\usepackage{amsmath, amsthm}
\usepackage{amscd}
\usepackage{enumerate}
\usepackage{float,amsmath,amssymb,mathrsfs,bm,multirow,graphics}
\usepackage[dvips]{graphicx}
\usepackage[percent]{overpic}
\usepackage[numbers,sort&compress]{natbib}
\usepackage{xcolor}
\usepackage{todonotes}
\usepackage{amsaddr} 
\usepackage{mathtools}
\usepackage{subcaption}
\usepackage{mathtools}
\usepackage{stackengine} 
 
\usepackage{tikz}
\usepackage{tikz-cd}
\usepackage{tqft}
\usetikzlibrary{decorations.pathreplacing,decorations.markings}

\newtheorem{proposition}{Proposition}[section]
\newtheorem{theorem}[proposition]{Theorem}
\newtheorem{lemma}[proposition]{Lemma}

\newtheorem{definition}[proposition]{Definition}
\newtheorem{assumption}[proposition]{Assumption}
\newtheorem{remark}[proposition]{Remark}

\newcommand{\re}{\text{\upshape Re} \,}
\newcommand{\im}{\text{\upshape Im} \,}

 \newcommand\ben{\begin{equation*}}
 \newcommand\ebn{\end{equation*}}
 \newcommand\beq{\begin{equation}}
 \newcommand\eeq{\end{equation}}
 \newcommand\lb{\left(}
  \newcommand\rb{\right)} 

\numberwithin{equation}{section}

\allowdisplaybreaks[1]

\usepackage[colorlinks=true]{hyperref}
\hypersetup{urlcolor=blue, citecolor=red, linkcolor=blue}
\title{Non-polynomial $q$-Askey scheme: integral representations, eigenfunction properties, and polynomial limits}
\author{Jonatan Lenells and Julien Roussillon}

\address{Department of Mathematics, KTH Royal Institute of Technology, \\ 100 44 Stockholm, Sweden.}
\email{julienro@kth.se}
\email{jlenells@kth.se}

\begin{document}

\begin{abstract}
We construct a non-polynomial generalization of the $q$-Askey scheme. Whereas the elements of the $q$-Askey scheme are given by $q$-hypergeometric series, the elements of the non-polynomial scheme are given by contour integrals, whose integrands are built from Ruijsenaars' hyperbolic gamma function. Alternatively, the integrands can be expressed in terms of Faddeev's quantum dilogarithm, Woronowicz's quantum exponential, or Kurokawa's double sine function.  We present the basic properties of all the elements of the scheme, including their integral representations, joint eigenfunction properties, and polynomial limits.
\end{abstract} 

\maketitle




\section{Introduction}

\subsection{Introduction}
The Askey scheme is a way of organizing orthogonal polynomials. Each element of the scheme represents a family of polynomials, and the families at a given level can be obtained from those at the level above by taking suitable limits \cite{AW1985}. 
The Wilson and Racah polynomials lie at the top level and are thus the most general polynomials of the scheme. In addition to the classical Jacobi, Laguerre, and Hermite polynomials, the Askey scheme also includes less well-known families such as the continuous dual Hahn and Meixner--Pollaczek polynomials. 

The $q$-Askey scheme is a generalization of the Askey scheme. The polynomials in the $q$-Askey scheme are $q$-analogs of the polynomials in the Askey scheme, and the elements in the latter scheme can be recovered by taking the limit $q \to 1$ \cite{KLS2010}. The Askey--Wilson and $q$-Racah polynomials lie at the top level and all the other polynomials in the $q$-Askey scheme can be obtained from these via appropriate limit operations. We refer to Appendix \ref{AppendixB} for a list of basic properties of elements in the $q$-Askey scheme.

\begin{figure}[h!]
\tikzstyle{block} = [rectangle, draw, fill=blue!20, 
    text width=10em, text centered, rounded corners, minimum height=2em]
\tikzstyle{line} = [draw, -latex, shorten >= 4pt, shorten <= 0pt]
   \centering
\begin{tikzpicture}[node distance = .3cm, auto]
    \node[align=left] at (-7,0) {Level $1$};
    \node[align=left] at (-7,-1.2) {Level $2$};
    \node[align=left] at (-7,-2.6) {Level $3$};
    \node[align=left] at (-7,-3.95) {Level $4$};
    \node[align=left] at (-7,-5.25) {Level $5$};

    \node [block] (AW) {Askey--Wilson};
    \node [block, below left=.4cm and -0.1cm of AW] (Hahn) {Continuous dual $q$-Hahn};
    \node [block, below right=.5cm and -0.1cm of AW] (Jacobi) {Big $q$-Jacobi};
   \node [block, below=.5cm of Hahn] (Salam) {Al-Salam Chihara};
   \node [block, below=.5cm of Salam] (Xn) {Continuous big $q$-Hermite};
   \node [block, below=.5cm of Xn] (Qn) {Continuous $q$-Hermite};
   \node[block, below left=.66cm and -1.4cm of Jacobi] (Yn) {Little $q$-Jacobi};
   \node[block, below right=.66cm and -1.4cm of Jacobi] (Ln) {Big $q$-Laguerre};
   \node[block, below left= .62cm and -1.4cm of Ln] (Wn) {Little $q$-Laguerre};
   \node[block,below=.5cm of Wn] (Mn) {Little $q$-Laguerre with $\alpha=0$};
    \path [line] (AW) -- (Hahn);
    \path [line] (AW) -- (Jacobi);
    \path[line] (Hahn) -- (Salam);
    \path[line] (Salam) -- (Xn);
    \path[line] (Xn) -- (Qn);
    \path[line] (Jacobi) -- (Yn);
    \path[line] (Jacobi) -- (Ln);
    \path[line] (Ln) -- (Wn);
    \path[line] (Yn) -- (Wn);
    \path[line] (Wn) -- (Mn);
\end{tikzpicture}
\caption{The part of the $q$-Askey scheme relevant for the present paper.
\label{qaskey}}
\end{figure}
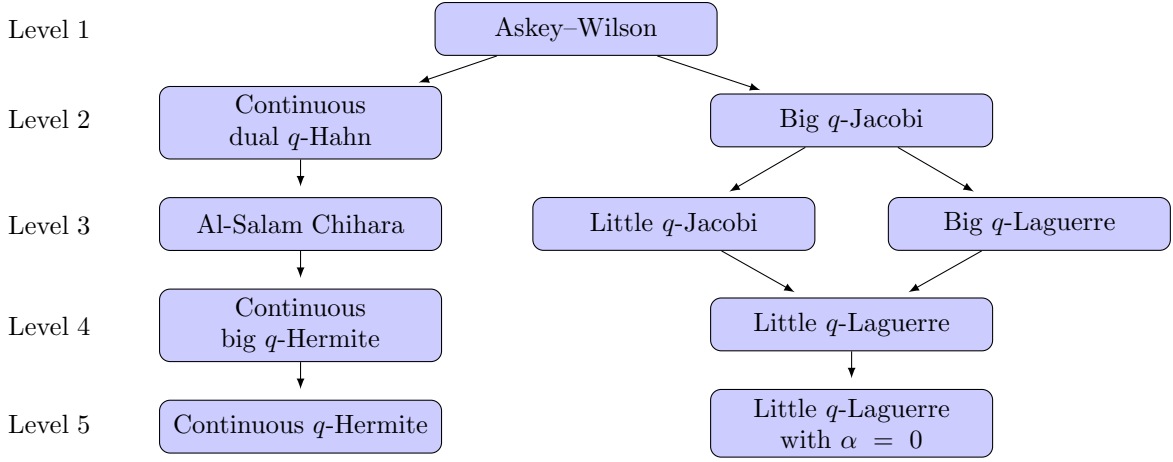

\begin{figure}[h!]
\tikzstyle{block} = [rectangle, draw, fill=blue!20, 
    text width=10em, text centered, rounded corners, minimum height=2em]
\tikzstyle{line} = [draw, -latex, shorten >= 4pt, shorten <= 0pt]
   \centering
\begin{tikzpicture}[node distance = .3cm, auto]
    \node[align=left] at (-6.3,0) {Level $1$};
    \node[align=left] at (-6.3,-1.3) {Level $2$};
    \node[align=left] at (-6.3,-2.5) {Level $3$};
    \node[align=left] at (-6.3,-3.7) {Level $4$};
    \node[align=left] at (-6.3,-4.9) {Level $5$};

    \node [block] (AW) {$\mathcal{R}\left[\substack{\theta_1\;\;\theta_t\vspace{0.1cm}\\ \theta_{\infty}\;\theta_0};\substack{\sigma_s \vspace{0.15cm} \\  \sigma_t}, b\right]$};
    \node [block, below =0.5cm of AW] (HahnJacobi) {$\mathcal{H}(b,\theta_0,\theta_t,\theta_*,\sigma_s,\nu)$};
\node [block, below left=.5cm and -.75cm  of HahnJacobi] (SalamLittleJacobi) {$\mathcal{S}(b,\theta_0,\theta_t,\sigma_s,\rho)$};
\node [block, below right=.5cm and -.75cm  of HahnJacobi] (BigLaguerre) {$\mathcal{L}(b,\theta_t,\theta,\lambda,\mu)$};
\node [block, below =.5cm of BigLaguerre] (LittleLaguerre) {$\mathcal{W}(b,\theta_t,\kappa,\omega)$};
\node [block, below=.5cm of LittleLaguerre] (M) {$\mathcal{M}(b,\zeta,\omega)$};
\node [block, below =.5cm of SalamLittleJacobi] (ContinuousBigHermite) {$\mathcal{X}(b,\theta,\sigma_s,\omega)$};
\node [block, below =.5cm of ContinuousBigHermite] (BigHermite) {$\mathcal{Q}(b,\sigma_s,\eta)$};

    \path [line] (AW) -- (HahnJacobi);
    \path [line] (HahnJacobi) -- (SalamLittleJacobi);
    \path [line] (HahnJacobi) -- (BigLaguerre);
    \path [line] (BigLaguerre) -- (LittleLaguerre);
     \path [line] (SalamLittleJacobi) -- (ContinuousBigHermite);
     \path [line] (ContinuousBigHermite) -- (BigHermite);
     \path [line] (LittleLaguerre) -- (M);

\end{tikzpicture}
\caption{The elements of the non-polynomial scheme which generalize the ones in Figure \ref{qaskey}. 
\label{nonpolynomialscheme}}
\end{figure}
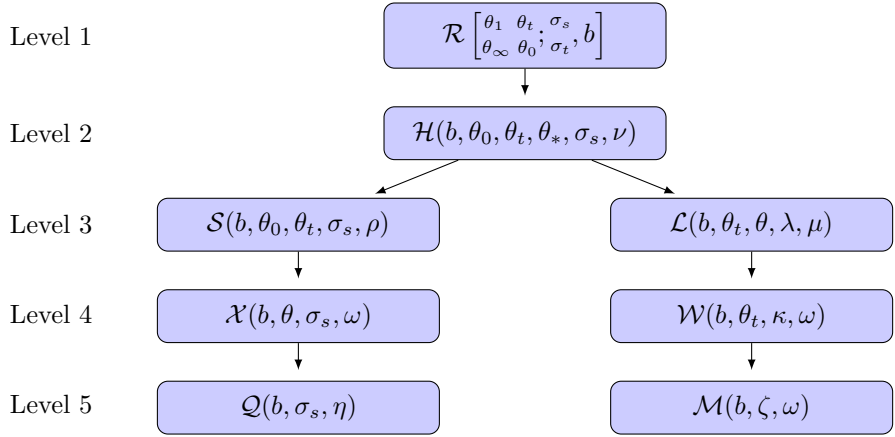

Our goal in this paper is to present a non-polynomial generalization of the $q$-Askey scheme. In fact, we will only be interested in continuous (as opposed to discrete) orthogonal polynomials and we will therefore only consider the part of the $q$-Askey scheme originating from the Askey--Wilson polynomials shown in Figure \ref{qaskey}. 
Just like the Askey and $q$-Askey schemes, the non-polynomial scheme we present has five levels, see Figure \ref{nonpolynomialscheme}.
Whereas each element in the $q$-Askey scheme is a family of polynomials, each element of the non-polynomial scheme is a meromorphic function. These meromorphic functions are given by contour integrals whose integrands are built from Ruijsenaars' hyperbolic gamma function \cite{R1996} (alternatively, the integrands can be expressed in terms of Faddeev's quantum dilogarithm \cite{FK1994}, Woronowicz's quantum exponential \cite{W2000}, or Kurokawa's double sine function \cite{K91}).
The non-polynomial scheme is a generalization of the $q$-Askey scheme in the sense that each element in Figure \ref{qaskey} is obtained from an element at the same level in Figure \ref{nonpolynomialscheme} when one of the variables is suitably discretized.


A crucial property of the elements of the $q$-Askey scheme is that they are joint eigenfunctions. More precisely, if $p_n(x)$ is a family of orthogonal polynomials in the $q$-Askey scheme, then the $p_n$ are eigenfunctions of a second-order difference operator in the $n$-variable as well as of a second-order $q$-difference operator in the $x$-variable. We refer to the corresponding eigenfunction equations as the recurrence and difference equations for $p_n$, respectively. The elements of the non-polynomial scheme are also joint eigenfunctions. Indeed, we will show that each function in the non-polynomial scheme satisfies two pairs of difference equations. In the polynomial limit, one of these pairs reduces to the recurrence relation, while the other pair reduces to the difference equation.

The first (top) level of the non-polynomial scheme consists of a single function, namely, Ruijsenaars' hypergeometric function $\mathcal{R}$\footnote{Ruijenaars' hypergeometric function is usually referred to as $R$ in the literature.  However, we denote the Askey--Wilson polynomials in \eqref{AW} by $R_n$. Thus, to avoid confusion, we denote Ruijenaars' hypergeometric function by $\mathcal R$.} \cite{R1999}. Below this top level, there are four further levels involving functions which can be obtained from $\mathcal{R}$ by taking various limits; we denote these functions by $\mathcal{H}$, $\mathcal{S}$, $\mathcal{X}$, $\mathcal{Q}$, $\mathcal{L}$, $\mathcal{W}$, and $\mathcal{M}$, where the letters are chosen so that $\mathcal{H}$ is a non-polynomial generalization of the continuous dual $q$-Hahn polynomials which are denoted by $H_n$ in Appendix \ref{AppendixB}, $\mathcal{S}$ is a non-polynomial generalization of the Al-Salam Chihara polynomials which are denoted by $S_n$ in Appendix \ref{AppendixB}, etc. 

Each of the functions in the non-polynomial scheme depends on a number of parameters as well as two variables. For example, $\mathcal{H}$ depends on the parameters $b, \theta_0, \theta_t, \theta_*$ as well as the two variables $\sigma_s, \nu$, while $\mathcal{S}$ depends on the parameters $b, \theta_0, \theta_t$ as well as the two variables $\sigma_s, \rho$, see Figure \ref{nonpolynomialscheme}.  
As a matter of notation, we will refer to the top level of the scheme (which involves $\mathcal{R}$) as the first level, to the next level (which involves $\mathcal{H}$) as the second level, etc. Each function at levels 2-5 will be defined as a limit of the function above it in the scheme. We will show that this limit exists (at least for a certain range of the two variables). We will also derive an integral representation for the function, show that it extends to a meromorphic function of each of the two variables everywhere in the complex plane, and establish two pairs of difference equations (one in each of the two variables). 
Finally, we will establish the polynomial limit to the corresponding element in the $q$-Askey scheme. Actually, two of the functions in the non-polynomial scheme, namely $\mathcal{H}$ and $\mathcal{S}$, possess two different polynomial limits: $\mathcal{H}$ reduces to the continuous dual $q$-Hahn polynomials when $\nu$ is discretized and to the big $q$-Jacobi polynomials when $\sigma_s$ is discretized; similarly, $\mathcal{S}$ reduces to the Al-Salam Chihara polynomials when $\rho$ is discretized and to the little $q$-Jacobi polynomials when $\sigma_s$ is discretized.

The orthogonal polynomials in the Askey and $q$-Askey schemes are of fundamental importance in a wide variety of fields. Elements in the $q$-Askey scheme have found applications, for instance, in models of statistical mechanics \cite{CK,Sasa,U07,USW}, in representation theory of quantum algebras \cite{AAK,BMZ,K96,K96bis,Koo93,Koo93bis,NM1,NM2,Z92}, and in the geometry of Painlev\'e equations \cite{M2016}. We expect the functions in the proposed non-polynomial scheme to also be relevant in many different contexts (see Section \ref{perspectivessec} for a related discussion). This is certainly true of the top element, Ruijsenaars' $\mathcal{R}$-function. As an example of the broad relevance of $\mathcal{R}$, we note that one of the authors recently showed \cite{R2021} that $\mathcal{R}$ is equivalent (up to a change of variables) to the Virasoro fusion kernel which is a central object in conformal field theory \cite{CBMP,PT1,PT2,Ribault}. In fact, it was in the context of conformal field theory that we first conceived of the non-polynomial scheme presented in this paper. First, in \cite{LR1}, we introduced a family of confluent Virasoro fusion kernels $\mathcal{C}_k$ while studying confluent conformal blocks of the second kind of the Virasoro algebra. Later, we realized that the $\mathcal{C}_k$ can be viewed as non-polynomial generalizations of the continuous dual $q$-Hahn and the big $q$-Jacobi polynomials, which led us to conjecture that there exists a non-polynomial generalization of the $q$-Askey scheme with the Virasoro fusion kernel as its top member \cite{LR2}. In this paper, we prove this conjecture. However, instead of adopting the Virasoro fusion kernel as the top element of the scheme, we use Ruijsenaars' $\mathcal{R}$-function as our starting point. The result of \cite{R2021} implies that these two choices are equivalent, but we have found that the scheme originating from $\mathcal{R}$ is simpler and mathematically more convenient.

\subsection{Organization of the paper}

In Section \ref{sbsec}, we introduce the function $s_b(z)$ which is the basic building block used to define the elements of the non-polynomial scheme. The eight functions $\mathcal{R}, \mathcal{H}, \mathcal{S}, \mathcal{X}, \mathcal{Q}, \mathcal{L}, \mathcal{W}, \mathcal{M}$ that make up the non-polynomial scheme are considered one by one in the eight Sections \ref{Rsec}--\ref{Msec}. In Section \ref{dualitysec}, we derive---as an easy application of the non-polynomial scheme---a few duality formulas that relate members of the $q$-Askey scheme. Section \ref{perspectivessec} presents some perspectives. In the two appendices, we have collected relevant definitions and properties of $q$-hypergeometric series and of the $q$-Askey scheme.

\subsection{Standing assumption}
Throughout the paper, we make the following assumption. 

\begin{assumption}[Restrictions on the parameters] \label{assumption} We assume that \beq(b,\theta_\infty,\theta_1,\theta_t,\theta_0,\theta_*,\theta) \in (0,\infty) \times \mathbb{R}^6.\eeq
\end{assumption}

Assumption \ref{assumption} is made primarily for simplicity; we expect that all our results admit meromorphic continuations to more general values of the parameters, such as $b\notin i\mathbb{R}$ and $(\theta_\infty,\theta_1,\theta_t,\theta_0,\theta_*,\theta) \in \mathbb{C}^6$.

\subsection*{Acknowledgements}
The authors acknowledge support from the European Research Council, Grant Agreement No. 682537, the Swedish Research Council, Grant No. 2015-05430, and the Ruth and Nils-Erik Stenb\"ack Foundation.

\section{The function $s_b(z)$}\label{sbsec}

The elements of the non-polynomial scheme presented in this article are given by contour integrals, whose integrands involve the function $s_b(z)$ defined by
\begin{equation}\label{defsb}
s_b(z)=\operatorname{exp}{\left[  i \int_0^\infty \frac{dy}{y} \left(\frac{\operatorname{sin}{2yz}}{2\operatorname{sinh}(b^{-1}y)\operatorname{sinh}(by)}-\frac{z}{y}\right)\right]}, \qquad |\im z|<\frac{Q}{2},
\end{equation}
where $Q := b + b^{-1}$. The function $s_b(z)$ is closely related to Ruijsenaars' hyperbolic gamma function \cite{R1996}, Faddeev's quantum dilogarithm function \cite{FK1994}, Woronowicz's quantum exponential function \cite{W2000}, and Kurokawa's double sine function \cite{K91}. More precisely, it is related to Ruijsenaars' hyperbolic gamma function $G$ in \cite[Eq. (A.3)]{R1999} by
\begin{align}\label{sbgbGE}
s_b(z) = G(b, b^{-1}; z).
\end{align}
It follows from \eqref{sbgbGE} and the results in \cite{R1999} that $s_b(z)$ is a meromorphic function of $z \in \mathbb{C}$ with zeros $\{z_{m,l}\}_{m,l=0}^\infty$ and poles $\{p_{m,l}\}_{m,l=0}^\infty$ located at 
\begin{equation}\label{polesb}
\begin{split}
&z_{m,l}=\frac{i Q}{2}+i m b +il b^{-1}, \qquad m,l = 0,1, 2,\dots, \qquad (\text{zeros}), 
	\\
&p_{m,l}=-\frac{i Q}{2} -i m b-il b^{-1}, \qquad m,l = 0,1, 2,\dots, \qquad (\text{poles}).
\end{split}
\end{equation}
The multiplicity of the zero $z_{m,l}$ in (\ref{polesb}) is given by the number of distinct pairs $(m_i,l_i) \in \mathbb{Z}_{\geq 0} \times \mathbb{Z}_{\geq 0}$ such that $z_{m_i,l_i}=z_{m,l}$. 
The pole $p_{m,l}$ has the same multiplicity as the zero $z_{m,l}$.
In particular, if $b^2$ is an irrational real number, then all the zeros and poles in (\ref{polesb}) are distinct and simple. The residue $s_b$ at the simple pole $z=-iQ/2$ is given by
\beq \label{ressb}
\underset{z=-\frac{iQ}{2}}{\text{Res}} s_b(z) =\frac{i}{2\pi}.
\eeq 
Furthermore, $s_b$ is a meromorphic solution of the following pair of difference equations:
\begin{equation}\label{differencesb}
\frac{s_b(z+\frac{ib}{2})}{s_b(z-\frac{ib}{2})}=2\operatorname{cosh}{\pi b z}, \qquad \frac{s_b(z+\frac{i}{2b})}{s_b(z-\frac{i}{2b})}=2\operatorname{cosh}{\frac{\pi z}{b}}.
\end{equation}
Applying the difference equations \eqref{differencesb} recursively, it can be verified that, for any integer $m \geq 0$,
\begin{align} \label{differencepochsb}
& \frac{s_b(x+imb^{\pm 1})}{s_b(x)} = e^{-\frac{\pi b^{\pm 1} m}{2} (2 x+i b^{\pm 1} m)} \left(-e^{i \pi  b^{\pm 2}} e^{2 \pi  b^{\pm 1} x};e^{2 i \pi  b^{\pm 2}}\right)_m,
\end{align}
where the $q$-Pochhammer symbol $(a;q)_m$ is defined in Appendix \ref{appendixA}.
Finally, the function $s_b(z)$ has the obvious symmetry
\beq \label{sbinverse}
s_b(z) = s_{b^{-1}}(z)
\eeq
and possesses an asymptotic formula which is a consequence of \cite[Theorem A.1]{R1999} and (\ref{sbgbGE}): 
For each $\epsilon >0$, 
\begin{align}\label{sbasymptotics}
\pm \ln s_b(z) = -\frac{i\pi z^2}{2} - \frac{i\pi}{24}(b^2 + b^{-2}) + O(e^{-\frac{2\pi(1-\epsilon)}{\max(b, b^{-1})}|\re z|}), \qquad \re z \to \pm \infty,
\end{align}
uniformly for $(b,\im z)$ in compact subsets of $(0,\infty) \times \mathbb{R}$. 

\section{The function $\mathcal{R}$}\label{Rsec}
The top element of the non-polynomial $q$-Askey scheme presented in this paper is Ruijsenaars' hypergeometric function $\mathcal{R}$ \cite{R1999}.
Using the notation of \cite{R2021}, this function can be expressed as
\beq\label{defR}
\mathcal{R}\left[\substack{\theta_1\;\;\theta_t\vspace{0.1cm}\\ \theta_{\infty}\;\theta_0}; \substack{\sigma_s \vspace{0.15cm} \\  \sigma_t}, b\right] = P_\mathcal{R}\left[\substack{\theta_1\;\;\theta_t\vspace{0.1cm}\\ \theta_{\infty}\;\theta_0};\substack{\sigma_s \vspace{0.15cm} \\  \sigma_t},b\right] \int_{\mathsf{R}} dx ~ I_\mathcal{R}\left[x;\substack{\theta_1\;\;\theta_t\vspace{0.1cm}\\ \theta_{\infty}\;\theta_0};\substack{\sigma_s \vspace{0.15cm} \\  \sigma_t},b\right],\eeq
where the prefactor $P_{\mathcal{R}}$ is given by
\beq \label{prefR}
P_\mathcal{R}\left[\substack{\theta_1\;\;\theta_t\vspace{0.1cm}\\ \theta_{\infty}\;\theta_0};\substack{\sigma_s \vspace{0.15cm} \\  \sigma_t},b\right] = s_b\lb \tfrac{iQ}2+2\theta_t \rb \prod_{\epsilon = \pm 1} s_b\lb \tfrac{iQ}2+\theta_0+\theta_1+\epsilon\theta_\infty+\theta_t \rb s_b(\epsilon \sigma_s-\theta_0-\theta_t)s_b(\epsilon \sigma_t-\theta_1-\theta_t),
\eeq
and the integrand $I_{\mathcal{R}}$ is defined by
\beq \label{IR}
I_\mathcal{R}\left[x;\substack{\theta_1\;\;\theta_t\vspace{0.1cm}\\ \theta_{\infty}\;\theta_0};\substack{\sigma_s \vspace{0.15cm} \\  \sigma_t},b\right] = \frac{1}{s_b(x+\tfrac{i Q}{2}) s_b(x+\tfrac{i Q}{2}+2 \theta_t)}\prod_{\epsilon=\pm1}\frac{s_b(x+\theta_0+\theta_t+\epsilon\sigma_s) s_b(x+\theta_1+\theta_t+\epsilon\sigma_t)}{s_b\left(x+\tfrac{i Q}{2}+\theta_0+\theta_1+\epsilon\theta_\infty+\theta_t\right)}.
\eeq
In view of \eqref{polesb}, the integrand $I_\mathcal{R}$ possesses eight semi-infinite sequences of poles in the complex $x$-plane.  With the restriction that $b>0$ imposed by Assumption \ref{assumption}, there are four vertical downward sequences starting at $x =- \theta_0-\theta_t\pm \sigma_s -\frac{i Q}{2}$ and $x = - \theta_1 - \theta_t \pm \sigma_t -\frac{i Q}{2}$, and four vertical upward sequences starting at $x=0$,  $x=-2\theta_t$, and $x=-\theta_0-\theta_1\pm \theta_\infty-\theta_t$. The contour $\mathsf{R}$ in (\ref{defR}) is any curve from  $-\infty$  to $+\infty$ which separates the four upward from the four downward sequences of poles. 
If in addition to Assumption \ref{assumption}, we also assume that $\sigma_s, \sigma_t \in \mathbb{R}$, then the contour of integration $\mathsf{R}$ can be chosen to be any curve from $-\infty$ to $+\infty$ lying in the open strip $\im x\in (-Q/2,0)$.

Thanks to the symmetry \eqref{sbinverse} of the function $s_b$, we have
\beq \label{Rinverse}
\mathcal{R}\left[\substack{\theta_1\;\;\theta_t\vspace{0.1cm}\\ \theta_{\infty}\;\theta_0};\substack{\sigma_s \vspace{0.15cm} \\  \sigma_t},b^{-1}\right] = \mathcal{R}\left[\substack{\theta_1\;\;\theta_t\vspace{0.1cm}\\ \theta_{\infty}\;\theta_0};\substack{\sigma_s \vspace{0.15cm} \\  \sigma_t},b\right].
\eeq  
Moreover, it follows directly from \eqref{defR} that $\mathcal{R}$ is even in each of the variables $\sigma_s$ and $\sigma_t$, and that it satisfies the following self-duality symmetry:
\beq\label{selfdual}
\mathcal{R}\left[\substack{\theta_1\;\;\theta_t\vspace{0.1cm}\\ \theta_{\infty}\;\theta_0};\substack{\sigma_s \vspace{0.15cm} \\  \sigma_t},b\right]  = \mathcal{R}\left[\substack{\theta_0\;\;\theta_t\vspace{0.1cm}\\ \theta_{\infty}\;\theta_1};\substack{\sigma_t \vspace{0.15cm} \\  \sigma_s},b\right].
\eeq

\subsection{Difference equations}
Let $e^{\pm i b \partial_x}$ be the translation operator which formally acts on meromorphic functions $f(x)$ by $e^{\pm i b \partial_x} f(x)=f(x\pm i b)$. Let $H_\mathcal{R}$ be the difference operator defined by
\beq\label{HR}
H_\mathcal{R} \left[\substack{\theta_1\;\;\theta_t\vspace{0.1cm}\\ \theta_{\infty}\;\theta_0};\substack{b,\;\sigma_s}\right] = H_\mathcal{R}^+\left[\substack{\theta_1\;\;\theta_t\vspace{0.1cm}\\ \theta_{\infty}\;\theta_0};\substack{b,\;\sigma_s}\right] e^{ib \partial_{\sigma_s}} + H_\mathcal{R}^+\left[\substack{\theta_1\;\;\theta_t\vspace{0.1cm}\\ \theta_{\infty}\;\theta_0};\substack{b,\;-\sigma_s}\right] e^{-ib \partial_{\sigma_s}} + H_\mathcal{R}^0\left[\substack{\theta_1\;\;\theta_t\vspace{0.1cm}\\ \theta_{\infty}\;\theta_0};\substack{b,\;\sigma_s}\right],
\eeq
where
\beq \label{HRplus}
H_\mathcal{R}^+ \left[\substack{\theta_1\;\;\theta_t\vspace{0.1cm}\\ \theta_{\infty}\;\theta_0};\substack{b,\;\sigma_s}\right] = - \frac{4 \prod _{\epsilon = \pm 1} \cosh \left(\pi  b \left(\frac{ib}{2} +\theta_t+\sigma_s+\epsilon\theta_0\right)\right) \cosh \left(\pi  b \left(\frac{ib}{2}+\theta_1+\sigma_s+\epsilon\theta_\infty\right)\right)}{\sinh (2 \pi  b \sigma_s) \sinh (\pi  b (2 \sigma_s+i b))}
\eeq
and
\beq\label{H0} \begin{split}
H_\mathcal{R}^0\left[\substack{\theta_1\;\;\theta_t\vspace{0.1cm}\\ \theta_{\infty}\;\theta_0};\substack{b,\;\sigma_s}\right]=  -2\cosh{( 2\pi b (\theta_1+\theta_t+\tfrac{ib}2))} - H_\mathcal{R}^+ \left[\substack{\theta_1\;\;\theta_t\vspace{0.1cm}\\ \theta_{\infty}\;\theta_0};\substack{b,\;\sigma_s}\right] - H_\mathcal{R}^+ \left[\substack{\theta_1\;\;\theta_t\vspace{0.1cm}\\ \theta_{\infty}\;\theta_0};\substack{b,\;-\sigma_s}\right].
\end{split} \eeq
For $(\sigma_s,\sigma_t) \in \mathbb C^2$, the function $\mathcal{R}$ satisfies the following four difference equations \cite{R1999} (using the notation of \cite{R2021}):
\begin{subequations}\label{diffeqR}\begin{align}
\label{diffeqR1} & H_\mathcal{R}\left[\substack{\theta_1\;\;\theta_t\vspace{0.1cm}\\ \theta_{\infty}\;\theta_0};\substack{b,\;\sigma_s}\right] \mathcal{R}\left[\substack{\theta_1\;\;\theta_t\vspace{0.1cm}\\ \theta_{\infty}\;\theta_0};\substack{\sigma_s \vspace{0.15cm} \\ \sigma_t},b\right] = 2\cosh{(2\pi b \sigma_t)} \mathcal{R}\left[\substack{\theta_1\;\;\theta_t\vspace{0.1cm}\\ \theta_{\infty}\;\theta_0};\substack{\sigma_s \vspace{0.15cm} \\  \sigma_t},b\right], \\
\label{diffeqR2} & H_\mathcal{R}\left[\substack{\theta_1\;\;\theta_t\vspace{0.1cm}\\ \theta_{\infty}\;\theta_0};\substack{b^{-1},\;\sigma_s}\right] \mathcal{R}\left[\substack{\theta_1\;\;\theta_t\vspace{0.1cm}\\ \theta_{\infty}\;\theta_0};\substack{\sigma_s \vspace{0.15cm} \\  \sigma_t},b\right] = 2\cosh{(2\pi b^{-1} \sigma_t)} \mathcal{R}\left[\substack{\theta_1\;\;\theta_t\vspace{0.1cm}\\ \theta_{\infty}\;\theta_0};\substack{\sigma_s \vspace{0.15cm} \\  \sigma_t},b\right], \\
\label{diffeqR3} & H_\mathcal{R}\left[\substack{\theta_0\;\;\theta_t\vspace{0.1cm}\\ \theta_{\infty}\;\theta_1};\substack{b,\;\sigma_t}\right] \mathcal{R}\left[\substack{\theta_1\;\;\theta_t\vspace{0.1cm}\\ \theta_{\infty}\;\theta_0};\substack{\sigma_s \vspace{0.15cm} \\  \sigma_t},b\right] = 2\cosh{(2\pi b \sigma_s)} \mathcal{R}\left[\substack{\theta_1\;\;\theta_t\vspace{0.1cm}\\ \theta_{\infty}\;\theta_0};\substack{\sigma_s \vspace{0.15cm} \\  \sigma_t},b\right], \\
\label{diffeqR4} & H_\mathcal{R}\left[\substack{\theta_0\;\;\theta_t\vspace{0.1cm}\\ \theta_{\infty}\;\theta_1};\substack{b^{-1},\;\sigma_t}\right] \mathcal{R}\left[\substack{\theta_1\;\;\theta_t\vspace{0.1cm}\\ \theta_{\infty}\;\theta_0};\substack{\sigma_s \vspace{0.15cm} \\  \sigma_t},b\right] = 2\cosh{(2\pi b^{-1} \sigma_s)} \mathcal{R}\left[\substack{\theta_1\;\;\theta_t\vspace{0.1cm}\\ \theta_{\infty}\;\theta_0};\substack{\sigma_s \vspace{0.15cm} \\  \sigma_t},b\right].
\end{align}\end{subequations}
Note that the difference equations \eqref{diffeqR2}, \eqref{diffeqR3}, and \eqref{diffeqR4} follow from \eqref{diffeqR1} and the symmetries \eqref{Rinverse}--\eqref{selfdual} of the function $\mathcal{R}$. 

\subsection{Polynomial limit}
It was shown in \cite{R1999} that the function $\mathcal{R}$ reduces to the Askey--Wilson polynomials when one of the variables $\sigma_s$ and $\sigma_t$ is suitably discretized. This result was reobtained in the CFT setting in \cite{LR2}. We now recall the result of \cite{LR2}. 
In addition to Assumption \ref{assumption}, we need the following assumption.

\begin{assumption}[Restriction on the parameters]\label{assumptionAW}
Assume that $b > 0$ is such that $b^2$ is irrational, and that, for $\epsilon,\epsilon' = \pm 1,$
\beq\label{restrictionAW}\begin{split}
& \theta_\infty, \theta_t , \re \sigma_s, \re \sigma_t \neq 0, \qquad 
\re\big(\theta_0 - \theta_1 + \epsilon \sigma_s + \epsilon' \sigma_t\big) \neq 0 , \qquad \theta_0+\theta_1+\epsilon\theta_\infty+\epsilon'\theta_t \neq 0.
\end{split}\eeq
\end{assumption}
Assumption \ref{assumptionAW} implies that the four increasing and the four decreasing sequences of poles of the integrand in \eqref{IR} do not overlap. The assumption that $b^2$ is irrational implies that all the poles of the integrand are simple.

\begin{theorem}{\cite[Theorem 4.2]{LR2}} \label{thmAW} Suppose Assumptions \ref{assumption} and \ref{assumptionAW} are satisfied. 
Define $\{\sigma_s^{(n)}\}_{n=0}^\infty \subset \mathbb{C}$ by
\beq\label{limsigmas}
  \sigma_s^{(n)} = \theta_0+\theta_t+\tfrac{iQ}2+ibn.
\eeq
Under the parameter correspondence
\beq\label{paramAW}
\alpha_R =e^{2 \pi  b \left(\frac{i Q}{2}+\theta_1+\theta_t\right)}, \quad \beta_R =e^{2 \pi  b \left(\frac{i Q}{2}+\theta_0-\theta_\infty\right)}, \quad \gamma_R=e^{2 \pi  b \left(\frac{i Q}{2}-\theta_1+\theta_t\right)}, \quad \delta_R=e^{2 \pi  b \left(\frac{i Q}{2}+\theta_0+\theta_\infty\right)}, \quad q=e^{2i\pi b^2},
\eeq
the Ruijsenaars hypergeometric function $\mathcal{R}$ defined in \eqref{defR} satisfies, for each integer $n \geq 0$,
\beq\label{limitAW}
\lim\limits_{\sigma_s \to \sigma_s^{(n)}} \mathcal{R} \left[\substack{\theta_1\;\;\;\theta_t\vspace{0.1cm}\\ \theta_{\infty}\;\;\theta_0};\substack{\sigma_s \vspace{0.15cm} \\  \sigma_t},b\right] 
= R_n(e^{2\pi b \sigma_t};\alpha_R,\beta_R,\gamma_R,\delta_R,q),
\eeq
where $R_n$ are the Askey--Wilson polynomials defined in (\ref{AW}).
\end{theorem}

\section{The function $\mathcal{H}$} \label{section4}
The Askey--Wilson polynomials $R_n$ form the top element of the $q$-Askey scheme.
Just like the elements of the $q$-Askey scheme are obtained from the polynomials $R_n$ via various limiting procedures, the elements of the non-polynomial scheme we present in this paper are obtained by taking various limits of Ruijsenaars' hypergeometric function $\mathcal{R}$. The second level of the non-polynomial scheme involves the function $\mathcal{H}(b,\theta_0,\theta_t,\theta_*,\sigma_s,\nu)$, which is defined as the confluent limit $\Lambda \to -\infty$ of the top element $\mathcal{R}$ evaluated at
\beq\label{defthetastar}
\theta_\infty = \frac{\Lambda-\theta_*}2, \qquad \theta_1 = \frac{\Lambda+\theta_*}2, \qquad \sigma_t = \frac{\Lambda}2 + \nu.
\eeq

In this section, we derive an integral representation for the function $\mathcal{H}$ and we show that it is a joint eigenfunction of four difference operators, two acting on $\sigma_s$ and the other two on $\nu$. We also show that $\mathcal{H}$ reduces to the continuous dual $q$-Hahn and the big $q$-Jacobi polynomials when $\nu$ and $\sigma_s$ are suitably discretized, respectively. Since these polynomials lie at the second level of the $q$-Askey scheme, this shows that $\mathcal{H}$ indeed provides a natural non-polynomial generalization of the elements at the second level. 

\subsection{Definition and integral representation}
\begin{definition}
The function $\mathcal{H}$ is defined by
\beq\label{fromRtoH}
\mathcal{H}(b,\theta_0,\theta_t,\theta_*,\sigma_s,\nu) = \lim\limits_{\Lambda\to -\infty} \mathcal{R}\left[\substack{\frac{\Lambda+\theta_*}2\;\; \theta_t\vspace{0.1cm}\\ \frac{\Lambda-\theta_*}2\;\; \theta_0};\substack{\sigma_s \vspace{0.15cm} \\  \frac{\Lambda}2+\nu},b\right].
\eeq
\end{definition}

The next theorem shows that, for each choice of $(b,\theta_0,\theta_t,\theta_*) \in (0,\infty) \times \mathbb{R}^3$, $\mathcal{H}$ is a well-defined and analytic function of 
\begin{align}\label{sigmasnudomain}
(\sigma_s,\nu) \in (\mathbb{C} \setminus \Delta_{\mathcal{H}, \sigma_s}) \times \big(\{\im \nu > -Q/2\} \setminus \Delta_{\nu}\big),
\end{align}
where $\Delta_{\mathcal{H}, \sigma_s}, \Delta_{\nu} \subset \mathbb{C}$ are discrete sets of points at which $\mathcal{H}$ may have poles. In particular, $\mathcal{H}$ is a meromorphic function of $\sigma_s \in \mathbb{C}$ and of $\nu$ for $\im \nu > -Q/2$. 
More precisely, $\Delta_{\mathcal{H}, \sigma_s}$ and $\Delta_{\nu}$ are given by
\begin{subequations}\label{Dsigmasnudef}
\begin{align}\label{Dsigmasdef}
\Delta_{\mathcal{H}, \sigma_s} := &\; \{\pm \sigma_s \, |\, \sigma_s \in \Delta_{\mathcal{H}, \sigma_s}'\},
	\\ \nonumber
\Delta_{\nu} := &\; \{\tfrac{\theta_*}{2} \pm \theta_t + \tfrac{iQ}{2} + i m b +il b^{-1}\}_{m,l=0}^\infty \cup \{-\tfrac{\theta_*}{2} - \theta_0 + \tfrac{iQ}{2} + i m b +il b^{-1}\}_{m,l=0}^\infty
	\\\label{Dnudef}
& \cup \{\tfrac{\theta_*}{2}+ \theta_t -\tfrac{i Q}{2} -i m b-il b^{-1}\}_{m,l=0}^\infty
\end{align}
\end{subequations}
where
\begin{align*}
\Delta_{\mathcal{H}, \sigma_s}' := &\; \{\theta_0 \pm \theta_t + \tfrac{iQ}{2} + i m b +il b^{-1}\}_{m,l=0}^\infty
\cup \{-\theta_* + \tfrac{iQ}{2} + i m b +il b^{-1}\}_{m,l=0}^\infty
	\\
& 
\cup \{\theta_0+ \theta_t -\tfrac{i Q}{2} -i m b-il b^{-1}\}_{m,l=0}^\infty.
\end{align*}
The theorem also provides an integral representation for $\mathcal{H}$ for $(\sigma_s,\nu)$ satisfying \eqref{sigmasnudomain}.
In fact, even if the requirement $\im \nu > -Q/2$ is needed to ensure convergence of the integral in the integral representation for $\mathcal{H}$, we will show later in this section, with the help of the difference equations satisfied by $\mathcal{H}$, that $\mathcal{H}$ extends to a meromorphic function of $(\sigma_s, \nu) \in \mathbb{C}^2$.

\begin{theorem} \label{thmforH}
Suppose that Assumption \ref{assumption} holds. Let $\Delta_{\mathcal H,\sigma_s}, \Delta_{\nu} \subset \mathbb{C}$ be the discrete subsets defined in \eqref{Dsigmasnudef}. Then the limit in \eqref{fromRtoH} exists uniformly for $(\sigma_s,\nu)$ in compact subsets of 
$$D_{\mathcal{H}} := (\mathbb{C} \setminus \Delta_{\mathcal H,\sigma_s}) \times \big(\{\im \nu > -Q/2\} \setminus \Delta_{\nu}\big).$$
Moreover, $\mathcal{H}$ is an analytic function of $(\sigma_s,\nu) \in D_{\mathcal{H}}$ and admits the following integral representation:
\beq\label{defH}
\mathcal{H}(b,\theta_0,\theta_t,\theta_*,\sigma_s,\nu) = P_\mathcal{H}\lb \sigma_s,\nu\rb \displaystyle \int_{\mathsf{H}} dx ~ I_\mathcal{H}\lb x, \sigma_s,\nu\rb \qquad \text{for $(\sigma_s,\nu) \in D_{\mathcal{H}}$},
\eeq
where the dependence of $P_\mathcal{H}$ and $I_\mathcal{H}$ on $b,\theta_0,\theta_t,\theta_*$ is omitted for simplicity, 
\begin{align}
\label{PH} &
P_\mathcal{H}\lb \sigma_s,\nu\rb = s_b\left(2 \theta_t+\tfrac{i Q}{2}\right) s_b\left(\theta_0+\theta_*+\theta_t+\tfrac{i Q}{2}\right) s_b\left(\nu-\tfrac{\theta_*}{2}-\theta_t \right) \prod _{\epsilon = \pm 1}  s_b\left(\epsilon\sigma_s-\theta_0-\theta_t\right), \\
&  \label{IH}
I_\mathcal{H}\lb x,\sigma_s,\nu\rb = e^{i \pi  x \left(\frac{\theta_*}{2}-\theta_0+\nu-\frac{i Q}{2}\right)} \frac{s_b\left(x+\frac{\theta_*}{2}+\theta_t-\nu\right) \prod_{\epsilon=\pm 1} s_b(x+\theta_0+\theta_t+\epsilon\sigma_s)}{s_b\left(x+\frac{i Q}{2}\right) s_b\left(x+2\theta_t+\frac{i Q}{2}\right) s_b\left(x+\theta_0+\theta_*+\theta_t+\frac{i Q}{2}\right)},
\end{align}
and the contour $\mathsf{H}$ is any curve from  $-\infty$  to $+\infty$ which separates the three upward from the three downward sequences of poles.
In particular, $\mathcal{H}$ is a meromorphic function of $(\sigma_s, \nu) \in \mathbb{C}\times \{\im \nu > -Q/2\}$. If $(\sigma_s,\nu) \in \mathbb{R}^2$, then the contour $\mathsf{H}$ can be any curve from $-\infty$ to $+\infty$ lying within the strip $\im{x} \in (-Q/2,0)$. 
\end{theorem}
\begin{proof}
Let $(b,\theta_0,\theta_t,\theta_*) \in (0,\infty) \times \mathbb{R}^3$.
Using the identity $s_b(z)= 1/s_b(-z)$, it is straightforward to verify that
\begin{align} \label{PRIR}
& P_\mathcal{R}\left[\substack{\frac{\Lambda+\theta_*}2\;\; \theta_t\vspace{0.1cm}\\ \frac{\Lambda-\theta_*}2\;\; \theta_0};\substack{\sigma_s \vspace{0.15cm} \\  \frac{\Lambda}2+\nu}\;,b\right] I_\mathcal{R}\left[x; \substack{\frac{\Lambda+\theta_*}2\;\; \theta_t\vspace{0.1cm}\\ \frac{\Lambda-\theta_*}2\;\; \theta_0};\substack{\sigma_s \vspace{0.15cm} \\  \frac{\Lambda}2+\nu} \;,b\right] = \;  P_\mathcal{H}\lb \sigma_s,\nu\rb X(x,\Lambda) I_\mathcal{H}\lb x, \sigma_s,\nu\rb,
\end{align}
where the dependence of $X(x,\Lambda)$ on $b,\theta_0,\theta_t,\theta_*,\sigma_s,\nu$ is omitted for simplicity, and
\beq \label{XforRtoH}
X(x,\Lambda) = e^{i\pi x (\frac{iQ}2+\theta_0-\frac{\theta_*}2-\nu)} \frac{s_b\left(\Lambda+\theta_0+\theta_t +\frac{i Q}{2}\right)}{s_b\left(\Lambda+\frac{\theta_*}{2}+\theta_t+\nu\right)} \frac{s_b\left(x+\Lambda+\frac{\theta_*}{2}+\theta_t+\nu\right)}{s_b\left(x+\Lambda+\theta_0+\theta_t +\frac{i Q}{2}\right)}.
\eeq

Due to the properties \eqref{polesb} of the function $s_b$, the function $I_\mathcal{H}(\cdot, \sigma_s, \nu)$ possesses three increasing sequences of poles starting at $x=0, x=-2\theta_t$ and $x =-\theta_0-\theta_*-\theta_t$, as well as three decreasing sequences of poles starting at $x= -\tfrac{iQ}2-\tfrac{\theta_*}{2}-\theta_t+\nu$ and $x=-\tfrac{iQ}2 \pm \sigma_s-\theta_0-\theta_t$. 
The discrete sets $\Delta_{\mathcal H,\sigma_s}$ and $\Delta_{\nu}$ contain all the values of $\sigma_s$ and $\nu$, respectively, for which poles in any of the three increasing collide with poles in any of the decreasing sequences. Indeed, consider for example the decreasing sequence starting at $x= -\tfrac{iQ}2-\tfrac{\theta_*}{2}-\theta_t+\nu$ and the increasing sequence starting at $x=0$. Poles from these two sequences collide if and only if
$$\nu \in \{\tfrac{\theta_*}{2} + \theta_t +\tfrac{iQ}{2} +  i m b +il b^{-1}\}_{m, l = 0}^\infty,$$
giving rise to the first set on the right-hand side of \eqref{Dnudef}.

Similarly, \eqref{polesb} implies that $X(\cdot, \Lambda)$ possesses one increasing sequence of poles starting at $x = -\Lambda - \theta_0 - \theta_t$ and one decreasing sequence of poles starting at $x = -\tfrac{iQ}2 -\Lambda - \tfrac{\theta_*}{2} - \theta_t - \nu$. The real parts of the poles in these two sequences tend to $+\infty$ as $\Lambda \to -\infty$. The increasing sequence lies in the half-plane $\im x \geq 0$ and the decreasing sequence lies in the half-plane $\im x \leq -\im \nu - Q/2$.

The sets $\Delta_{\mathcal H,\sigma_s}$ and $\Delta_{\nu}$ also contain all the values of $\sigma_s$ and $\nu$ at which the prefactor $P_\mathcal{H}(\sigma_s,\nu)$ has poles. For example, $P_\mathcal{H}$ has poles originating from the factor $s_b\left(\nu-\tfrac{\theta_*}{2}-\theta_t \right) $ at 
$$\nu = \tfrac{\theta_*}{2}+ \theta_t -\frac{i Q}{2} -i m b-il b^{-1}, \qquad m, l = 0,1, 2, \dots,$$
giving rise to the last set on the right-hand side of \eqref{Dnudef}.

Let $K_{\sigma_s}$ be a compact subset of $\mathbb{C} \setminus \Delta_{\mathcal H,\sigma_s}$ and let $K_{\nu}$ be a compact subset of $\{\im \nu > -Q/2\} \setminus \Delta_{\nu}$. Suppose $(\sigma_s, \nu) \in K_{\sigma_s} \times K_{\nu}$. Then, the above discussion shows that it is possible to choose a contour $\mathsf{H} = \mathsf{H}(\sigma_s, \nu)$ from  $-\infty$  to $+\infty$ which separates the four upward from the four downward sequences of poles of $X(\cdot,\Lambda)  I_\mathcal{H}\lb \cdot, \sigma_s,\nu\rb$. It also follows that if we let the right tail of $\mathsf{H}$ approach the horizontal line $\im x = -\epsilon$ as $\re x \to +\infty$, where $\epsilon > 0$ is sufficiently small, then there exists a $N < 0$ such that $\mathsf{H}$ can be chosen to be independent of $\Lambda$ for $\Lambda < -N$.
Thus, for such a choice of $\mathsf{H}$, \eqref{defR} and \eqref{PRIR} imply that, for all $(\sigma_s, \nu) \in K_{\sigma_s} \times K_{\nu}$ and all $\Lambda < -N$,
\begin{align}\label{RPXI}
\mathcal{R}\left[\substack{\frac{\Lambda+\theta_*}2\;\; \theta_t\vspace{0.1cm}\\ \frac{\Lambda-\theta_*}2\;\; \theta_0};\substack{\sigma_s \vspace{0.15cm} \\  \frac{\Lambda}2+\nu},b\right]
= P_\mathcal{H}\lb \sigma_s,\nu\rb \int_{\mathsf{H}} dx X(x,\Lambda)  I_\mathcal{H}\lb x, \sigma_s,\nu\rb.
\end{align}
If $(\sigma_s,\nu) \in \mathbb{R}^2$, then $\mathsf{H}$ can be any curve from $-\infty$ to $+\infty$ lying within the strip $\im{x} \in (-Q/2,0)$.

Utilizing the asymptotic formula \eqref{sbasymptotics} for the function $s_b$ with $\epsilon = 1/2$, we find that
\beq\label{lnX}
\operatorname{ln}{\lb X(x, \Lambda)\rb} = O\Big( e^{-\frac{\pi \lvert \Lambda \lvert}{\max(b,b^{-1})} } \Big), \qquad \Lambda \to -\infty,
\eeq
uniformly for $(\sigma_s, \nu) \in K_{\sigma_s} \times K_{\nu}$ and for $x$ in bounded subsets of $\mathsf{H}$. We deduce that
\beq\label{Xpointwiselimit}
\lim\limits_{\Lambda \to -\infty} X(x,\Lambda) = 1,
\eeq
uniformly for $(\sigma_s, \nu) \in K_{\sigma_s} \times K_{\nu}$ and $x$ in bounded subsets of $\mathsf{H}$.

Using the asymptotic formula \eqref{sbasymptotics} for $s_b$ with $\epsilon=1/2$,  we find that $I_\mathcal{H}$ obeys the estimate 
 \begin{align} \label{decayIH}
I_\mathcal{H}\lb x,\sigma_s,\nu\rb = \begin{cases}
O\lb e^{-2\pi(\frac{Q}{2} + \im \nu) |\re x|} \rb, & \re x \to + \infty,\\
O\lb e^{-2\pi Q |\re x|} \rb, & \re x \to - \infty, 
\end{cases}
\end{align}
uniformly for $(\sigma_s,\nu) \in K_{\sigma_s} \times K_{\nu}$ and $\im x$ in compact subsets of $\mathbb{R}$. 
Since the contour $\mathsf{H}$ stays a bounded distance away from the increasing and the decreasing pole sequences, we infer that there exists a constant $C_1>0$ such that
\beq \label{inequalityIH}
\lvert I_\mathcal{H}\lb x,\sigma_s,\nu\rb \lvert \leq \begin{cases}
C_1e^{-2\pi(\frac{Q}{2} + \im \nu) |\re x|}, & x \in \mathsf{H}, ~ \re x \geq 0, \\
C_1e^{-2\pi Q |\re x|}, & x \in \mathsf{H}, ~ \re x \leq 0, 
\end{cases}
\eeq
uniformly for $(\sigma_s,\nu) \in K_{\sigma_s} \times K_{\nu}$. In particular, since $K_{\nu} \subset \{\im \nu > -Q/2\}$, $I_\mathcal{H}$ has exponential decay on the left and right tails of the contour $\mathsf{H}$. 

Suppose we can show that there exist constants $c>0$ and $C>0$ such that
\beq \label{inequalityPRIR}
|X(x,\Lambda)  I_\mathcal{H}\lb x, \sigma_s,\nu\rb| \leq C e^{-c \lvert \re x \lvert}
\eeq
uniformly for all $\Lambda < -N$, $x \in \mathsf{H}$, and $(\sigma_s, \nu) \in K_{\sigma_s} \times K_{\nu}$. Then it follows from \eqref{RPXI}, \eqref{Xpointwiselimit}, and Lebesgue's dominated convergence theorem that the limit in \eqref{fromRtoH} exists uniformly for $(\sigma_s, \nu) \in K_{\sigma_s} \times K_{\nu}$ and is given by \eqref{defH}. Since $K_{\sigma_s} \subset \mathbb{C} \setminus \Delta_{\mathcal H,\sigma_s}$ and $K_{\nu} \subset \mathbb{C} \setminus \Delta_{\nu}$ are arbitrary compact subsets, this proves that the limit in \eqref{fromRtoH} exists uniformly for $(\sigma_s,\nu)$ in compact subsets of $D_{\mathcal{H}}$ and proves \eqref{defH}. Moreover, the analyticity of $\mathcal{H}$ as a function of $(\sigma_s,\nu) \in D_{\mathcal{H}}$ follows from the analyticity of $\mathcal{R}$ together with the uniform convergence on compact subsets. Alternatively, the analyticity of $\mathcal{H}$ in $D_{\mathcal{H}}$ can be inferred directly from the representation \eqref{defH}. Indeed, the possible poles of $\mathcal{H}$ lie at such values of $(\sigma_s,\nu)$ at which either the prefactor $P_{\mathcal{H}}$ has a pole or at which the contour of integration gets pinched between two poles of the integrand $I_\mathcal{H}$, and the definitions of $\Delta_{\mathcal H,\sigma_s}$ and $\Delta_\nu$ exclude both of these situations. Thus to complete the proof of the theorem, it only remains to prove \eqref{inequalityPRIR}.

To prove \eqref{inequalityPRIR}, we need to estimate the function $X$ defined in \eqref{XforRtoH}.
The asymptotic formula (\ref{sbasymptotics}) for $s_b$ with $\epsilon=1/2$ implies that there exist constants $C_2, C_3, C_4>0$ such that the inequalities 
\begin{align}
& \bigg|\frac{s_b\left(\Lambda+\theta_0+\theta_t +\frac{i Q}{2}\right)}{s_b\left(\Lambda+\frac{\theta_*}{2}+\theta_t+\nu\right)}\bigg| \leq C_2 e^{\pi (\frac{Q }{2} - \im \nu) |\Lambda|}, \qquad \Lambda < -N,  \\
& \bigg|\frac{s_b\left(x+\Lambda+\frac{\theta_*}{2}+\theta_t+\nu\right)}{s_b\left(x+\Lambda+\theta_0+\theta_t +\frac{i Q}{2}\right)} 
\bigg| \leq C_3 e^{-\pi(\frac{Q}{2} - \im \nu)|\Lambda + \re{x}|}, \qquad x \in \mathsf{H}, ~ \Lambda \in \mathbb{R}, \\
& |e^{i\pi x (\frac{iQ}2+\theta_0-\frac{\theta_*}2-\nu)}| \leq C_4 e^{-\pi(\frac{Q}{2}-\im \nu)\re{x}}, \qquad x \in \mathsf{H},
\end{align}
hold uniformly for $(\sigma_s,\nu) \in K_{\sigma_s} \times K_{\nu}$.
Combining the above estimates, we infer that there exists a constant $C_5$ such that 
\beq \label{inequalityX}
|X(x,\Lambda)| \leq C_5 e^{\pi(\frac{Q}{2}-\im \nu)(|\Lambda| - |\Lambda + \re{x}| - \re{x})}, \qquad x \in \mathsf{H}, ~ \Lambda < -N,
\eeq
uniformly for $(\sigma_s,\nu) \in K_{\sigma_s} \times K_{\nu}$. The inequality \eqref{inequalityX} can be rewritten as follows:
\beq\label{inequalityXC5}
|X(x,\Lambda)| \leq \begin{cases} C_5 e^{-2\pi (\frac{Q}{2} - \im \nu)(\Lambda + \re{x})}, & \Lambda + \re{x} \geq 0, \\
C_5, &  \Lambda + \re{x} \leq 0, \end{cases} \quad x \in \mathsf{H}, ~ \Lambda<-N,
\eeq
uniformly for $(\sigma_s,\nu) \in K_{\sigma_s} \times K_{\nu}$. 

If $\nu \in K_{\nu}$ is such that $-Q/2 < \im \nu \leq Q/2$, then \eqref{inequalityXC5} implies that $|X|$ is uniformly bounded for all $x \in \mathsf{H}$, $\Lambda<-N$, and $(\sigma_s,\nu) \in K_{\sigma_s} \times K_{\nu}$; hence \eqref{inequalityPRIR} follows from \eqref{inequalityIH} in this case.
On the other hand, if $\nu \in K_{\nu}$ is such that $\im \nu \geq Q/2$, then \eqref{inequalityIH} and \eqref{inequalityXC5} yield the existence of a constant $C_6> 0$ independent of $x \in \mathsf{H}$, $\Lambda < -N$, and $(\sigma_s,\nu) \in K_{\sigma_s} \times K_{\nu}$ such that
\beq\label{inequalityXC6}
|X(x,\Lambda)  I_\mathcal{H}\lb x, \sigma_s,\nu\rb|
\leq  \begin{cases} 
C_6e^{2\pi (\frac{Q}{2} - \im \nu)|\Lambda|}e^{-2\pi Q|\re{x}|}
\leq C_6e^{-2\pi Q|\re{x}|}, & \re{x} \geq -\Lambda, 
	\\
C_6e^{-2\pi(\frac{Q}{2} + \im \nu) |\re x|} \leq C_6 e^{-2\pi Q|\re x|}, & 0 \leq \re x \leq -\Lambda,
	\\
C_6e^{-2\pi Q |\re x|}, & \re x \leq 0, 
\end{cases}
\eeq
which shows \eqref{inequalityPRIR} also in this case. This completes the proof.
\end{proof}

Furthermore, thanks to the symmetry \eqref{sbinverse} of $s_b$, we have
\beq \label{Hbbinverse}
\mathcal{H}(b^{-1},\theta_0,\theta_t,\theta_*,\sigma_s,\nu) = \mathcal{H}(b,\theta_0,\theta_t,\theta_*,\sigma_s,\nu).
\eeq

\subsection{Difference equations}\label{Hdifferencesubsec}

We now show that the four difference equations \eqref{diffeqR} satisfied by the function $\mathcal{R}$ survive in the confluent limit \eqref{fromRtoH}. This implies that the function $\mathcal{H}$ is a joint eigenfunction of four difference operators, two acting on $\sigma_s$ and the remaining two on $\nu$. 

We know from Theorem \ref{thmforH} that $\mathcal{H}$ is a well-defined meromorphic function of $(\sigma_s, \nu) \in \mathbb{C} \times \{\im \nu > -Q/2\}$. The difference equations will first be derived as equalities between meromorphic functions defined on this limited domain. However, the difference equations in $\nu$ can then be used to show that: (i) the limit in \eqref{fromRtoH} exists for all $\nu$ in the whole complex plane away from a discrete subset, (ii) $\mathcal{H}$ is in fact a meromorphic function of $(\sigma_s, \nu)$ in all of $\mathbb{C}^2$, and (iii) the four difference equations hold as equalities between meromorphic functions on $\mathbb{C}^2$, see Proposition \ref{Hextensionprop}.

\subsubsection{First pair of difference equations}
Define the difference operator $H_\mathcal{H}(b,\sigma_s)$ by
\beq\label{HH}\begin{split}
H_\mathcal{H}(b,\sigma_s) & = H^+_\mathcal{H}(b,\sigma_s) e^{ib\partial_{\sigma_s}} + H^+_\mathcal{H}(b,-\sigma_s) e^{-ib\partial_{\sigma_s}} + H_\mathcal{H}^{0}(b,\sigma_s),
 \end{split}\eeq
where $H^0_{\mathcal{H}}$ is defined by 
 \beq \label{HH0} \begin{split}
H^0_{\mathcal{H}}(b,\sigma_s) &=e^{- \pi b(iQ+\theta_*+2\theta_t)}-H^+_\mathcal{H}(b,\sigma_s)-H^+_\mathcal{H}(b,-\sigma_s),
\end{split} \eeq
with
\beq\label{HHplus}
H^+_\mathcal{H}(b,\sigma_s) = -2 e^{-\pi b\left(\sigma_s+\frac{i b}{2}\right)}\cosh(\pi  b (\tfrac{i b}{2}+\theta_*+\sigma_s))\frac{\prod _{\epsilon =\pm1} \cosh \left(\pi  b \left(\frac{ib}{2}+\theta_t+\sigma_s+\epsilon\theta_0 \right)\right)}{\sinh (\pi  b (2 \sigma_s+i b)) \sinh (2 \pi  b \sigma_s)}
\eeq

\begin{proposition}\label{diffeqHprop}
For $\sigma_s \in \mathbb{C}$ and $\im \nu > -Q/2$, the function $\mathcal{H}$ defined by \eqref{fromRtoH} satisfies the following pair of difference equations:
\begin{subequations} \label{diffeqH}\begin{align}
\label{diffeqH1} & H_\mathcal{H}(b,\sigma_s) ~ \mathcal{H}(b,\theta_0,\theta_t,\theta_*,\sigma_s,\nu) = e^{-2\pi b \nu} \mathcal{H}(b,\theta_0,\theta_t,\theta_*,\sigma_s,\nu), \\
\label{diffeqH2} & H_\mathcal{H}(b^{-1},\sigma_s) ~ \mathcal{H}(b,\theta_0,\theta_t,\theta_*,\sigma_s,\nu) = e^{-2\pi b^{-1} \nu} \mathcal{H}(b,\theta_0,\theta_t,\theta_*,\sigma_s,\nu).
\end{align} \end{subequations}
\end{proposition}
\begin{proof}
The proof consists of taking the confluent limit \eqref{fromRtoH} of the difference equation \eqref{diffeqR1}. On the one hand, we have
\beq\label{limitcosh}
\lim\limits_{\Lambda \to -\infty} \Big(e^{\pi b \Lambda}2\cosh{2\pi b \sigma_t}\lvert_{\sigma_t = \tfrac{\Lambda}2+\nu}\Big) = e^{-2\pi b \nu}.
\eeq
On the other hand, straightforward computations using asymptotics of hyperbolic functions show that the following limits hold:
\beq
\lim\limits_{\Lambda \to -\infty} e^{\pi b \Lambda} H_\mathcal{R}^+\left[\substack{\tfrac{\Lambda+\theta_*}{2}\;\;\theta_t\vspace{0.1cm}\\ \tfrac{\Lambda-\theta_*}{2} \;\theta_0};\substack{b,\;\pm\sigma_s}\right] = H_\mathcal{H}^+(b,\pm\sigma_s), \qquad \lim\limits_{\Lambda \to -\infty} e^{\pi b \Lambda} H_\mathcal{R}^0\left[\substack{\tfrac{\Lambda+\theta_*}{2}\;\;\theta_t\vspace{0.1cm}\\ \tfrac{\Lambda-\theta_*}{2} \;\theta_0};\substack{b,\;\sigma_s}\right] = H^0_\mathcal{H}(b,\sigma_s),
\eeq
where $H_\mathcal{R}^+$ and $H_\mathcal{R}^0$ are defined in \eqref{HRplus} and \eqref{H0}, respectively. Therefore we obtain
\beq \label{HRtoHHsigmas}
\lim\limits_{\Lambda \to -\infty} e^{\pi b \Lambda}  H_\mathcal{R}\left[\substack{\tfrac{\Lambda+\theta_*}{2}\;\;\theta_t\vspace{0.1cm}\\ \tfrac{\Lambda-\theta_*}{2} \;\theta_0};\substack{b,\;\sigma_s}\right] = H_\mathcal{H}(b,\sigma_s),
\eeq
where $H_\mathcal{R}$ is given in \eqref{HR}. By Theorem \ref{thmforH}, the limit in \eqref{fromRtoH} exists whenever $(\sigma_s, \nu)\in D_{\mathcal{H}}$. Thus, the difference equation \eqref{diffeqH1} follows after multiplying \eqref{diffeqR1} by $e^{\pi b \Lambda}$ and utilizing \eqref{limitcosh}, \eqref{HRtoHHsigmas}, and the definition \eqref{fromRtoH} of $\mathcal{H}$. Finally,  \eqref{diffeqH2} follows from \eqref{diffeqH1} and the symmetry \eqref{Hbbinverse} of $\mathcal{H}$.
\end{proof}

\subsubsection{Second pair of difference equations}

Define the dual difference operator $\tilde{H}_\mathcal{H}(b,\nu)$ by
\begin{align} \label{HHtilde}
\tilde{H}_\mathcal{H}(b,\nu) = \tilde{H}^+_\mathcal{H}(b,\nu) e^{ib\partial_\nu} + \tilde{H}^-_\mathcal{H}(b,\nu) e^{-ib\partial_\nu} + \tilde{H}_\mathcal{H}^0(b,\nu),
\end{align}
where $\tilde{H}_{\mathcal{H}}^0$ is defined by
\beq \label{HHtilde0}
\tilde{H}_{\mathcal{H}}^0(b,\nu) = -2\operatorname{cosh}{(2\pi b \lb \tfrac{ib}2+\theta_0+\theta_t \rb)}-\tilde{H}^+_\mathcal{H}(b,\nu)-\tilde{H}^{-}_\mathcal{H}(b,\nu),
\eeq
with
\beq\label{HHtildeplus}
\tilde{H}^\pm_\mathcal{H}(b,\nu) = -4e^{2\pi b \nu}e^{\mp \pi b(\theta_0+\theta_t)} \operatorname{cosh}{\lb \pi b \lb \tfrac{ib}2+\theta_0 \pm (\nu + \tfrac{\theta_*}{2}) \rb \rb} \operatorname{cosh}{\lb \pi b \lb \tfrac{ib}2+\theta_t \pm (\nu - \tfrac{\theta_*}{2}) \rb \rb}.
\eeq

\begin{proposition}\label{diffeqtildeHprop}
For $\sigma_s \in \mathbb{C}$ and $\im (\nu - ib^{\pm 1}) > -Q/2$, the function $\mathcal{H}$ satisfies the following pair of difference equations:
\begin{subequations} \label{diffeqtildeH}\begin{align}
\label{diffeqtildeH1} & \tilde{H}_\mathcal{H}(b,\nu) ~ \mathcal{H}(b,\theta_0,\theta_t,\theta_*,\sigma_s,\nu) = 2\cosh{(2\pi b \sigma_s)} ~ \mathcal{H}(b,\theta_0,\theta_t,\theta_*,\sigma_s,\nu), \\
\label{diffeqtildeH2} & \tilde{H}_\mathcal{H}(b^{-1},\nu) ~ \mathcal{H}(b,\theta_0,\theta_t,\theta_*,\sigma_s,\nu) = 2\cosh{(2\pi b^{-1} \sigma_s)} ~ \mathcal{H}(b,\theta_0,\theta_t,\theta_*,\sigma_s,\nu).
\end{align} \end{subequations}
\end{proposition}
\begin{proof}
It is straightforward to verify that the following limits hold:
\beq
\lim\limits_{\Lambda \to -\infty} H_\mathcal{R}^+\left[\substack{\theta_0\;\;\;\;\;\;\theta_t\vspace{0.1cm}\\ \frac{\Lambda-\theta_*}2\;\frac{\Lambda+\theta_*}2};\substack{b,\;\pm(\frac{\Lambda}2+\nu)}\right] = \tilde{H}^\pm_\mathcal{H}(b,\nu), \quad \lim\limits_{\Lambda \to -\infty} H_\mathcal{R}^0\left[\substack{\theta_0\;\;\;\;\;\;\theta_t\vspace{0.1cm}\\ \frac{\Lambda-\theta_*}2\;\frac{\Lambda+\theta_*}2};\substack{b,\;\frac{\Lambda}2+\nu}\right] = \tilde{H}^0_\mathcal{H}(b,\nu),
\eeq
where $H_\mathcal{R}^+$ and $H_\mathcal{R}^0$ are defined in \eqref{HRplus} and \eqref{H0}, respectively. We obtain 
\beq \label{HRtoHHtilde}
\lim\limits_{\Lambda \to -\infty} H_\mathcal{R}\left[\substack{\theta_0\;\;\;\;\;\;\theta_t\vspace{0.1cm}\\ \frac{\Lambda-\theta_*}2\;\frac{\Lambda+\theta_*}2};\substack{b,\;\frac{\Lambda}2+\nu}\right] = \tilde{H}_\mathcal{H}(b,\nu),
\eeq
where $H_\mathcal{R}$ is defined in \eqref{HR}.  The difference equation \eqref{diffeqtildeH1} follows from \eqref{diffeqR3}, \eqref{fromRtoH}, \eqref{HRtoHHtilde}, and Theorem \ref{thmforH}. 
Finally, \eqref{diffeqtildeH2} follows from \eqref{diffeqtildeH1} and the symmetry \eqref{Hbbinverse} of $\mathcal{H}$.
\end{proof}


Using the difference equations \eqref{diffeqtildeH}, we can show that $\mathcal{H}$ extends to a meromorphic function of $\nu$ everywhere in the complex plane. More precisely, we have the following proposition.

\begin{proposition}\label{Hextensionprop}
Let $(b,\theta_0,\theta_t,\theta_*) \in (0,\infty) \times \mathbb{R}^3$ and $\sigma_s \in \mathbb{C} \setminus \Delta_{\mathcal H,\sigma_s}$. Then there is a discrete subset $\Delta \subset \mathbb{C}$ such that the limit in \eqref{fromRtoH} exists for all $\nu \in \mathbb{C}\setminus \Delta$. Moreover, the function $\mathcal{H}$ defined by \eqref{fromRtoH} is a meromorphic function of $(\sigma_s, \nu) \in \mathbb{C}^2$ and the four difference equations \eqref{diffeqH} and \eqref{diffeqtildeH} hold as equalities between meromorphic functions of $(\sigma_s, \nu) \in \mathbb{C}^2$.
\end{proposition}
\begin{proof}
Consider $\nu$ such that $\im \nu > -Q/2$ but $\im (\nu-ib) \leq -Q/2$. Solving \eqref{diffeqR3} for 
$$e^{-ib\partial_\nu}\bigg(\mathcal{R}\left[\substack{\frac{\Lambda+\theta_*}2\;\; \theta_t\vspace{0.1cm}\\ \frac{\Lambda-\theta_*}2\;\; \theta_0};\substack{\sigma_s \vspace{0.15cm} \\  \frac{\Lambda}2+\nu},b\right]\bigg),$$
 taking the confluent limit $\Lambda \to -\infty$, and using \eqref{fromRtoH}, \eqref{HRtoHHtilde}, and Theorem \ref{thmforH}, we conclude that the limit in \eqref{fromRtoH} exists also for $\nu$ in the strip $\{\nu \in \mathbb{C} | -b -Q/2 < \im \nu \leq -Q/2\} \setminus \Delta_1$, where $\Delta_1$ is a discrete set.  By iteration, we conclude that the limit in \eqref{fromRtoH} exists for all $\nu \in \mathbb{C} \setminus \Delta$, where $\Delta$ is a discrete set. This proves the first assertion. The remaining assertions now follow by repeating the proofs of Propositions \ref{diffeqHprop} and \ref{diffeqtildeHprop} with $\nu \in \mathbb{C}\setminus \Delta$.
\end{proof}

\subsection{First polynomial limit}

In this subsection, we show that the function $\mathcal{H}$ reduces to the continuous dual $q$-Hahn polynomials when $\nu$ is suitably discretized. 
In addition to Assumption \ref{assumption}, we make the following assumption.

\begin{assumption}[Restriction on the parameters]\label{assumptionhahnjacobi} 
Assume that $b > 0$ is such that $b^2$ is irrational, and that
\beq\label{restrictionhahnjacobi}\begin{split}
& \theta_t, \re \sigma_s \neq 0, \qquad \re\big(\tfrac{\theta_*}{2}- \nu -\theta_0 \pm  \sigma_s \big) \neq 0, \qquad \theta_0 + \theta_* \pm \theta_t \neq 0.
\end{split}\eeq
\end{assumption}
Assumption \ref{assumptionhahnjacobi} implies that the three increasing and the three decreasing sequences of poles of the integrand in \eqref{defH} do not overlap. 
The assumption that $b^2$ is irrational ensures that all the poles of the integrand are simple and that $q=e^{2i\pi b^2}$ is not a root of unity.

Define $\{\nu_n\}_{n=0}^\infty \subset \mathbb{C}$ by
\beq\label{nun}
\nu_n = \tfrac{\theta_*}{2}+\theta_t+\tfrac{i Q}2+inb.
\eeq
The sequence $\{\nu_n\}_{n=0}^\infty$ is a subset of the set $\Delta_\nu$ of possible poles of $\mathcal{H}$ defined in \eqref{Dnudef}. The following theorem shows that $\mathcal{H}$ still has a finite limit as $\nu \to \nu_n$ for each $n \geq 0$ and that the limit is given by the continuous dual $q$-Hahn polynomials. The reason the limit is finite is that the prefactor $P_\mathcal{H}$ has a simple zero at each $\nu_n$; this zero cancels the simple pole that the integral in \eqref{defH} has due to the contour of integration being pinched between two poles of the integrand. 

\begin{theorem}[From $\mathcal{H}$ to the continuous dual $q$-Hahn polynomials] \label{thhahn} 
Let $\sigma_s \in \mathbb{C} \setminus \Delta_{\mathcal H,\sigma_s}$ and suppose that Assumptions \ref{assumption} and \ref{assumptionhahnjacobi} are satisfied.  Under the parameter correspondence 
\beq\label{paramhahn}
\alpha_H=e^{2\pi b(\theta_t+\theta_0+\frac{iQ}2)}, \quad \beta_{H}=e^{2\pi b (\theta_t-\theta_0+\frac{iQ}2)}, \quad \gamma_{H}=e^{2\pi b (\theta_*+\frac{iQ}2)}, \qquad q=e^{2i\pi b^2},
\eeq
the function $\mathcal{H}$ defined in \eqref{defH} satisfies, for each integer $n\geq 0$,
\beq\label{limHtoHn}
\lim\limits_{\nu \to \nu_n} \mathcal{H}(b,\theta_0,\theta_t,\theta_*,\sigma_s,\nu) = 
H_n(e^{2\pi b \sigma_s};\alpha_H,\beta_{H},\gamma_{H},q),
\eeq
where $H_n$ are the continuous dual $q$-Hahn polynomials defined in \eqref{qhahn}.
\end{theorem}

\begin{proof}
There are two different ways to prove \eqref{limHtoHn}. The first approach consists of taking the limit $\nu\to \nu_n$ in the integral representation \eqref{defH} for $\mathcal{H}$ for each $n$; the second approach computes the limit for $n=0$ and then uses the limit of the difference equation \eqref{diffeqtildeH1} to extend the result to other values of $n$.  The first approach is described in detail in Section \ref{Msec} for the function $\mathcal{M}$. Here we use the second approach.

We first show that the limit in \eqref{limHtoHn} exists for $n=0$ and equals $1$. The function $s_b(\nu-\tfrac{\theta_*}{2}-\theta_t)$ in \eqref{PH} has a simple zero at $\nu_0=\tfrac{\theta_*}{2}+\theta_t+\tfrac{iQ}2$.  Moreover, in the limit $\nu \to \nu_0$, the pole of the function $s_b(x+\tfrac{\theta_*}{2}+\theta_t-\nu)$ in \eqref{IH} at $x_0 :=-\tfrac{iQ}2-\tfrac{\theta_*}{2}-\theta_t+\nu$ collides with the pole of $s_b(x+\tfrac{iQ}2)^{-1}$ located at $x=0$, pinching the contour $\mathsf{H}$.  Therefore, before taking the limit $\nu\to \nu_0$, we deform the contour $\mathsf{H}$ into a contour $\mathsf{H}'$ which passes below $x_0$. We obtain
\beq \label{Hdeform}
\mathcal{H}(b,\theta_0,\theta_t,\theta_*,\sigma_s,\nu) = -2i\pi \; P_\mathcal{H}\lb \sigma_s,\nu\rb \underset{x=x_0}{\text{Res}}\lb I_\mathcal{H}(x,\sigma_s,\nu)\rb + P_\mathcal{H}(\sigma_s,\nu) \displaystyle \int_{\mathsf{H}'} dx \; I_\mathcal{H}(x,\sigma_s,\nu).
\eeq
Using the residue \eqref{ressb},  a straightforward computation yields
\begin{align} \label{resIH}
\nonumber -2i\pi  \underset{x=x_0}{\text{Res}}\lb I_\mathcal{H}(x,\sigma_s,\nu)\rb = & \; e^{i \pi  \left(\nu-\theta_0+\frac{\theta_*}{2} -\frac{i Q}{2}\right) \left(\nu-\frac{\theta_*}{2}-\theta_t -\frac{i Q}{2}\right)} \\
& \times \frac{s_b\left(\theta_0-\frac{\theta_*}{2}+\nu -\frac{i Q}{2}-\sigma_s\right) s_b\left(\theta_0-\frac{\theta_*}{2}+\nu -\frac{i Q}{2}+\sigma_s\right)}{s_b\left(\theta_0+\frac{\theta_*}{2}+\nu \right) s_b\left(\nu-\frac{\theta_*}{2}-\theta_t\right) s_b\left(\nu-\frac{\theta_*}{2}+\theta_t \right)}.
\end{align}
The right-hand side of \eqref{resIH} has a simple pole at $\nu=\nu_0$ due to the factor $s_b(\nu-\tfrac{\theta_*}{2}-\theta_t)^{-1}$. Moreover, in the limit $\nu\to \nu_0$ the second term in \eqref{Hdeform} vanishes thanks to the zero of $P_\mathcal{H}$. Thus we obtain
\beq
\lim\limits_{\nu\to \nu_0} \mathcal{H}(b,\theta_0,\theta_t,\theta_*,\sigma_s,\nu) = -2i\pi \lim\limits_{\nu\to \nu_0}\Big( P_\mathcal{H}\lb \sigma_s,\nu\rb \underset{x=x_0}{\text{Res}}\lb I_\mathcal{H}(x,\sigma_s,\nu)\rb\Big).
\eeq
Using $s_b(x)=s_b(-x)^{-1}$, it is straightforward to verify that the right-hand side equals $1$; this proves \eqref{limHtoHn}.

For each integer $n\geq 0$, let $P_n$ denote the left-hand side of \eqref{limHtoHn}:
\beq\label{limHtoHn2}
P_n := \lim\limits_{\nu\to \nu_n} \mathcal{H}(b,\theta_0,\theta_t,\theta_*,\sigma_s,\nu).
\eeq
The same kind of contour deformation used to establish the case $n=0$ shows that the limit in \eqref{limHtoHn2} exists for all $n \geq 0$. To show that $P_n$ equals the continuous dual $q$-Hahn polynomials $H_n$, we consider the limit $\nu \to \nu_n$ of the difference equation \eqref{diffeqtildeH1}. Using the parameter correspondence \eqref{paramhahn}, it is straightforward to verify that 
\beq
\lim\limits_{\nu\to \nu_n} \tilde{H}_\mathcal{H}(b,\nu) = R_{H_n},
\eeq
where the operators $\tilde{H}_\mathcal{H}$ and $R_{H_n}$ are defined in \eqref{HHtilde} and \eqref{recurrenceoperatorHn}, respectively. Hence taking the limit $\nu\to\nu_n$ of the difference equation \eqref{diffeqtildeH1}, we see that $P_n$ satisfies
\beq \label{RHnPn}
R_{H_n} P_n = \lb z+z^{-1} \rb P_n,
\eeq
where $z=e^{2\pi b \sigma_s}$. Thus the $P_n$ satisfy the same recurrence relation \eqref{recurrenceHn} as the continuous dual $q$-Hahn polynomials evaluated at $z=e^{2\pi b \sigma_s}$. Since we have already shown that $P_0=H_0=1$, we infer that $P_n=H_n$ for all $n\geq 0$ by induction, where $H_n$ is evaluated at $z=e^{2\pi b \sigma_s}$. This completes the proof of \eqref{limHtoHn}.
\end{proof}

\subsection{Second polynomial limit}

In this subsection, we show that $\mathcal{H}$ reduces to the big $q$-Jacobi polynomials when $\sigma_s$ is suitably discretized.

\begin{theorem} [From $\mathcal{H}$ to the big $q$-Jacobi polynomials] \label{thmHtoJn}
Let $\nu \in \{\im \nu > -Q/2\} \setminus \Delta_{\nu}$ and suppose that Assumptions \ref{assumption} and \ref{assumptionhahnjacobi} are satisfied. 
Under the parameter correspondence 
\beq\label{paramjacobi}
\alpha_J = e^{4\pi b \theta_t}, \quad \beta_J = e^{4\pi b \theta_0}, \quad \gamma_J = e^{2\pi b(\theta_0+\theta_*+\theta_t)}, \quad x_J=e^{2\pi b (\theta_t+\frac{\theta_*}2+\frac{iQ}2)}e^{-2\pi b \nu}, \quad q=e^{2i\pi b^2},
\eeq
the function $\mathcal{H}$ defined in \eqref{defH} satisfies, for each integer $n\geq 0$,
\beq\label{limHtoJn}
\lim\limits_{\sigma_s \to \sigma_s^{(n)}} \mathcal{H}(b,\theta_0,\theta_t,\theta_*,\sigma_s,\nu) = J_n(x_J;\alpha_J,\beta_J,\gamma_J;q),
\eeq
where $ \sigma_s^{(n)} \in \mathbb{C}$ is given in \eqref{limsigmas} and where $J_n$ are the big $q$-Jacobi polynomials defined in \eqref{Jn}.
\end{theorem}
\begin{proof}
We first prove that \eqref{limHtoJn} holds for $n=0$. The function $s_b(\sigma_s-\theta_0-\theta_t)$ in \eqref{PH} has a simple zero located at $\sigma_s=\sigma_s^{(0)}=\theta_0+\theta_t+\tfrac{iQ}2$. On the other hand, in the limit $\sigma_s=\sigma_s^{(0)}$, the contour $\mathsf{H}$ is squeezed between the pole of $s_b(x+\theta_0+\theta_t-\sigma_s)$ in \eqref{IH} located at $x_0=-\tfrac{iQ}2+\sigma_s-\theta_0-\theta_t$ and the pole of $s_b(x+\tfrac{iQ}2)^{-1}$ at $x=0$. Therefore, before taking the limit $\sigma_s \to \sigma_s^{(0)}$, we deform the contour $\mathsf{H}$ into a contour $\mathsf{H}'$ which passes below $x_0$. We obtain
\beq \label{Hdeform2}
\mathcal{H}(b,\theta_0,\theta_t,\theta_*,\sigma_s,\nu) = -2i\pi \; P_\mathcal{H}\lb \sigma_s,\nu\rb \underset{x=x_0}{\text{Res}}\lb I_\mathcal{H}(x,\sigma_s,\nu)\rb + P_\mathcal{H}\lb \sigma_s,\nu\rb \displaystyle \int_{\mathsf{H}'} dx \; I_\mathcal{H}(x,\sigma_s,\nu).
\eeq
A straightforward computation using \eqref{ressb} shows that
\begin{align*}
-2i\pi  \underset{x=x_0}{\text{Res}} I_\mathcal{H}(x,\sigma_s,\nu) = &\; e^{i \pi  \left(-\theta_0+\frac{\theta_*}{2}+\nu -\frac{i Q}{2}\right) \left(-\theta_0-\theta_t-\frac{i Q}{2}+\sigma_s\right)} 
 \frac{s_b\left(2 \sigma_s-\frac{i Q}{2}\right) s_b\left(\sigma_s-\theta_0+\frac{\theta_*}{2}-\nu -\frac{i Q}{2}\right)}{s_b(\theta_*+\sigma_s) s_b(\sigma_s-\theta_0-\theta_t) s_b(\sigma_s-\theta_0+\theta_t)}.
\end{align*}
The right-hand side has a simple pole at $\sigma_s=\sigma_s^{(0)}$ due to the factor $s_b(\sigma_s-\theta_0-\theta_t)^{-1}$. Moreover, the second term in \eqref{Hdeform2} vanishes at $\sigma_s=\sigma_s^{(0)}$ thanks to the zero of $P_\mathcal{H}$. Therefore,
\beq
\lim\limits_{\sigma_s\to\sigma_s^{(0)}} \mathcal{H}(b,\theta_0,\theta_t,\theta_*,\sigma_s,\nu) = -2i\pi \lim\limits_{\sigma_s\to\sigma_s^{(0)}} \Big(P_\mathcal{H}\lb \sigma_s,\nu\rb \underset{x=x_0}{\text{Res}}\lb I_\mathcal{H}(x,\sigma_s,\nu)\rb\Big).
\eeq
A straightforward computation using $s_b(x)=s_b(-x)^{-1}$ shows that the right-hand side equals $1$; this proves \eqref{limHtoJn} for $n=0$.

We now use the difference equation \eqref{diffeqH1} to show that \eqref{limHtoJn} holds also for $n \geq 1$. Let
\beq
P_n := \lim\limits_{\sigma_s\to\sigma_s^{(n)}} \mathcal{H}(b,\theta_0,\theta_t,\theta_*,\sigma_s,\nu), \qquad n = 0, 1, \dots.
\eeq
The same kind of contour deformation used for $n=0$ shows that the limit in \eqref{limHtoJn} exists for all $n \geq 0$.  
Moreover, under the parameter correspondence \eqref{paramjacobi}, we have
\beq
 \lim\limits_{\sigma_s\to\sigma_s^{(n)}} H_\mathcal{H}(b,\sigma_s) = e^{-2\pi b(\theta_t+\frac{\theta_*}2+\frac{iQ}2)} R_{J_n},
\eeq
where the operators $H_\mathcal{H}$ and $R_{J_n}$ are defined in \eqref{HH} and \eqref{recurrenceoperatorJn}, respectively. Hence taking the limit $\sigma_s\to\sigma_s^{(n)}$ of \eqref{diffeqH1} and recalling that $x_J =e^{2\pi b (\theta_t+\frac{\theta_*}2+\frac{iQ}2)}e^{-2\pi b \nu}$, we see that $P_n$ satisfies $R_{J_n} P_n = x_J P_n$. Thus the $P_n$ satisfy the same recurrence relation \eqref{recurrenceJn} as the big $q$-Jacobi polynomials $J_n$, and the limit \eqref{limHtoJn} for $n\geq 1$ follows by induction.
\end{proof}

\begin{remark}
It is also possible to define a function $\mathcal{H}'$ by sending $\Lambda \to +\infty$ in \eqref{fromRtoH} instead of $\Lambda \to -\infty$:
\beq\label{fromAtoHprime}
\mathcal{H}'(b,\theta_0,\theta_t,\theta_*,\sigma_s,\nu) = \lim\limits_{\Lambda\to +\infty} \mathcal{R}\left[\substack{\frac{\Lambda+\theta_*}2\;\; \theta_t\vspace{0.1cm}\\ \frac{\Lambda-\theta_*}2\;\; \theta_0};\substack{\sigma_s \vspace{0.15cm} \\ \frac{\Lambda}2+\nu}\right].
\eeq
It can be shown that the limit in \eqref{fromAtoHprime} exists for $(\sigma_s,\nu) \in D_{\mathcal{H}}$.  Moreover, due to the asymptotic formula \eqref{sbasymptotics} for $s_b$, the only difference between $\mathcal{H}'$ and $\mathcal{H}$ resides in the sign of the phases in the representation \eqref{defH}.  In fact, the following two limits hold:
\begin{align}
& \lim\limits_{\nu \to \nu_n} \mathcal{H}'(b,\theta_0,\theta_t,\theta_*,\sigma_s,\nu) = 
H_n(e^{2\pi b \sigma_s};\alpha_H^{-1},\beta_{H}^{-1},\gamma_{H}^{-1},q^{-1}),  \\
& \lim\limits_{\sigma_s \to \sigma_s^{(n)}} \mathcal{H}'(b,\theta_0,\theta_t,\theta_*,\sigma_s,\nu) = J_n(x_J^{-1};\alpha_J^{-1},\beta_J^{-1},\gamma_J^{-1};q^{-1}).
\end{align}
We expect that a similar phenomenon is present for all the elements of the non-polynomial scheme. For simplicity, we will only study one of the two representatives for each element.
\end{remark}

\section{The function $\mathcal{S}$}
In this section, we introduce the function $\mathcal{S}(b,\theta_0,\theta_t,\sigma_s,\rho)$ which is one of the two elements at the third level of the non-polynomial scheme, see Figure \ref{nonpolynomialscheme}. The function $\mathcal{S}$ is defined as a confluent limit of the function $\mathcal{H}$. We show that $\mathcal{S}$ is a joint eigenfunction of four difference operators and that it reduces to the Al-Salam Chihara and the little $q$-Jacobi polynomials, which lie at the third level of the $q$-Askey scheme, when $\rho$ and $\sigma_s$ are suitably discretized, respectively. 

\subsection{Definition and integral representation}
Let $\rho$ be a new parameter defined in terms of $\theta_*$ and $\nu$ by
\beq
\nu = \frac{\theta_*}2+\rho.
\eeq
Define the open set $D_\mathcal{S} \subset \mathbb{C}^2$ by
\beq \label{sigmasrho}
D_\mathcal{S} := (\mathbb{C} \backslash \Delta_{\mathcal S,\sigma_s}) \times (\{\im \rho > -Q/2\} \backslash \Delta_\rho),
\eeq 
where the discrete subsets $\Delta_{\mathcal S,\sigma_s}$ and $\Delta_\rho$ are given by
\begin{align*} 
& \Delta_{\mathcal S,\sigma_s} := \{\pm \sigma_s \, |\, \sigma_s \in \Delta_{\mathcal{S}, \sigma_s}'\},
	\\
& \Delta_\rho := \{\pm \theta_t+\tfrac{iQ}2+ibm+ilb^{-1} \}_{m,l=0}^\infty \cup \{\theta_t-\tfrac{iQ}2-ibm-ilb^{-1} \}_{m,l=0}^\infty,
\end{align*}
with
\begin{align*}
\Delta_{\mathcal S,\sigma_s}' := &\; \{\theta_0 \pm \theta_t + \tfrac{iQ}{2} + i m b +il b^{-1}\}_{m,l=0}^\infty \cup \{\theta_0+ \theta_t -\tfrac{i Q}{2} -i m b-il b^{-1}\}_{m,l=0}^\infty.
\end{align*}

\begin{definition} \label{DefS}
Let $\mathcal{H}$ be defined by \eqref{fromRtoH}. The function $\mathcal{S}$ is defined for $(\sigma_s,\rho) \in D_\mathcal{S}$ by 
\beq \label{fromHtoS}
\mathcal{S}(b,\theta_0,\theta_t,\sigma_s,\rho) = \lim\limits_{\theta_* \to -\infty} \mathcal{H}\lb b,\theta_0,\theta_t,\theta_*,\sigma_s,\tfrac{\theta_*}{2}+\rho\rb
\eeq 
and is extended meromorphically to $(\sigma_s,\rho) \in \mathbb{C}^2$. 
\end{definition}

The next theorem shows that $\mathcal{S}$ is a well-defined meromorphic function of $(\sigma_s,\rho) \in \mathbb{C}^2$.

\begin{theorem} \label{thmforS}
Suppose that Assumption \ref{assumption} is satisfied. The limit in \eqref{fromHtoS} exists uniformly for $(\sigma_s,\rho)$ in compact subsets of $D_\mathcal{S}$.
Moreover, the function $\mathcal{S}$ is an analytic function of $(\sigma_s,\rho) \in (\mathbb{C} \backslash \Delta_{\mathcal S,\sigma_s}) \times (\mathbb{C}\backslash \Delta_\rho)$ and admits the following integral representation:
\beq \label{defS}
\mathcal{S}(b,\theta_0,\theta_t,\sigma_s,\rho) = P_\mathcal{S}(\sigma_s,\rho) \int_{\mathsf{S}} dx~I_\mathcal{S}(x,\sigma_s,\rho) \qquad \text{for $(\sigma_s,\rho) \in (\mathbb{C} \backslash \Delta_{\mathcal S,\sigma_s}) \times (\mathbb{C}\backslash \Delta_\rho)$}, \eeq
where the dependence of $P_\mathcal{S}$ and $I_\mathcal{S}$ on $b,\theta_0,\theta_t$ is omitted for simplicity, 
\begin{align}
\label{PS} & P_\mathcal{S}(\sigma_s,\rho) = s_b\left(2 \theta_t+\tfrac{i Q}{2}\right)  s_b\lb \rho- \theta_t \rb \prod_{\epsilon=\pm 1} s_b \lb \epsilon \sigma_s -\theta_0-\theta_t\rb, 
	\\
\label{IS} & I_\mathcal{S}(x,\sigma_s,\rho) = e^{-\frac{i\pi x^2}{2} -i \pi  x \left(iQ+2\theta_0+\theta_t-\rho\right)} \frac{s_b(x+\theta_t-\rho) }{s_b\left(x+\frac{i Q}{2}\right) s_b\left(x+\frac{i Q}{2}+2 \theta_t\right)} \prod_{\epsilon=\pm 1} s_b(x+\theta_0+\theta_t+\epsilon \sigma_s),
\end{align}
and the contour $\mathsf{S}$ is any curve from $-\infty$ to $+\infty$ which separates the three decreasing from the two increasing sequences of poles, with the requirement that its right tail satisfies
\beq \label{conditionimxforS}
\im x -\im \rho < -\delta \qquad \text{for all $x \in \mathsf{S}$ with $\re x$ sufficiently large},
\eeq
for some $\delta > 0$.
In particular, $\mathcal{S}$ is a meromorphic function of $(\sigma_s,\rho) \in \mathbb{C}^2$.
If $(\sigma_s,\rho) \in \mathbb{R}^2$, then the contour $\mathsf S$ can be any curve from $-\infty$ to $+\infty$ lying within the strip $\im x \in (-Q/2,-\delta)$. 
\end{theorem}
\begin{proof}
The proof is similar to the proof of Theorem \ref{thmforH}, but there are some differences because the exponent in \eqref{IS} is a quadratic (rather than a linear) polynomial in $x$.
Let $(b,\theta_0,\theta_t) \in (0,\infty) \times \mathbb{R}^2$. It can be verified that
\beq \label{PHIH}
P_\mathcal{H}\lb \sigma_s,\tfrac{\theta_*}{2}+\rho\rb I_\mathcal{H}\lb x,\sigma_s,\tfrac{\theta_*}{2}+\rho\rb = P_\mathcal{S}(\sigma_s,\rho) Z(x,\theta_*)  I_\mathcal{S}(x,\sigma_s,\rho),
\eeq
where
\beq \label{ZS}
Z(x,\theta_*) = e^{\frac{i \pi  x^2}{2}} e^{i \pi  x \left(\theta_0+\theta_*+\theta_t+\frac{i Q}{2}\right)} \frac{s_b\left(\theta_0+\theta_*+\theta_t+\frac{i Q}{2}\right)}{s_b\left(x+\theta_0+\theta_*+\theta_t+\frac{i Q}{2}\right)}.
\eeq
Due to the properties \eqref{polesb} of the function $s_b$, the function $I_\mathcal{S}(\cdot,\sigma_s,\rho)$ possesses two increasing sequences of poles starting at $x=0$ and $x=-2\theta_t$, as well as three decreasing sequences of poles starting at $x=-\tfrac{iQ}2+\rho-\theta_t$ and $x=-\tfrac{iQ}2\pm \sigma_s-\theta_0-\theta_t$. The discrete sets $\Delta_{\mathcal S,\sigma_s}$ and $\Delta_\rho$ contain all the values of $\sigma_s$ and $\rho$, respectively, for which poles in any of the two increasing sequences collide with poles in any of the three decreasing sequences. The sets $\Delta_{\mathcal S,\sigma_s}$ and $\Delta_\rho$ also contain all the values of $\sigma_s$ and $\rho$ at which the prefactor $P_\mathcal{S}(\sigma_s,\rho)$ has poles. 
Furthermore, $Z(x,\theta_*)$ possesses one increasing sequence of poles starting at $x=-\theta_0-\theta_*-\theta_t$ which lies in the half-plane $\im x \geq 0$. The real parts of the poles in this sequence tend to $+\infty$ as $\theta_*\to -\infty$.  

Let $K_{\sigma_s}$ and $K_\rho$ be compact subsets of $\mathbb C \backslash \Delta_{\mathcal S,\sigma_s}$ and $\{\im \rho > -Q/2\} \backslash \Delta_\rho$, respectively. Suppose $(\sigma_s,\rho) \in K_{\sigma_s} \times K_\rho$. Then, the above discussion shows that it is possible to choose a contour $\mathsf S = \mathsf S(\sigma_s,\rho)$ from $-\infty$ to $+\infty$ which separates the two upward from the three downward sequences of poles of $Z(\cdot,\theta_*) I_\mathcal{S}(\cdot,\sigma_s,\rho)$.
Let us choose $\mathsf S$ so that its right tail approaches the horizontal line $\im x = -Q/2-\delta$ as $\re x \to +\infty$ for some $\delta>0$. Then there is an $N>0$ such that $\mathsf{S}$ is independent of $\theta_*$ for $\theta_*<-N$, and \eqref{defH} and \eqref{PHIH} imply that, for all $(\sigma_s,\rho) \in K_{\sigma_s} \times K_\rho$ and all $\theta_* < -N$,
\beq \label{HequalsPSZIS}
\mathcal{H}\lb b,\theta_0,\theta_t,\theta_*,\sigma_s,\tfrac{\theta_*}{2}+\rho\rb = P_\mathcal{S}(\sigma_s,\rho) \displaystyle \int_{\mathsf{S}} dx \; Z(x,\theta_*) I_\mathcal{S}(x,\sigma_s,\rho).
\eeq

Utilizing the asymptotic formula \eqref{sbasymptotics}, it can be verified that the following limit
\beq \label{limZ}
\lim\limits_{\theta_* \to -\infty} Z(x,\theta_*) = 1,
\eeq
holds uniformly for $(\sigma_s,\rho) \in K_{\sigma_s} \times K_\rho$ and for $x$ in bounded subsets of $\mathsf{S}$.
Moreover, using \eqref{sbasymptotics} with $\epsilon=1/2$ we find that $I_\mathcal{S}$ obeys the estimates
\begin{align} \label{ISestimate}
 I_\mathcal{S}(x,\sigma_s,\rho) = \begin{cases}
O\lb e^{2\pi \lvert \re x \lvert(\im x-\im \rho)}\rb, & \re x \to +\infty, \\
O\lb e^{-2\pi Q \lvert\re x\lvert}\rb, & \re x \to -\infty,
\end{cases}
\end{align}
uniformly for $(\sigma_s,\rho)$ in compact subsets of $K_{\sigma_s} \times K_\rho$ and $\im x$ in compact subsets of $\mathbb{R}$.  Since the contour $\mathsf{S}$ stays a bounded distance away from the poles of the integrand $I_\mathcal{S}$, we infer that there exists a constant $C_1>0$ such that 
\beq \label{inequalityIS}
\lvert  I_\mathcal{S}(x,\sigma_s,\rho) \rvert \leq C_1 \times \begin{cases} \; e^{2\pi \lvert \re x\lvert(\im x-\im \rho)}, & \re x \geq 0, \\ \; e^{-2\pi Q \lvert\re x\lvert}, & \re x \leq 0, \end{cases}
\eeq
uniformly for $(\sigma_s,\rho)$ in compact subsets of $K_{\sigma_s} \times K_\rho$.  In particular, the integrand $I_\mathcal{S}$ has exponential decay along the left and right tails of the contour $\mathsf S$.

Suppose we can show that there exist constants $c>0$ and $C>0$ such that
\beq \label{boundZIS}
\lvert Z(x,\theta_*) I_\mathcal{S}(x,\sigma_s,\rho) \rvert \leq C e^{-c \lvert \re x \lvert},
\eeq
uniformly for all $\theta_*<-N$, $x\in \mathsf{S}$ and $(\sigma_s,\rho)\in K_{\sigma_s} \times K_\rho$. Then it follows from \eqref{HequalsPSZIS}, \eqref{limZ} and Lebesgue's dominated convergence theorem that the limit in \eqref{fromHtoS} exists uniformly for $(\sigma_s,\rho)\in K_{\sigma_s} \times K_\rho$ and is given by \eqref{defS}. This proves that the limit in \eqref{fromHtoS} exists uniformly for $(\sigma_s,\rho)$ in compact subsets of $D_\mathcal{S}$ and proves \eqref{defS} for $(\sigma_s,\rho) \in D_\mathcal{S}$. 
The analyticity of $\mathcal S$ as a function of $(\sigma_s,\rho) \in D_\mathcal{S}$ follows from the analyticity of $\mathcal H$ together with the uniform convergence on compact subsets.
Whenever the condition \eqref{conditionimxforS} holds, the estimate \eqref{inequalityIS} ensures that the integral representation \eqref{defS}  provides a meromorphic continuation of $\mathcal{S}$ to $(\sigma_s,\rho) \in \mathbb{C}^2$ which is analytic for $(\sigma_s,\rho) \in (\mathbb{C} \backslash \Delta_{\mathcal S,\sigma_s}) \times (\mathbb{C}\backslash \Delta_\rho)$.

To complete the proof of the theorem, it only remains to prove \eqref{boundZIS}. 
The asymptotic formula \eqref{sbasymptotics} for $s_b$ with $\epsilon=1/2$ implies that there exist constants $C_2,C_3,C_4>0$ such that the inequalitites
\begin{align}
& \lvert s_b\left(\theta_0+\theta_*+\theta_t+\tfrac{i Q}{2}\right) \lvert \leq C_2 \; e^{\frac{\pi Q \lvert \theta_* \lvert}2}, \qquad \theta_*<-N,  \\
& \lvert s_b\left(x+\theta_0+\theta_*+\theta_t+\tfrac{i Q}{2}\right)^{-1} \lvert \leq C_3 \; e^{-\frac{\pi Q \lvert \theta_*+\re x \lvert}2} e^{-\pi (\im x) \lvert \theta_*+\re x \lvert },  \qquad x \in \mathsf{S}, \; \theta_* \in \mathbb{R}, \\
& \lvert e^{\frac{i \pi  x^2}{2}} e^{i \pi  x \left(\theta_0+\theta_*+\theta_t+\frac{i Q}{2}\right)} \lvert \leq C_4 \; e^{-\pi(\im x)(\theta_*+\re x)} e^{-\frac{\pi Q \re x}2}, \qquad x \in \mathsf{S}, 
\end{align} 
hold uniformly for $(\sigma_s,\rho)\in K_{\sigma_s} \times K_\rho$. Therefore, in view of \eqref{ZS}, there exists a constant $C_5>0$ such that
\beq \label{inequalityforS}
\lvert Z(x,\theta_*) \rvert \leq C_5 \; e^{-\pi (\im x+\frac{Q}2)(\theta_*+\re x+\lvert \theta_*+\re x \lvert)}, \quad x \in \mathsf S, \; \theta_*<-N,
\eeq
uniformly for $(\sigma_s,\rho)\in K_{\sigma_s} \times K_\rho$. The inequality \eqref{inequalityforS} can be rewritten as follows:
\beq \label{inequalityZS}
\lvert Z(x,\theta_*) \rvert \leq \begin{cases} C_5 \; e^{-2\pi(\im x+\frac{Q}2)(\theta_*+\re x)}, & \theta_* + \re x \geq 0, \\ C_5, &  \theta_* + \re x \leq 0, \end{cases} \quad x \in \mathsf S,  ~ \theta_* < -N,
\eeq
uniformly for $(\sigma_s,\rho)\in K_{\sigma_s} \times K_\rho$.  The inequalities \eqref{inequalityIS} and \eqref{inequalityZS} yield the existence of a constant $C_6>0$ independent of $x \in \mathsf S$,  $\theta_* < -N$ and $(\sigma_s,\rho)\in K_{\sigma_s} \times K_\rho$ such that 
\beq \label{ineqPHIH}
\lvert Z(x,\theta_*) I_\mathcal{S}(x,\sigma_s,\rho) \lvert \leq \begin{cases}
C_6 \; e^{-2\pi (\re x) (\frac{Q}2+\im \rho)} e^{-2\pi \theta_*(\frac{Q}2+\im x)}, & \re x \geq -\theta_*, \\
C_6 \; e^{2\pi (\re x)(\im x-\im \rho)}, & 0 \leq \re x \leq -\theta_*,  \\
C_6 \; e^{-2\pi Q |\re x|}, & \re x \leq 0,
\end{cases}
\eeq
uniformly for $(\sigma_s,\rho)\in K_{\sigma_s} \times K_\rho$.  Since $\im \rho > -Q/2 + \delta$ for $\rho \in K_\rho$ and $\im x < -Q/2 -\delta$ on the right tail of the contour for some $\delta > 0$, this proves \eqref{boundZIS} and completes the proof. 
\end{proof}

Furthermore, thanks to the property \eqref{sbinverse} of $s_b$ the function $\mathcal{S}$ has the symmetry
\beq \label{Sbinverse}
\mathcal{S}(b,\theta_0,\theta_t,\sigma_s,\rho)=\mathcal{S}(b^{-1},\theta_0,\theta_t,\sigma_s,\rho).
\eeq

\subsection{Difference equations} 
By taking the confluent limit \eqref{fromHtoS} of the difference equations \eqref{diffeqH} and \eqref{diffeqtildeH} satisfied by $\mathcal{H}$, it follows that the function $\mathcal{S}$ is a joint eigenfunction of four different difference operators, two acting on $\sigma_s$ and the remaining two acting on $\rho$. The four difference equations will hold as equalities between meromorphic functions on $\mathbb C^2$. Since the derivations are similar to those presented in Section \ref{Hdifferencesubsec}, we state these results without proofs.

\subsubsection{First pair of difference equations}

Define the difference operator $H_{\mathcal{S}}(b,\sigma_s)$ by
\beq\label{HS}
H_{\mathcal{S}}(b,\sigma_s) = H^+_{\mathcal{S}}(b,\sigma_s)e^{ib\partial_{\sigma_s}}+H^+_{\mathcal{S}}(b,-\sigma_s)e^{-ib\partial_{\sigma_s}}+H^0_{\mathcal{S}}(b,\sigma_s),
\eeq
where
\begin{align}
\label{HSplussigma} & H^+_{\mathcal{S}}(b,\sigma_s) = e^{-\pi b (2 \sigma_s+i Q)} \frac{\prod_{\epsilon=\pm1}\cosh \left(\pi  b \left(\frac{ib}{2}+\epsilon\theta_0+\theta_t+\sigma_s\right)\right) }{\text{sinh}(\pi  b (2 \sigma_s+i b)) \text{sinh}(2 \pi  b \sigma_s)}, \\
\label{HS0sigma} & H^0_{\mathcal{S}}(b,\sigma_s)=e^{-\pi  b (2 \theta_t+i Q)}-H^+_{\mathcal{S}}(b,\sigma_s)-H^+_{\mathcal{S}}(b,-\sigma_s).
\end{align}
\begin{proposition}
For $(\sigma_s,\rho)\in \mathbb C^2$, the function $\mathcal{S}$ satisfies the pair of difference equations
\begin{subequations} \label{diffeqS} \begin{align}
\label{diffeqS1} & H_{\mathcal{S}}(b,\sigma_s) \; \mathcal{S}(b,\theta_0,\theta_t,\sigma_s,\rho) = e^{-2\pi b \rho} \; \mathcal{S}(b,\theta_0,\theta_t,\sigma_s,\rho), \\
\label{diffeqS2} & H_{\mathcal{S}}(b^{-1},\sigma_s) \; \mathcal{S}(b,\theta_0,\theta_t,\sigma_s,\rho) =  e^{-2\pi b^{-1} \rho} \; \mathcal{S}(b,\theta_0,\theta_t,\sigma_s,\rho).
\end{align}\end{subequations}
\end{proposition}

\subsubsection{Second pair of difference equations}
Define the difference operator $\tilde{H}_{\mathcal{S}}(b,\rho)$ such that
\beq\label{HStilde}
\tilde{H}_{\mathcal{S}}(b,\rho) = \tilde{H}^+_{\mathcal{S}}(b,\rho) e^{ib \partial_\rho} + \tilde{H}^-_{\mathcal{S}}(b,\rho) e^{-ib \partial_\rho} + \tilde{H}^0_{\mathcal{S}}(b,\rho),
\eeq
where
\begin{align}\label{HSpmrho}
& \tilde{H}^\pm_{\mathcal{S}}(b,\rho) = -2 e^{\pi b \rho} e^{\mp \pi  b \left(\tfrac{i b}{2}+2 \theta_0+\theta_t \right)} \cosh \left(\pi  b \left(\tfrac{i b}{2}+\theta_t \pm\rho \right)\right), \\ 
\label{HS0rho} & \tilde{H}^0_{\mathcal{S}}(b,\rho) = - 2 \cosh \left(2 \pi  b \left(\tfrac{i b}{2}+\theta_0+\theta_t\right)\right)-\tilde{H}^+_{\mathcal{S}}(b,\rho)-\tilde{H}^-_{\mathcal{S}}(b,\rho) .
\end{align}

\begin{proposition}
For $(\sigma_s,\rho)\in \mathbb C^2$, the function $\mathcal{S}$ satisfies the following pair of difference equations:
\begin{subequations}\label{diffeqtildeS} \begin{align}
\label{diffeqtilde1S}& \tilde{H}_{\mathcal{S}}(b,\rho)\; \mathcal{S}(b,\theta_0,\theta_t,\sigma_s,\rho) = 2\operatorname{cosh}{(2\pi b \sigma_s)}\; \mathcal{S}(b,\theta_0,\theta_t,\sigma_s,\rho), \\
\label{diffeqtilde2S} & \tilde{H}_{\mathcal{S}}(b^{-1},\rho)\;\mathcal{S}(b,\theta_0,\theta_t,\sigma_s,\rho) = 2\operatorname{cosh}{(2\pi b^{-1} \sigma_s)}\; \mathcal{S}(b,\theta_0,\theta_t,\sigma_s,\rho).
\end{align}\end{subequations}
\end{proposition}

\subsection{Polynomial limits}
Our next two theorems state that $\mathcal{S}$ reduces to the Al-Salam Chihara polynomials when the variable $\rho$ is suitably discretized and to the little $q$-Jacobi polynomials when $\sigma_s$ is suitably discretized. The proofs proceed along the same lines as the proofs of Theorem \ref{thhahn} and Theorem \ref{thmHtoJn} and are therefore omitted. 

We make the following assumption which ensures that the poles of the integrand $I_\mathcal{S}$ in \eqref{IS} are simple.

\begin{assumption}\label{assumptionS}
Assume that $b>0$ is such that $b^2$ is irrational. Moreover, assume that
\beq
\theta_t, \re \sigma_s \neq 0, \qquad \re\big(\pm \sigma_s+ \rho +\theta_0\big) \neq 0.
\eeq
\end{assumption}
 
\begin{theorem}[From $\mathcal{S}$ to the Al-Salam Chihara polynomials] \label{thmStoSn} Let $\sigma_s \in \mathbb{C} \backslash \Delta_{\mathcal S,\sigma_s}$ and suppose that Assumptions \ref{assumption} and \ref{assumptionS} are satisfied.  
Define $\{\rho_n\}_{n=0}^\infty \subset \mathbb{C}$ by
 \beq \label{rhon}
 \rho_n = \theta_t+\tfrac{i Q}2+inb.
 \eeq
 Under the parameter correspondence
  \beq \label{paramSn}
 \alpha_{S}=e^{2\pi b(\theta_t+\theta_0+\frac{iQ}2)}, \qquad \beta_{S}= e^{2\pi b(\theta_t-\theta_0+\frac{iQ}2)}, \qquad q=e^{2i\pi b^2},
 \eeq
the function $\mathcal{S}$ satisfies, for each $n\geq 0$,
 \beq \label{limStoSn}
 \lim\limits_{\rho\to \rho_n} \mathcal{S}(b,\theta_0,\theta_t,\sigma_s,\rho) = S_n\lb e^{2\pi b \sigma_s};\alpha_{S},\beta_{S};q\rb, 
 \eeq
 where $S_n$ are the Al-Salam Chihara polynomials defined in \eqref{Sn}.
 \end{theorem}
 
 \begin{theorem}[From $\mathcal{S}$ to the little $q$-Jacobi polynomials] \label{thmStoYn}
 Let $\rho \in \mathbb{C}\backslash \Delta_\rho$ and suppose that Assumptions \ref{assumption} and \ref{assumptionS} are satisfied. Under the parameter correspondence
 \beq\label{paramYn}
\alpha_Y = e^{4\pi b \theta_t}, \qquad \beta_Y = e^{4\pi b \theta_0}, \qquad x_Y = e^{\pi b(iQ+2\theta_t)} e^{-2\pi b \rho},  \qquad q= e^{2i\pi b^2},
 \eeq
the function $\mathcal{S}$ satisfies, for each $n\geq 0$,
 \beq \label{limStoYn}
 \lim\limits_{\sigma_s\to \sigma_s^{(n)}} \mathcal{S}(b,\theta_0,\theta_t,\sigma_s,\rho) = Y_n \lb x_Y;\alpha_Y,\beta_Y,q  \rb,
 \eeq
 where $\sigma_s^{(n)}$ is defined in \eqref{limsigmas} and where $Y_n$ are the little $q$-Jacobi polynomials defined in \eqref{Yn}.
 \end{theorem}

\section{The function $\mathcal{X}$}
In this section, we define the function $\mathcal{X}(b,\theta,\sigma_s,\omega)$ which generalizes continuous big $q$-Hermite polynomials. It lies at the fourth level of the non-polynomial scheme and is defined as a confluent limit of $\mathcal{S}$. We show that $\mathcal{X}$ is a joint eigenfunction of four difference operators and that it reduces to the continuous big $q$-Hermite polynomials, which lie at the fourth level of the $q$-Askey scheme, when $\omega$ is suitably discretized.

\subsection{Definition and integral representation}
Let $\theta$ and $\omega$ be two new parameters defined by
\beq\label{deftheta}
\theta_0 = \frac{\theta+\Lambda}2, \qquad \theta_t = \frac{\theta-\Lambda}2,\qquad \rho = -\frac{\Lambda}2+\omega.
\eeq
\begin{definition}
The function $\mathcal{X}$ is defined by 
\beq\label{fromStoX}
\mathcal{X}(b,\theta,\sigma_s,\omega) = \lim\limits_{\Lambda\to +\infty} \mathcal{S}\lb b,\tfrac{\theta+\Lambda}2, \tfrac{\theta-\Lambda}2,\sigma_s,-\tfrac{\Lambda}2+\omega\rb,
\eeq
where $\mathcal{S}$ is given in \eqref{defS}.
\end{definition}

The next theorem shows that for each choice of $(b,\theta) \in (0,\infty) \times \mathbb R$, $\mathcal{X}$ is a well-defined analytic function of $(\sigma_s,\omega) \in (\mathbb{C} \backslash \Delta_{\mathcal X,\sigma_s}) \times (\mathbb{C} \backslash \Delta_\omega)$, where $\Delta_{\mathcal S,\sigma_s}, \Delta_\omega \subset \mathbb C$ are discrete sets of points where $\mathcal S$ may have poles. The proof is omitted since it involves computations which are similar to those presented in the proofs of Theorems \ref{thmforH} and \ref{thmforS}.

\begin{theorem} \label{thmforX}
Suppose that Assumption \ref{assumption} is satisfied. The limit \eqref{fromStoX} exists uniformly for $(\sigma_s,\omega)$ in compact subsets of
\beq
D_\mathcal{X} := (\mathbb C \backslash \Delta_{\mathcal X,\sigma_s}) \times (\mathbb{C} \backslash \Delta_\omega),
\eeq
where 
\begin{align*}
& \Delta_{\mathcal X,\sigma_s} := \{\pm \sigma_s \, |\, \sigma_s \in \Delta_{\mathcal{X}, \sigma_s}'\},
	\\ 
& \Delta_\omega := \{\tfrac{\theta}2+\tfrac{iQ}2+ibm+ilb^{-1} \}_{m,l=0}^\infty \cup \{\tfrac{\theta}2-\tfrac{iQ}2-ibm-ilb^{-1} \}_{m,l=0}^\infty,
\end{align*}
with
\begin{align*}
& \Delta_{\mathcal X,\sigma_s}' := \{\theta + \tfrac{iQ}{2} + i m b +il b^{-1}\}_{m,l=0}^\infty
\cup \{\theta - \tfrac{iQ}{2} - i m b - il b^{-1}\}_{m,l=0}^\infty.
\end{align*}
Moreover, the function $\mathcal{X}$ admits the following integral representation:
\beq\label{defX}
\mathcal{X}(b,\theta,\sigma_s,\omega) = P_\mathcal{X}(\sigma_s,\omega) \displaystyle \int_{\mathsf{X}} dx \; I_\mathcal{X}(x,\sigma_s,\omega) \qquad \text{for $(\sigma_s,\omega) \in D_\mathcal{X}$},
\eeq
where
\begin{align}\label{PX}
& P_\mathcal{X}(\sigma_s,\omega) = s_b\left(\omega -\tfrac{\theta }{2}\right) \prod_{\epsilon=\pm 1} s_b(\epsilon \sigma_s-\theta), \\
\label{IX} &
I_\mathcal{X}(x,\sigma_s,\omega) = e^{-i \pi  x^2} e^{-i \pi  x \left(\frac{5\theta }{2}+\frac{3iQ}{2}-\omega \right)} \frac{s_b\left(x+\frac{\theta }{2}-\omega \right)}{s_b(x+\frac{i Q}{2})} \prod_{\epsilon=\pm 1} s_b\left(x+\theta+\epsilon\sigma_s\right),
\end{align}
and the contour $\mathsf{X}$ is any curve from $-\infty$ to $+\infty$ which separates the three decreasing from the increasing sequences of poles, with the requirement that its right tail satisfies
\beq \label{conditionimxforX}
\im x + \frac{Q}{4} - \frac{\im \omega}{2} <  - \delta
\qquad \text{for $x \in \mathsf{X}$ with $\re x$ sufficiently large},
\eeq
for some $\delta >0$. In particular, $\mathcal X$ is a meromorphic function of $(\sigma_s,\omega) \in \mathbb C^2$.
If $(\sigma_s,\omega) \in \mathbb{R}^2$, then the contour $\mathsf{X}$ can be any curve from $-\infty$ to $+\infty$ lying within the strip $\im x \in (-Q/2,-Q/4-\delta)$.
\end{theorem}

Furthermore, as a consequence of \eqref{sbinverse}, $\mathcal{X}$ satisfies
\beq \label{Xinverseb}
\mathcal{X}(b,\theta,\sigma_s,\omega) = \mathcal{X}(b^{-1},\theta,\sigma_s,\omega).
\eeq

\subsection{Difference equations}
The function $\mathcal{X}(b,\theta,\sigma_s,\omega)$ is a joint eigenfunction of four difference operators, two acting on $\sigma_s$ and the remaining two on $\omega$.
This follows by taking the confluent limit \eqref{fromStoX} of the difference equations \eqref{diffeqS} and \eqref{diffeqtildeS} satisfied by $\mathcal{S}$. 
The proofs of the following two propositions are omitted since they are similar to those presented in Section \ref{Hdifferencesubsec}.

\subsubsection{First pair of equations}

Introduce a difference operator $H_{\mathcal{X}}(b,\sigma_s)$ such that
\beq\label{HX}
H_{\mathcal{X}}(b,\sigma_s) = H_{\mathcal{X}}^+(b,\sigma_s) e^{ib \partial_{\sigma_s}} + H_{\mathcal{X}}^+(b,-\sigma_s) e^{-ib \partial_{\sigma_s}} + H_{\mathcal{X}}^0(b,\sigma_s),
\eeq
where
\begin{align}
\label{HXplussigmas} & H_{\mathcal{X}}^+(b,\sigma_s) = -e^{-\frac{3}{2} \pi  b (2 \sigma_s+i b)} \frac{\cosh \left(\pi  b \left(\frac{i b}{2}+\theta +\sigma_s\right)\right)}{2 \sinh (2 \pi  b \sigma_s) \sinh (\pi  b (2 \sigma_s+i b))}, \\
\label{HX0sigmas} & H_{\mathcal{X}}^0(b,\sigma_s) = e^{-\pi  b (\theta + i Q)} - H_{\mathcal{X}}^+(b,\sigma_s)-H_{\mathcal{X}}^+(b,-\sigma_s).
\end{align}
\begin{proposition}
For $(\sigma_s,\omega) \in \mathbb C^2$, the function $\mathcal{X}$ satisfies the following pair of difference equations:
\begin{subequations}\label{diffeqX}\begin{align}
\label{diffeqX1} & H_{\mathcal{X}}(b,\sigma_s) \; \mathcal{X}(b,\theta,\sigma_s,\omega) = e^{-2\pi b \omega} \; \mathcal{X}(b,\theta,\sigma_s,\omega) \\
\label{diffeqX2} & H_{\mathcal{X}}(b^{-1},\sigma_s) \; \mathcal{X}(b,\theta,\sigma_s,\omega) = e^{-2\pi b^{-1} \omega} \; \mathcal{X}(b,\theta,\sigma_s,\omega).
\end{align}\end{subequations}
\end{proposition}

\subsubsection{Second pair of equations}
Introduce the difference operator $\tilde{H}_{\mathcal{X}}(b,\omega)$ such that
\beq\label{HXtilde}
\tilde{H}_{\mathcal{X}}(b,\omega) = \tilde{H}_{\mathcal{X}}^+(b,\omega) e^{ib \partial_\omega} + \tilde{H}_{\mathcal{X}}^-(b,\omega) e^{-ib \partial_\omega} + \tilde{H}_{\mathcal{X}}^0(b,\omega),
\eeq
where
\begin{align*}
& \tilde{H}_{\mathcal{X}}^+(b,\omega) = e^{-2\pi b(\theta+\frac{iQ}2)}, 
\qquad 
\tilde{H}_{\mathcal{X}}^-(b,\omega) = -2 e^{\pi  b \big(\tfrac{i b}{2}+\frac{3 \theta }{2}+\omega \big)} \cosh \left(\pi  b \left(\tfrac{i b}{2}+\tfrac{\theta }{2}-\omega \right)\right), 
	\\
& \tilde{H}_{\mathcal{X}}^0(b,\omega) = -2 \cosh (\pi  b (2 \theta +i b)) - \tilde{H}_{\mathcal{X}}^+(b,\omega) - \tilde{H}_{\mathcal{X}}^-(b,\omega)=e^{\pi  b (\theta+2 \omega )}.
\end{align*}
\begin{proposition}
For $(\sigma_s,\omega) \in \mathbb C^2$, the function $\mathcal{X}$ satisfies the following pair of difference equations:
\begin{subequations}\label{diffeqtildeX}\begin{align}
\label{diffeqtildeX1} & \tilde{H}_{\mathcal{X}}(b,\omega) \; \mathcal{X}(b,\theta,\sigma_s,\omega) = 2\operatorname{cosh}{(2\pi b \sigma_s)} \; \mathcal{X}(b,\theta,\sigma_s,\omega) , \\
\label{diffeqtildeX2} & \tilde{H}_{\mathcal{X}}(b^{-1},\omega) \; \mathcal{X}(b,\theta,\sigma_s,\omega) = 2\operatorname{cosh}{(2\pi b^{-1} \sigma_s)} \; \mathcal{X}(b,\theta,\sigma_s,\omega).
\end{align}\end{subequations}
\end{proposition}

\subsection{Polynomial limit}
Our next theorem shows that $\mathcal{X}$ reduces to the continuous big $q$-Hermite polynomials when $\omega$ is suitably discretized. We make the following assumption, which implies that all the poles of the integrand $I_\mathcal{X}$ are simple and that $q=e^{2i\pi b^2}$ is not a root of unity. 

\begin{assumption} \label{assumptionX}
Assume that $b>0$ is such that $b^2$ is irrational and that
\beq
\re \sigma_s \neq 0, \qquad \re\big(\omega+\tfrac{\theta}2\pm \sigma_s\big) \neq 0.
\eeq
\end{assumption}

The proof of the next theorem is analogous to the proof of Theorem \ref{thhahn} and is omitted.

\begin{theorem}[From $\mathcal{X}$ to the continuous big $q$-Hermite polynomials] \label{thmXtoXn} 
Let $\sigma_s \in \mathbb C \backslash \Delta_{\mathcal X,\sigma_s}$ and suppose that Assumptions \ref{assumption} and \ref{assumptionX} are satisfied. Define  $\{\omega_n\}_{n=0}^\infty \subset \mathbb{C}$ by
\beq \label{omegan}
\omega_n = \frac{\theta }{2}+\frac{iQ}{2}+ibn.
\eeq
Under the parameter correspondence
\beq\label{paramXn}
\alpha_{X} = e^{2\pi  b (\theta +\frac{i Q}2)}, \qquad q=e^{2i\pi b^2},
\eeq
the function $\mathcal{X}$ satisfies, for each $n\geq 0$,
\beq \label{limXtoXn}
\lim\limits_{\omega\to\omega_n} \mathcal{X}(b,\theta,\omega,\sigma_s) = X_n(e^{2\pi b \sigma_s};\alpha_{X},q),
\eeq
where $X_n$ are the continuous big $q$-Hermite polynomials defined in \eqref{Xn}.
\end{theorem}

\section{The function $\mathcal{Q}$}

The function $\mathcal{Q}(b,\sigma_s,\eta)$ is defined as a confluent limit of $\mathcal{X}$ and is one of the two elements at the fifth and lowest level of the non-polynomial scheme. We show that $\mathcal{Q}$ is a joint eigenfunction of four difference operators,  two acting on $\sigma_s$ and two acting on $\eta$. Finally, we show that $\mathcal{Q}$ reduces to the continuous $q$-Hermite polynomials, which lie at the lowest level of the $q$-Askey scheme, when $\eta$ is suitably discretized. 

Interestingly, the mechanism behind the polynomial limit for $\mathcal{Q}$ is different from that of all the other polynomial limits in this paper. In all other cases, the simple pole of the contour integral which compensates for the simple zero in the prefactor arises because the contour of integration is squeezed between two poles of the integrand. However, in the case of $\mathcal{Q}$, the simple pole of the contour integral arises because the integrand loses its decay at infinity in the relevant polynomial limit.

The derivation of the difference equations for $\mathcal{Q}$ also presents some novelties compared to the other derivations of difference equations appearing in this paper. More precisely, the first pair of difference equations for $\mathcal{Q}$ obtained by taking the confluent limit of the difference equations for $\mathcal{X}$ are ``squares'' of the simplest possible difference equations for $\mathcal{Q}$. By finding square roots of the relevant difference operators, we obtain the equations in their simplified form (this is the form that reduces to the standard difference equations for the continuous $q$-Hermite polynomials in the polynomial limit).

\subsection{Definition and integral representation}
Let $\eta$ be a new parameter defined as follows:
\beq \label{defomega}
\omega = \frac{\theta}2+\eta.
\eeq
Moreover,  define the normalization factor $K$ by
\beq \label{K}
K(\eta,\theta) = e^{2 i \pi  \left(\eta - \frac{i Q}{2}\right) \left(\theta +\frac{i Q}{2}\right)}.
\eeq 
The factor $K$ is a non-polynomial analog of the factor $\alpha^{-n}$ in \eqref{Qn}.

\begin{definition}
The function $\mathcal{Q}$ is defined by 
\beq\label{fromXtoQ}
\mathcal{Q}(b,\sigma_s,\eta) = \lim\limits_{\theta\to -\infty} \lb K(\eta,\theta) \mathcal{X}(b,\theta,\sigma_s,\tfrac{\theta}2+\eta)\rb.
\eeq
\end{definition}
The next theorem shows that $\mathcal Q$ is a well-defined and analytic function of $(\sigma_s,\eta) \in \mathbb C \times (\{\im \eta < Q/2 \} \backslash \Delta_\eta)$, where $\Delta_\eta$ is a discrete set of points at which $\mathcal Q$ may have poles. 
The theorem also provides an integral representation for $\mathcal Q$. Even if the requirement $\im \eta < Q/2$ is needed to ensure that the integral in the representation for $\mathcal{Q}$ converges, it will follow from the difference equations established later that $\mathcal Q$ extends to a meromorphic function of $\eta \in \mathbb{C}$. 

\begin{theorem} \label{thmforQ}
Suppose that Assumption \ref{assumption} is satisfied. The limit \eqref{fromXtoQ} exists uniformly for $(\sigma_s,\eta)$ in compact subsets of
\beq
D_\mathcal{Q} := \mathbb C \times (\{ \im \eta <Q/2 \} \backslash \Delta_\eta),
\eeq
where
\beq \label{Detadef}
\Delta_\eta := \{-\tfrac{iQ}2 -imb - ilb^{-1} \}_{m,l=0}^\infty.
\eeq
Moreover, the function $\mathcal{Q}$ admits the following integral representation:
\beq\label{defQ}
\mathcal{Q}(b,\sigma_s,\eta) = P_\mathcal{Q}(\sigma_s,\eta) \displaystyle \int_{\mathsf{Q}} dx\; I_\mathcal{Q}(x,\sigma_s,\eta) \qquad \text{for $(\sigma_s,\eta) \in D_\mathcal{Q}$},
\eeq
where 
\begin{align}
\label{PQ} & P_\mathcal{Q}(\sigma_s,\eta) = e^{i \pi  \left(\frac{1}{6}+\frac{7 Q^2}{24}-\frac{\eta^2}{2}+i \eta  Q- \sigma_s^2\right)} s_b(\eta), \\
\label{IQ} & I_\mathcal{Q}(x,\sigma_s,\eta) = e^{-i\pi x^2} e^{2i\pi x(\eta-\frac{iQ}2)} s_b\left(x+\sigma_s\right) s_b\left(x-\sigma_s\right),
\end{align}
and the contour of integration $\mathsf{Q}$ is any curve from $-\infty$ to $+\infty$ passing above the points $x= \pm \sigma_s - iQ/2$ and with the requirement that its right tail satisfies
\beq\label{Qrighttailcondition}
\im x + \tfrac{Q}4 - \tfrac{ \im \eta}2 < -\delta \qquad
\text{for all $x \in \mathsf{Q}$ with $\re x$ sufficiently large},
\eeq
for some $\delta>0$.  If $(\sigma_s,\eta) \in \mathbb R^2$, then the contour $\mathsf{Q}$ can be any curve from $-\infty$ to $+\infty$ lying within the strip $\im x \in (-Q/2,-Q/4-\delta)$.
\end{theorem}
\begin{proof}
The proof will be omitted since it involves computations which are similar to those presented in the proofs of Theorems \ref{thmforH} and  \ref{thmforS}. However, we point out that before taking the limit $\theta \to -\infty$ of the integral representation \eqref{defX} for $\mathcal{X}$, one should make the change of variables $x \to x - \theta$ and thus write
\beq
K(\eta, \theta) P_\mathcal{X}(\sigma_s,\tfrac{\theta}2+\eta) I_\mathcal{X}(x-\theta,\sigma_s,\tfrac{\theta}2+\eta) = P_\mathcal{Q}(\sigma_s,\eta) Z_\mathcal{Q}(x, \theta) I_\mathcal{Q}(x,\sigma_s,\eta),
\eeq
where
\beq
Z_\mathcal{Q}(x, \theta) := e^{\frac{i \pi}{6} (3 \eta ^2+\frac{5 Q^2}{4}+6 \sigma_s^2-1)} e^{\frac{\pi Q}{2}(x-\theta )} e^{-i \pi  \eta  x} e^{i \pi  \theta  (\eta +\theta )} \frac{s_b(x-\eta -\theta )s_b(\sigma_s-\theta )s_b(-\sigma_s-\theta) }{s_b(x-\theta +\frac{i Q}{2})}
\eeq
satisfies $Z_\mathcal{Q}(x, \theta) \to 1$ as $\theta \to -\infty$.
\end{proof}

Thanks to \eqref{sbinverse}, we have
\beq \label{Qinverse}
\mathcal{Q}(b,\sigma_s,\eta) = \mathcal{Q}(b^{-1},\sigma_s,\eta).
\eeq

\begin{remark}
A close relative of the function $\mathcal{Q}$ has appeared in \cite{KLS2002,R2011} in the context of quantum relativistic Toda systems. More precisely, the function $\mathcal{H}(a_-,a_+,x,y)$ in \cite[Eq. (5.56)]{R2011} is related to $\mathcal{Q}$ by
\beq\label{QequalH}
\mathcal{Q}(b,\sigma_s,\eta) = \; e^{\frac{1}{4} i \pi  \big(\frac{Q^2}{4}+\frac{1}{2}-\eta ^2-8 \sigma_s^2\big)} g_b(\eta) \; \mathcal{H}(b,b^{-1},-\eta,2\sigma_s),
\eeq
where $g_b(z)$ satisfies $s_b(z)=g_b(z)/g_b(-z)$ and is defined by
\beq \label{gb}
g_b(z)=\operatorname{exp}{\left\{ \int_0^\infty \frac{dt}{t}\left[\frac{e^{2i z t}-1}{4 \operatorname{sinh}{b t} \operatorname{sinh}{b^{-1} t}}+\frac{1}{4}z^2 \lb e^{-2bt}+e^{-\frac{2t}b}\rb-\frac{iz}{2t} \right]\right\}}, \qquad \im z > -\frac{Q}2.
\eeq
\end{remark}

\subsection{Difference equations}
We show that $\mathcal{Q}(b,\sigma_s,\eta)$ is a joint eigenfunction of four difference operators, two acting on $\sigma_s$ and the remaining two on $\eta$.

We know from Theorem \ref{thmforQ} that $\mathcal Q$ is a well-defined analytic function of $\sigma_s \in \mathbb C$ and a meromorphic function of $\eta$ for $\im \eta <Q/2$. The difference equations will first be derived as equalities between meromorphic functions defined on this limited domain. However, the difference equations in $\eta$ can then be used to extend the results to $\eta \in \mathbb{C}$, see Proposition \ref{Qextensionprop}.

\subsubsection{First pair of equations}
Define the difference operator $H_\mathcal{Q}(b,\sigma_s)$ by
\beq \label{HQ}
H_\mathcal{Q}(b,\sigma_s) = H^+_\mathcal{Q}(b,\sigma_s) e^{ib \partial_{\sigma_s}} + H^+_\mathcal{Q}(b,-\sigma_s) e^{-ib \partial_{\sigma_s}} + H^0_\mathcal{Q}(b,\sigma_s),
\eeq
where
\begin{align}
 \label{HQplus} & H^+_\mathcal{Q}(b,\sigma_s) = -\frac{e^{-2 \pi  b (2 \sigma_s+i b)}}{4 \sinh (2 \pi  b \sigma_s) \sinh (\pi  b (2 \sigma_s+i b))}, \\
\label{HQ0} & H^0_\mathcal{Q}(b,\sigma_s) = e^{-i\pi b Q} - H^+_\mathcal{Q}(b,\sigma_s) - H^+_\mathcal{Q}(b,-\sigma_s).
\end{align}
\begin{proposition}
For $\sigma_s \in \mathbb C$ and $\im \eta < Q/2$,  the function $\mathcal{Q}$ satisfies the following pair of difference equations:
\begin{subequations} \label{diffeqQ}\begin{align}
\label{diffeqQ1} & H_\mathcal{Q}(b,\sigma_s) \; \mathcal{Q}(b,\sigma_s,\eta) = e^{-2\pi b \eta} \; \mathcal{Q}(b,\sigma_s,\eta),  \\
\label{diffeqQ2} & H_\mathcal{Q}(b^{-1},\sigma_s) \; \mathcal{Q}(b,\sigma_s,\eta) = e^{-2\pi b^{-1} \eta} \; \mathcal{Q}(b,\sigma_s,\eta).
\end{align}\end{subequations}
\end{proposition}
\begin{proof}
It is straightforward to verify that the following limits hold:
\beq
\lim\limits_{\theta\to -\infty} e^{\pi b \theta} H^+_\mathcal{X}(b,\pm\sigma_s) = H^+_\mathcal{Q}(b,\pm\sigma_s), \qquad \lim\limits_{\theta\to -\infty} e^{\pi b \theta} H^0_\mathcal{X}(b,\sigma_s) = H^0_\mathcal{Q}(b,\sigma_s),
\eeq
where $H^+_\mathcal{X}$ and $H^0_\mathcal{X}$ are defined in  \eqref{HXplussigmas} and \eqref{HX0sigmas}, respectively. It follows that
\beq \label{fromHXtoHQ}
\lim\limits_{\theta\to -\infty} e^{\pi b \theta} H_\mathcal{X}(b,\sigma_s) = H_\mathcal{Q}(b,\sigma_s).
\eeq
where $H_\mathcal{X}$ is defined in \eqref{HX}. The difference equation \eqref{diffeqQ1} follows after taking the confluent limit \eqref{fromXtoQ} of the difference equation \eqref{diffeqX1} and utilizing \eqref{fromHXtoHQ}.
\end{proof}

In what follows, we show that the function $\mathcal{Q}$ satisfies a pair of difference equations which is more fundamental than \eqref{diffeqQ}.

\begin{proposition}
Define the difference operator $\hat{H}_\mathcal{Q}(b,\sigma_s)$ by
\beq\label{HhatQ}
\hat{H}_\mathcal{Q}(b,\sigma_s) = \hat{H}^+_\mathcal{Q}(b,\sigma_s) e^{\frac{ib}2\partial_{\sigma_s}} + \hat{H}^-_\mathcal{Q}(b,\sigma_s) e^{-\frac{ib}2\partial_{\sigma_s}},
\eeq
with
\beq
\hat{H}^+_\mathcal{Q}(b,\sigma_s)=-\frac{e^{-\pi b \left(2\sigma_s+\frac{i Q}{2}\right)}}{2 \sinh (2\pi  b \sigma_s)}, \qquad \hat{H}^-_\mathcal{Q}(b,\sigma_s)=\hat{H}^+_\mathcal{Q}(b,-\sigma_s).
\eeq
The following operator identity holds:
\beq \label{squareHhatQ}
\lb \hat{H}_\mathcal{Q}(b,\sigma_s)\rb^2 = H_\mathcal{Q}(b,\sigma_s),
\eeq
where $H_\mathcal{Q}$ is defined in \eqref{HQ}.
\end{proposition}
\begin{proof}
The left-hand side of \eqref{squareHhatQ} can be written as
\begin{align} \label{squareHhatQexpand}
\nonumber \lb \hat{H}_\mathcal{Q}(b,\sigma_s)\rb^2 = &\; \hat{H}^+_\mathcal{Q}(b,\sigma_s) \hat{H}^+_\mathcal{Q}(b,\sigma_s+\tfrac{ib}2) e^{ib\partial_{\sigma_s}} 
+ \hat{H}^-_\mathcal{Q}(b, \sigma_s) \hat{H}^-_\mathcal{Q}(b,\sigma_s-\tfrac{ib}2) e^{-ib\partial_{\sigma_s}} 
	\\
& +  \hat{H}^+_\mathcal{Q}(b,\sigma_s) \hat{H}^-_\mathcal{Q}(b,\sigma_s+\tfrac{ib}2) + \hat{H}^-_\mathcal{Q}(b,\sigma_s) \hat{H}^+_\mathcal{Q}(b,\sigma_s-\tfrac{ib}2) 
\end{align}
and straightforward computations show that the right-hand side coincides with the operator $H_\mathcal{Q}(b,\sigma_s)$.
\end{proof}

We next show that the difference equations \eqref{diffeqQ} satisfied by the function $\mathcal{Q}$ can be simplified using the identity \eqref{squareHhatQ}. The simplified equations can be viewed as ``square roots'' of the equations in \eqref{diffeqQ}.

\begin{proposition}\label{diffeqQrootprop}
For $\sigma_s \in \mathbb C$ and $\im \eta < Q/2$,  the function $\mathcal{Q}$ satisfies
\begin{subequations} \label{diffeqQroot} \begin{align}
\label{diffeqQroot1} & \hat{H}_\mathcal{Q}(b,\sigma_s) \; \mathcal{Q}(b,\sigma_s,\eta) = e^{-\pi b \eta} \; \mathcal{Q}(b,\sigma_s,\eta),  \\
\label{diffeqQroot2} & \hat{H}_\mathcal{Q}(b^{-1},\sigma_s) \; \mathcal{Q}(b,\sigma_s,\eta) = e^{-\pi b^{-1} \eta} \; \mathcal{Q}(b,\sigma_s,\eta).
\end{align}\end{subequations}
\end{proposition}
\begin{proof}
Let us rewrite \eqref{defQ} as 
\beq \label{defQ2}
\mathcal{Q}(b,\sigma_s,\eta) = \displaystyle\int_{\mathsf{Q}} dx \; \psi(x,\sigma_s,\eta) e^{2i\pi x \eta},
\eeq
where
\begin{align}
\psi(x,\sigma_s,\eta) = e^{i \pi  \left(\frac{1}{6}+\frac{7 Q^2}{24}-\frac{\eta ^2}{2}+i \eta  Q-\sigma_s^2\right)} \; s_b(\eta) e^{-i\pi x^2} e^{\pi Q x} s_b\left(x+\sigma_s\right) s_b\left(x-\sigma_s\right),
\end{align}
and where the contour $\mathsf{Q}$ is such that it passes above the points $x= \pm \sigma_s$, and such that $\im x < \tfrac{ \im \eta}2 - Q/4 - \delta$ for some small but fixed $\delta > 0$ as $\re x \to +\infty$. 
The following estimates, which are easily established with the help of \eqref{sbasymptotics}, imply that the integrand has exponential decay on $\mathsf{Q}$ as $\re x \to +\infty$:
\beq \label{estimateIQ}
\psi(x,\sigma_s,\eta) e^{2i\pi x \eta} = \begin{cases} O\lb e^{4\pi (\re x)( \im x+\frac{Q}4-\frac{\im \eta}2)} \rb, & \re x \to +\infty, \\
O\lb e^{2\pi \re x (\frac{Q}{2} - \im \eta)} \rb, & \re x \to -\infty,
 \end{cases}
\eeq
uniformly for $(b,\im x,\sigma_s,\eta)$ in compact subsets of $(0,\infty) \times \mathbb R \times \mathbb{C}^2$. 

Utilizing the difference equation \eqref{differencesb}, we verify that the following identity holds:
\beq \label{identitypsi}
\hat{H}_\mathcal{Q}(b,\sigma_s) \psi(x,\sigma_s,\eta) = \psi(x-\tfrac{ib}2,\sigma_s,\eta).
\eeq
Letting the difference operator $\hat{H}_\mathcal{Q}$ act inside the contour integral in \eqref{defQ2} and using \eqref{identitypsi}, we obtain
\begin{align*}
\hat{H}_\mathcal{Q}(b,\sigma_s) \mathcal{Q}(b,\sigma_s,\eta) = \displaystyle \int_{\mathsf{Q}} dx \; \psi(x-\tfrac{ib}2,\sigma_s,\eta) e^{2i\pi x \eta}.
\end{align*}
Performing the change of variables $x = y + ib/2$, we find
\begin{align*}
\hat{H}_\mathcal{Q}(b,\sigma_s) \mathcal{Q}(b,\sigma_s,\eta) = 
e^{-\pi b \eta}  \int_{\mathsf{Q} - \frac{ib}{2}} dy \; \psi(y,\sigma_s,\eta) e^{2i\pi y \eta}.
\end{align*}
Deforming the contour back to $\mathsf{Q}$, noting that no poles are crossed and that the integrand retains its exponential decay at infinity throughout the deformation, we arrive at \eqref{diffeqQroot1}. The difference equation \eqref{diffeqQroot2} follows from \eqref{diffeqQroot1} and the symmetry \eqref{Qinverse} of $\mathcal{Q}$.
\end{proof}

\subsubsection{Second pair of equations}

Define the difference operator $\tilde{H}_\mathcal{Q}(b,\eta)$ by
\beq \label{HQtilde}
\tilde{H}_\mathcal{Q}(b,\eta) = e^{ib\partial_\eta} + \lb 1+e^{2 \pi  b \left(\eta-\frac{i b}{2}\right)} \rb e^{-ib\partial_\eta}.
\eeq
\begin{proposition}
For $\sigma_s \in \mathbb C$ and $\im(\eta+ib^{\pm 1}) < Q/2$, the function $\mathcal{Q}$ satisfies the pair of difference equations
\begin{subequations} \label{diffeqtildeQ}\begin{align}
\label{diffeqtildeQ1} & \tilde{H}_\mathcal{Q}(b,\eta) \; \mathcal{Q}(b,\sigma_s,\eta) = 2\operatorname{cosh}{(2\pi b \sigma_s)} \; \mathcal{Q}(b,\sigma_s,\eta),  \\
\label{diffeqtildeQ2} & \tilde{H}_\mathcal{Q}(b^{-1},\eta) \; \mathcal{Q}(b,\sigma_s,\eta) = 2\operatorname{cosh}{(2\pi b^{-1} \sigma_s)} \; \mathcal{Q}(b,\sigma_s,\eta).
\end{align}\end{subequations}
\end{proposition}
\begin{proof}
The difference equation \eqref{diffeqtildeX1} can be rewritten as follows:
\beq \label{rewritediffeqXtildeq}
\lb K(\eta,\theta) \tilde{H}_{\mathcal{X}}(b,\tfrac{\theta}2+\eta) K(\eta,\theta)^{-1} \rb K(\eta,\theta) \mathcal{X}(b,\theta,\sigma_s,\tfrac{\theta}2+\eta) = 2\operatorname{cosh}{(2\pi b \sigma_s)} \; K(\eta,\theta) \mathcal{X}(b,\theta,\sigma_s,\tfrac{\theta}2+\eta),
\eeq
where the operator $\tilde{H}_{\mathcal{X}}$ is given in \eqref{HXtilde}. Moreover, we have
\beq
\lim \limits_{\theta\to -\infty} K(\eta,\theta) \tilde{H}_{\mathcal{X}}(b,\tfrac{\theta}2+\eta) K(\eta,\theta)^{-1} = \tilde{H}_\mathcal{Q}(b,\eta).
\eeq
Taking the limit $\theta\to -\infty$ of \eqref{rewritediffeqXtildeq} and recalling the definition \eqref{fromXtoQ}, we obtain \eqref{diffeqtildeQ1}. Finally,  \eqref{diffeqtildeQ2} follows from \eqref{diffeqtildeQ1} and the symmetry \eqref{Qinverse} of $\mathcal{Q}$.
\end{proof}

The next proposition shows that $\mathcal{Q}$ extends to a meromorphic function of $\eta$ everywhere in the complex plane. The proof will be omitted, since it is similar to that of Proposition \ref{Hextensionprop}.

\begin{proposition} \label{Qextensionprop}
Let $b \in (0,\infty)$ and $\sigma_s \in \mathbb C$. Then there is a discrete subset $\Delta \subset \mathbb C$ such that the limit in \eqref{fromXtoQ} exists for all $\eta \in \mathbb C \backslash \Delta$. Moreover, the function $\mathcal Q$ defined by \eqref{fromXtoQ} is a meromorphic function of $(\sigma_s,\eta) \in \mathbb C^2$ and the four difference equations \eqref{diffeqQroot} and \eqref{diffeqtildeQ} hold as equalities between meromorphic functions of $(\sigma_s,\eta) \in \mathbb C^2$.
\end{proposition}

\subsection{Polynomial limit}
In this subsection, we show that the function $\mathcal{Q}$ reduces to the continuous $q$-Hermite polynomials when $\eta$ is suitably discretized.

\begin{theorem}[From $\mathcal{Q}$ to the continuous $q$-Hermite polynomials] \label{thmQtoQn}
Let $\sigma_s \in \mathbb{C}$. Suppose that $b>0$ is such that $b^2$ is irrational. Define $\{\eta_n\}_{n=0}^\infty \subset \mathbb{C}$ by
\beq \eta_n = \tfrac{iQ}2+ibn. \eeq
For each integer $n \geq 0$, the function $\mathcal{Q}$ satisfies
\beq\label{QtoQn}
\lim\limits_{\eta\to \eta_n} \mathcal{Q}(b,\eta,\sigma_s) = Q_n\lb e^{2\pi b \sigma_s};e^{2i\pi b^2}\rb,
\eeq
where $Q_n$ are the continuous $q$-Hermite polynomials defined in \eqref{Qn}.
\end{theorem}

In order to prove Theorem \ref{thmQtoQn}, we will need the following two lemmas.

\begin{lemma}\label{lemmaestimateIQ}
For each $\epsilon >0$, the integrand $I_\mathcal{Q}$ defined in \eqref{IQ} obeys the estimates
\begin{subequations} \label{estimatesIQ}
\begin{align}
\label{estimateIQplus} & \operatorname{ln}{\lb I_\mathcal{Q}(x,\sigma_s,\eta) \rb} = \mathcal{A}^+_\mathcal{Q}(x,\eta) + O\Big( e^{-\frac{2\pi (1-\epsilon)}{\max(b,b^{-1})}|\re{x}|}\Big), \qquad \re{x} \to +\infty,  \\
\label{estimateIQminus} & \operatorname{ln}{\lb I_\mathcal{Q}(x,\sigma_s,\eta) \rb} = \mathcal{A}^-_\mathcal{Q}(x,\eta) + O\Big( e^{-\frac{2\pi (1-\epsilon)}{\max(b,b^{-1})}|\re{x}|}\Big), \qquad \re{x} \to -\infty,
\end{align}\end{subequations}
uniformly for $(b,\im x, \sigma_s,\eta)$ in compact subsets of $(0,\infty) \times \mathbb{R} \times \mathbb{C}^2$, where
\begin{subequations}\begin{align}
\label{AQplus} & \mathcal{A}^+_\mathcal{Q}(x,\eta) = -2i\pi x^2 + 2i\pi x(\eta-\tfrac{iQ}2)-i\pi\lb\sigma_s^2 + \tfrac{Q^2-2}{12}\rb , \\
\label{AQminus} & \mathcal{A}^-_\mathcal{Q}(x,\eta) = 2i\pi x(\eta-\tfrac{iQ}2)+i\pi\lb\sigma_s^2 + \tfrac{Q^2-2}{12}\rb.
 \end{align}\end{subequations}
\end{lemma}
\begin{proof}
The lemma follows from \eqref{sbasymptotics} and \eqref{IQ}. 
\end{proof}

\begin{lemma} \label{lemmaforQ1}
Let
\beq \label{epsilonQ}
\epsilon_\mathcal{Q}(x,\eta) = I_\mathcal{Q}(x,\sigma_s,\eta) - e^{\mathcal{A}_\mathcal{Q}^-(x,\eta)},
\eeq
where $\mathcal{A}_\mathcal{Q}^-$ is defined in \eqref{AQminus}. 
There is a neighborhood $U_0$ of $\eta_0=iQ/2$ and constants $c > 0$ and $M>0$ such that 
\beq\label{epsilonQxetaleq}
\left \lvert \epsilon_\mathcal{Q}(x,\eta) \right \lvert \leq M e^{-c|\re x|} \qquad \text{for $\re x \leq 0$ and for $\eta \in U_0$ with $\im \eta \leq \im \eta_0$},
\eeq
uniformly for $(\im x, \sigma_s)$ in compact subsets of $\mathbb{R} \times \mathbb{C}$.
\end{lemma}
\begin{proof}
Utilizing the estimate \eqref{estimateIQminus} with $\epsilon = 1/2$, we find that
\beq
\epsilon_\mathcal{Q}(x,\eta) = e^{\mathcal{A}_\mathcal{Q}^-(x,\eta)} \Big( e^{O\lb \text{exp}(\frac{\pi \re x}{\max(b,b^{-1})}) \rb} - 1 \Big)
= O\lb e^{\mathcal{A}_\mathcal{Q}^-(x,\eta)} e^{-\frac{\pi |\re x|}{\max(b,b^{-1})}} \rb, \qquad \re x \to -\infty,
\eeq 
uniformly for $(b,\im x,\sigma_s,\eta)$ in compact subsets of $(0,\infty) \times \mathbb{R}\times \mathbb{C}^2$. Since
\beq\label{ereAminusestimate}
e^{\re \mathcal{A}_\mathcal{Q}^-(x,\eta)} = e^{2\pi (\re x)(\frac{Q}2-\im \eta)} e^{-2\pi (\im x) (\re \eta)}e^{-2\pi (\im \sigma_s) (\re \sigma_s)},
\eeq
the desired conclusion follows. 
\end{proof}

\begin{proof}[Proof of Theorem \ref{thmQtoQn}]
Suppose $b >0$ and $\sigma_s \in \mathbb{C}$. 
The function $I_{\mathcal{Q}}(x,\sigma_s,\eta)$ has two downward sequences of poles starting at $x = \pm \sigma_s - iQ/2$. 
Consider the representation \eqref{defQ} for $\mathcal{Q}$ with the contour $\mathsf{Q}$ passing above the points $x= \pm \sigma_s - iQ/2$ and satisfying \eqref{Qrighttailcondition} on its right tail.

Taking $\epsilon=1/2$ in the estimate \eqref{estimateIQplus}, we find that there exists a  constant $M_1$ such that
\beq
\left\lvert I_\mathcal{Q}(x,\sigma_s,\eta) \right\lvert \leq M_1  \big| e^{\mathcal{A}_\mathcal{Q}^+(x,\eta)}  \big|
= M_1  e^{\re \mathcal{A}_\mathcal{Q}^+(x,\eta)}, \qquad x \in \mathsf{Q}, \; \re x \geq 0,
\eeq
uniformly for $(b,\sigma_s, \eta)$ in compact subsets of $(0,\infty) \times \mathbb{C}^2$. Using \eqref{Qrighttailcondition} and noting that 
\beq\label{ereAplusestimate}
e^{\re \mathcal{A}^+_\mathcal{Q}(x,\eta)} = e^{4\pi(\re x)\lb \im x+\frac Q4 - \frac{\im \eta}{2} \rb} e^{-2\pi (\im x) (\re \eta)}e^{2\pi (\im \sigma_s)(\re \sigma_s)},
\eeq
we infer the existence of a neighborhood $U_0$ of $\eta_0=iQ/2$ and constants $c > 0$ and $M>0$ such that 
\beq\label{IcalQestimate}
\left \lvert I_\mathcal{Q}(x,\sigma_s,\eta) \right \lvert \leq M e^{-c \re x}
\eeq
for all $x \in \mathsf{Q}$ with $\re x \geq 0$ and for all $\eta \in U_0$.

 Writing $\mathsf{Q} = \mathsf{Q}_+ \cup \mathsf{Q}_-$, where $\mathsf{Q}_+ = \mathsf{Q} \cap \{\re x \geq 0\}$ and $\mathsf{Q}_- = \mathsf{Q} \cap \{\re x \leq 0\}$, we can express \eqref{defQ} as
\beq\label{rewritedefQ}
\mathcal{Q}(b,\sigma_s,\eta) = P_\mathcal{Q}(\sigma_s,\eta) \displaystyle \int_{\mathsf{Q}_+} dx\; I_\mathcal{Q}(x,\sigma_s,\eta) 
+ P_\mathcal{Q}(\sigma_s,\eta) \displaystyle \int_{\mathsf{Q}_-} dx\; I_\mathcal{Q}(x,\sigma_s,\eta).
\eeq
The prefactor $P_\mathcal{Q}$ defined in \eqref{PQ} has a simple zero at $\eta = \eta_0=iQ/2$. 
Hence, by \eqref{IcalQestimate}, the first term on the right-hand side of \eqref{rewritedefQ} vanishes in the limit $\eta \to \eta_0$.
On the other hand, by \eqref{epsilonQ}, 
\beq \label{rewriteQ1}
\displaystyle \int_{\mathsf{Q}_-} dx\; I_\mathcal{Q}(x,\sigma_s,\eta) = \displaystyle \int_{\mathsf{Q}_-} dx\; e^{\mathcal{A}_\mathcal{Q}^-(x,\eta)} + \displaystyle \int_{\mathsf{Q}_-} dx \; \epsilon_\mathcal{Q}(x,\eta). 
\eeq
In view of Lemma \ref{lemmaforQ1}, there exist constants $M_3, M_4 > 0$ such that 
\beq\label{intQminuseA}
\left\lvert \displaystyle \int_{\mathsf{Q}_-} dx \; \epsilon_\mathcal{Q}(x,\eta) \right \lvert \leq M_3 \displaystyle \int_{\mathsf{Q}_-} dx \; e^{-c|\re x|}
\leq M_4
\eeq
for all $\eta$ in a small neighborhood of $\eta_0$ with $\im \eta \leq \im \eta_0$. 

By Proposition \ref{Qextensionprop}, $\mathcal Q$ is a meromorphic function of $\eta \in \mathbb{C}$, and so has at most a pole at $\eta_0$. To prove \eqref{QtoQn}, it is therefore enough to consider the limit $\eta \to \eta_0$ with $\eta$ such that $\im \eta < \im \eta_0 = Q/2$. In the remainder of the proof, we assume $\im \eta < Q/2$. Then, after multiplication by the prefactor $P_\mathcal{Q}$, the second term on the right-hand side of \eqref{rewriteQ1} vanishes in the limit $\eta \to \eta_0$ as a consequence of \eqref{intQminuseA}.
Furthermore, employing \eqref{estimateIQminus} and using that $\im \eta < \im \eta_0$ so that the contribution from $-\infty+ia$ vanishes, we obtain
\beq \label{resultintQ1}
 \int_{\mathsf{Q}_-} dx\; e^{\mathcal{A}_\mathcal{Q}^-(x,\eta)}
 = \int_{-\infty+ia}^{ia} dx\; e^{\mathcal{A}_\mathcal{Q}^-(x,\eta)} = e^{i\pi \lb \sigma_s^2+\frac{Q^2-2}{12}\rb} \frac{e^{-2\pi a(\eta-\eta_0)}}{2i\pi(\eta-\eta_0)}.
\eeq
The right-hand side of \eqref{resultintQ1} has a simple pole at $\eta=\eta_0$.  Therefore, collecting the above conclusions,
\beq
\lim\limits_{\eta\to\eta_0} \mathcal{Q}(b,\sigma_s,\eta) 
= \lim\limits_{\eta\to\eta_0} P_\mathcal{Q}(\sigma_s,\eta) \int_{\mathsf{Q}_-} dx\; e^{\mathcal{A}_\mathcal{Q}^-(x,\eta)}
= \lim\limits_{\eta\to\eta_0} P_\mathcal{Q}(\sigma_s,\eta) e^{i\pi \lb \sigma_s^2+\frac{Q^2-2}{12}\rb} \frac{e^{-2\pi a(\eta-\eta_0)}}{2i\pi(\eta-\eta_0)},
\eeq
where the limits are taken with $\im\eta < \im \eta_0$.
Utilizing \eqref{PQ} and the identity $s_b(z)=s_b(-z)^{-1}$, we obtain
\beq \label{limitQeta0}
\lim\limits_{\eta\to\eta_0} \mathcal{Q}(b,\sigma_s,\eta) = \frac{1}{2i\pi} \lim\limits_{\eta\to\eta_0} \lb \frac{1}{s_b(-\eta)(\eta-\eta_0)}\rb.
\eeq
Setting $z=-\eta$, recalling that $\eta_0=iQ/2$, and using \eqref{ressb}, we find
\beq \label{limQateta0}
\lim\limits_{\eta\to\eta_0} \mathcal{Q}(b,\sigma_s,\eta) = -\frac{1}{2i\pi} \lim\limits_{z\to -\tfrac{iQ}2} \lb \frac{1}{s_b(z)(z+\frac{iQ}2)}\rb = 1,
\eeq
which proves \eqref{QtoQn} for $n=0$. 

To show \eqref{QtoQn} also for $n \geq 1$, we rewrite the difference equation \eqref{diffeqtildeQ1} as follows:
\beq \label{rewritediffeqtildeQ1}
\mathcal{Q}(b,\sigma_s,\eta+ib) = - \lb 1 + e^{2\pi b(\eta-\frac{ib}2)} \rb \mathcal{Q}(b,\sigma_s,\eta-ib) +2\operatorname{cosh}{(2\pi b \sigma_s)} \mathcal{Q}(b,\sigma_s,\eta).
\eeq
Note that $1 + e^{2\pi b(\eta-\frac{ib}2)}$ vanishes for $\eta=\eta_0$. Moreover, the function $\mathcal{Q}(b,\sigma_s,\eta-ib)$ is analytic at $\eta=\eta_0$. Indeed, as $\eta$ approaches $\eta_0$, the contour $\mathsf{Q}$ remains above the two decreasing sequences of poles, and, in view of \eqref{estimatesIQ} (see also \eqref{ereAplusestimate} and \eqref{ereAminusestimate}), the integrand $I_\mathcal{Q}$ retains its exponential decay provided that the right tail of the contour is  deformed downwards. Therefore, evaluating \eqref{rewritediffeqtildeQ1} at $\eta=\eta_0$ and using \eqref{limQateta0}, we obtain $\mathcal{Q}(b,\sigma_s,\eta_1) = 2\operatorname{cosh}{(2\pi b \sigma_s)}$. 
Evaluating the recurrence relation \eqref{recurrenceQn} satisfied by the continuous $q$-Hermite polynomials at $n=0$ and $z=e^{2\pi b \sigma_s}$, we find
\beq
Q_1\lb e^{2\pi b \sigma_s};q \rb = 2\operatorname{cosh}{(2\pi b \sigma_s)} = \mathcal{Q}(b,\sigma_s,\eta_1),
\eeq
which proves \eqref{QtoQn} for $n=1$. More generally, suppose that the function $\mathcal{Q}(b,\sigma_s,\eta_n)$ exists for all $n\leq N$ and coincides with the polynomials  $Q_n\lb e^{2\pi b \sigma_s};q \rb$. Evaluating \eqref{rewritediffeqtildeQ1} at $n=N+1$, we obtain
\beq
\mathcal{Q}(b,\sigma_s,\eta_{N+1}) = - \lb 1 -q^{N+1} \rb Q_N\lb e^{2\pi b \sigma_s};q \rb +2\operatorname{cosh}{(2\pi b \sigma_s)} Q_{N-1}\lb e^{2\pi b \sigma_s};q \rb,
\eeq
and hence the recurrence relation \eqref{recurrenceQn} implies that
\beq
\mathcal{Q}(b,\sigma_s,\eta_{N+1}) = Q_{N+1}\lb e^{2\pi b \sigma_s};q \rb.
\eeq
By induction, we conclude that $\mathcal{Q}(b,\sigma_s,\eta_n)$ exists for all $n\geq 0$ and coincides with $Q_n\lb e^{2\pi b \sigma_s};q \rb$. This completes the proof of the theorem.
\end{proof}

\section{The function $\mathcal{L}$}
In this section, we define the function $\mathcal{L}(b,\theta_t,\theta,\lambda,\mu)$ which generalizes the big $q$-Laguerre polynomials. It is defined as a confluent limit of $\mathcal{H}$ and lies at the third level of the non-polynomial scheme. We show that $\mathcal{L}$ is a joint eigenfunction of four difference operators. Finally, we show that $\mathcal{L}$ reduces to the big $q$-Laguerre polynomials, which lie at the third level of the $q$-Askey scheme, when $\lambda$ is suitably discretized.

\subsection{Definition and integral representation}
Let $\theta, \Lambda, \lambda, \mu$ be defined as follows:
\beq \label{paramforL}
\theta_0 = \frac{\theta+\Lambda}{2}, \qquad \theta_*= \frac{\theta-\Lambda}{2}, \qquad \sigma_s = \lambda + \frac{\Lambda}2, \qquad \nu = \mu-\frac{\Lambda}4.
\eeq
Define the open set $D_\mathcal{L} \subset \mathbb{C}^2$ by
\beq \label{DLdef}
D_\mathcal{L} := \; (\mathbb C \backslash \Delta_\lambda) \times (\{\im \mu>-Q/2\} \backslash \Delta_\mu),
\eeq
where the discrete subsets $\Delta_\lambda$ and $\Delta_\mu$ are given by
\begin{align*} 
\nonumber \Delta_\lambda := & \; \{\tfrac{iQ}2 \pm \theta_t + \tfrac{\theta}2+ibm+ilb^{-1} \}_{m,l=0}^\infty  \cup \{\tfrac{iQ}2 - \tfrac{\theta}2+ibm+ilb^{-1} \}_{m,l=0}^\infty
	\\
 &  \cup \{-\tfrac{iQ}2+\theta_t + \tfrac{\theta}2-ibm-ilb^{-1} \}_{m,l=0}^\infty,
	\\\nonumber 
\Delta_\mu := &\; \{\tfrac{iQ}2 \pm \theta_t + \tfrac{\theta}4+ibm+ilb^{-1} \}_{m,l=0}^\infty  \cup \{\tfrac{iQ}2 - \tfrac{3\theta}4+ibm+ilb^{-1} \}_{m,l=0}^\infty 
	\\
 &  \cup \{-\tfrac{iQ}2+\theta_t + \tfrac{\theta}4-ibm-ilb^{-1} \}_{m,l=0}^\infty.
\end{align*}

\begin{definition}
Let $\mathcal{H}$ be defined by \eqref{fromRtoH}. The function $\mathcal{L}$ is defined for $(\lambda,\mu) \in D_\mathcal{L}$ by 
\beq \label{fromHtoL}
\mathcal{L}(b,\theta_t,\theta,\lambda,\mu) = \lim\limits_{\Lambda \to -\infty} \mathcal{H}\lb b,  \tfrac{\theta+\Lambda}{2}, \theta_t, \tfrac{\theta-\Lambda}{2},  \lambda + \tfrac{\Lambda}2,  \mu-\tfrac{\Lambda}4\rb,
\eeq
and is extended meromorphically to $(\lambda,\mu) \in \mathbb{C}^2$.
\end{definition}

The next theorem shows that $\mathcal{L}$ is a well-defined meromorphic function of $(\lambda,\mu) \in \mathbb{C}^2$.

\begin{theorem} \label{thmforL}
Suppose that Assumption \ref{assumption} is satisfied. The limit in \eqref{fromHtoL} exists uniformly for $(\lambda,\mu)$ in compact subsets of $D_\mathcal{L}$. Moreover, $\mathcal L$ is an analytic function of $(\lambda,\mu) \in (\mathbb C \backslash \Delta_\lambda) \times (\mathbb C \backslash \Delta_\mu)$ and admits the following integral representation:
\beq\label{defL}
\mathcal{L}(b,\theta_t,\theta,\lambda,\mu) = P_\mathcal{L}(\lambda,\mu) \displaystyle \int_{\mathsf{L}} dx \; I_\mathcal{L}(x,\lambda,\mu) \qquad \text{for $(\lambda,\mu) \in (\mathbb C \backslash \Delta_\lambda) \times (\mathbb C \backslash \Delta_\mu)$},
\eeq
where
\begin{align} \label{PL} & P_\mathcal{L}(\lambda,\mu) = s_b\left(2 \theta_t+\tfrac{i Q}{2}\right) s_b\left(\theta +\theta_t+\tfrac{i Q}{2}\right) s_b\left(\lambda-\tfrac{\theta }{2}-\theta_t \right) s_b\left(\mu-\tfrac{\theta }{4}-\theta_t \right), \\
\label{IL} & I_\mathcal{L}(x,\lambda,\mu) = e^{\frac{i\pi x^2}2} e^{i\pi x (\frac{\theta}4+\theta_t+\lambda+\mu-\frac{iQ}2)} \frac{s_b\left(x+\frac{\theta }{2}+\theta_t-\lambda \right) s_b\left(x+\frac{\theta }{4}+\theta_t-\mu \right)}{s_b\left(x+\frac{i Q}{2}\right) s_b\left(x+2 \theta_t+\frac{i Q}{2}\right) s_b\left(x+\theta +\theta_t+\frac{i Q}{2}\right)},
\end{align}
and the contour $\mathsf{L}$ is any curve from $-\infty$ to $+\infty$ which separates the three increasing from the two decreasing sequences of poles, with the requirement that its right tail satisfies
\beq
\im x + \tfrac{Q}2 + \im \lambda + \im \mu > \delta \qquad
\text{for all $x \in \mathsf{S}$ with $\re x$ sufficiently large},
\eeq
for some $\delta >0$. In particular, $\mathcal L$ is a meromorphic function of $(\lambda,\mu) \in \mathbb C^2$.  If $(\lambda,\mu) \in \mathbb R^2$, the contour $\mathsf L$ can be any curve from $-\infty$ to $+\infty$ lying within the strip $\im x \in (-Q/2 +\delta,0)$. 
\end{theorem}

Furthermore, the function $\mathcal{L}$ obeys the symmetry
\beq \label{Linverse}
\mathcal{L}(b,\theta_t,\theta,\lambda,\mu) = \mathcal{L}(b^{-1},\theta_t,\theta,\lambda,\mu).
\eeq

\subsection{Difference equations} 
The next two propositions, whose proofs are omitted because they are similar to those presented in Section \ref{Hdifferencesubsec}, show that the two pairs of difference equations \eqref{diffeqH} and \eqref{diffeqtildeH} satisfied by the function $\mathcal{H}$ survive in the confluent limit \eqref{fromHtoL}, implying that $\mathcal{L}(b,\theta_t,\theta,\lambda,\mu)$ is a joint eigenfunction of four difference operators, two acting on $\lambda$ and the other two on $\mu$. 
The four difference equations hold as equalities between meromorphic functions of $(\lambda, \mu) \in \mathbb C^2$.

\subsubsection{First pair of difference equations}

Define a difference operator $H_\mathcal{L}(b,\lambda)$ such that
\beq \label{HL}
H_\mathcal{L}(b,\lambda) = H^+_\mathcal{L}(b,\lambda) e^{ib \partial_\lambda} + H^-_\mathcal{L}(b,\lambda) e^{-ib \partial_\lambda} + H^0_\mathcal{L}(b,\lambda),
\eeq
where
\begin{align}
\label{HLplus} & H^+_\mathcal{L}(b,\lambda) =  -4 e^{-\pi b \left(\frac{\theta }{2}+\theta_t-2 \lambda \right)} \cosh \left(\pi  b \left(\tfrac{i b}{2}+\tfrac{\theta }{2}+\lambda \right)\right) \cosh \left(\pi  b \left(\tfrac{i b}{2}-\tfrac{\theta }{2}+\theta_t+\lambda \right)\right),
\\
\label{HLminus} & H^-_\mathcal{L}(b,\lambda) = -2 e^{\pi  b \left(\theta_t-\frac{ib}{2}+3 \lambda \right)} \cosh \left(\pi  b \left(\tfrac{i b}{2}+\tfrac{\theta }{2}+\theta_t-\lambda \right)\right),
\end{align}
and
\beq \label{HL0}
H^0_\mathcal{L}(\lambda) = e^{-\pi b \left(\frac{\theta }{2}+2 \theta_t+i Q\right)} - H^+_\mathcal{L}(b,\lambda) - H^-_\mathcal{L}(b,\lambda).
\eeq
\begin{proposition}
For $(\lambda,\mu) \in \mathbb C^2$, the function $\mathcal{L}$ satisfies the other pair of difference equations:
\begin{subequations}\label{diffeqL}\begin{align}
\label{diffeqL1} & H_\mathcal{L}(b,\lambda) \; \mathcal{L}(b,\theta_t,\theta,\lambda,\mu) = e^{-2\pi b \mu} \; \mathcal{L}(b,\theta_t,\theta,\lambda,\mu), \\
\label{diffeqL2} & H_\mathcal{L}(b^{-1},\lambda) \; \mathcal{L}(b,\theta_t,\theta,\lambda,\mu) = e^{-2\pi b^{-1} \mu} \; \mathcal{L}(b,\theta_t,\theta,\lambda,\mu).
\end{align}\end{subequations}
\end{proposition}

\subsubsection{Second pair of difference equations}

Define a difference operator $\tilde{H}_\mathcal{L}(b,\mu)$ such that
\beq \label{HLtilde}
\tilde{H}_\mathcal{L}(b,\mu) = \tilde{H}^+_\mathcal{L}(b,\mu) e^{ib \partial_\mu} + \tilde{H}^-_\mathcal{L}(b,\mu) e^{-ib \partial_\mu} + \tilde{H}^0_\mathcal{L}(b,\mu),
\eeq
where
\begin{align}
\label{HtildeLplus} & \tilde{H}^+_\mathcal{L}(b,\mu) = -4 e^{-\pi b \left(\frac{\theta }{2}+\theta_t-2 \mu \right)} \cosh \left(\pi  b \left(\tfrac{i b}{2}+\tfrac{3 \theta }{4}+\mu \right)\right) \cosh \left(\pi  b \left(\tfrac{i b}{2}-\tfrac{\theta }{4}+\theta_t+\mu \right)\right) , \\
\label{HtildeLminus} & \tilde{H}^-_\mathcal{L}(b,\mu) = -2 e^{\pi  b \left(-\frac{ib}{2}+\frac{\theta }{4}+\theta_t+3 \mu \right)} \cosh \left(\pi  b \left(\tfrac{i b}{2}+\tfrac{\theta }{4}+\theta_t-\mu \right)\right), 
\end{align}
and
\beq \label{HtildeL0}
\tilde{H}^0_\mathcal{L}(b,\mu) = e^{-\pi b (\theta +2 \theta_t+i Q)} - \tilde{H}^+_\mathcal{L}(b,\mu) - \tilde{H}^+_\mathcal{L}(b,-\mu).
\eeq

\begin{proposition}
For $(\lambda,\mu) \in \mathbb C^2$, the function $\mathcal{L}$ satisfies the pair of difference equations
\begin{subequations} \label{diffeqtildeL} \begin{align} 
\label{diffeqtildeL1} & \tilde{H}_\mathcal{L}(b,\mu) \; \mathcal{L}(b,\theta_t,\theta,\lambda,\mu) = e^{-2\pi b \lambda} \; \mathcal{L}(b,\theta_t,\theta,\lambda,\mu), \\
\label{diffeqtildeL2} & \tilde{H}_\mathcal{L}(b^{-1},\mu) \; \mathcal{L}(b,\theta_t,\theta,\lambda,\mu) = e^{-2\pi b^{-1} \lambda} \; \mathcal{L}(b,\theta_t,\theta,\lambda,\mu).
\end{align} \end{subequations}
\end{proposition}

\subsection{Polynomial limit}
Our next theorem, whose proof is similar to that of Theorem \ref{thhahn}, shows that the function $\mathcal{L}$ reduces to the big $q$-Laguerre polynomials $L_n$ when $\lambda$ is suitably discretized. 

The following assumption ensures that all the poles of the integrand  in \eqref{defL} are distinct and simple.

\begin{assumption} \label{assumptionL}
Assume that $b>0$ is such that $b^2$ is irrational and that
\beq
\theta_t \neq 0, \qquad  \theta \pm \theta_t \neq 0, \qquad \re\big(\tfrac{\theta}4 -\lambda + \mu\big) \neq 0.
\eeq
\end{assumption}

\begin{theorem}[From $\mathcal{L}$ to the big $q$-Laguerre polynomials] \label{thmLtoLn} Let $\mu \in \{\im \mu>-Q/2\} \backslash \Delta_\mu$ and suppose that Assumptions \ref{assumption} and \ref{assumptionL} are satisfied. Define $\{\lambda_n\}_{n=0}^\infty \subset \mathbb{C}$ by
\beq
\lambda_n = \frac{\theta}2+\theta_t+\tfrac{i Q}2+ibn.
\eeq
Under the parameter correspondence
\beq\label{paramLn}
\alpha_{L} = e^{4\pi b \theta_t},  \qquad \beta_{L} = e^{2\pi b(\theta+\theta_t)},  \qquad x_{L}= e^{\pi b(iQ+\frac{\theta}2+2\theta_t)} e^{-2\pi b \mu}  , \qquad q=e^{2i\pi b^2},
\eeq
the function $\mathcal{L}$ satisfies, for each $n\geq 0$,
\beq\label{limitLn}
\lim\limits_{\lambda \to \lambda_n} \mathcal{L}(b,\theta_t,\theta,\lambda,\mu)  = L_n(x_{L};\alpha_{L},\beta_{L},q),
\eeq
where $L_n$ are the big $q$-Laguerre polynomials defined in \eqref{Ln}.
\end{theorem}

\section{The function $\mathcal{W}$}
In this section, we define the function $\mathcal{W}(b,\theta_t,\kappa,\omega)$ which generalizes the little $q$-Laguerre polynomials. It is defined as a confluent limit of $\mathcal{L}$ and lies at the fourth level of the non-polynomial scheme. We show that $\mathcal{W}$ is a joint eigenfunction of four difference operators and that it reduces to the little $q$-Laguerre polynomials, which lie at the fourth level of the $q$-Askey scheme, when $\kappa$ is suitably discretized.

\subsection{Definition and integral representation}
Introduce two new parameters $\kappa$ and $\omega$ by
\beq\label{kappaomega}
\lambda = \frac{\theta}2+\kappa, \qquad \mu = -\frac{3\theta}4+\omega.
\eeq

\begin{definition}
The function $\mathcal{W}(b,\theta_t,\kappa,\omega)$ is defined by
\beq\label{fromLtoW}
\mathcal{W}(b,\theta_t,\kappa,\omega) = \lim\limits_{\theta \to +\infty} \mathcal{L}(b,\theta_t,\theta,\tfrac{\theta}2+\kappa, -\tfrac{3\theta}4+\omega),
\eeq
where $\mathcal{L}$ is defined by \eqref{defL}.
\end{definition} 

The next theorem, whose proof is similar to that of Theorem \ref{thmforH} and will be omitted, shows that, for each choice of $(b,\theta_t) \in (0,\infty) \times \mathbb R$, $\mathcal W$ is a well-defined and analytic function of $(\kappa,\omega) \in (\mathbb{C} \backslash \Delta_\kappa) \times \{\im \omega < Q/2 \}$, where $\Delta_\kappa \subset \mathbb C$ is a discrete set of points at which $\mathcal W$ may have poles. In particular, $\mathcal W$ is a meromorphic function of $\kappa \in \mathbb C$ and of $\omega$ for $\im \omega < Q/2$. The theorem also provides an integral representation for $\mathcal W$ for $(\kappa,\omega) \in (\mathbb{C} \backslash \Delta_\kappa) \times \{\im \omega < Q/2 \}$. In fact, even if the requirement $\im \omega < Q/2$ is needed to ensure that the integral representation converges, it will be shown later, with the help of the difference equations, that $\mathcal W$ extends to a meromorphic function of $(\kappa,\omega) \in \mathbb C^2$.

\begin{theorem} \label{thmforW}
Suppose that Assumption \ref{assumption} holds. The limit \eqref{fromLtoW} exists uniformly for $(\kappa,\omega)$ in compact subsets of
\beq \label{DWdef}
D_\mathcal{W} := \; (\mathbb C \backslash \Delta_\kappa) \times \{\im \omega < Q/2 \},
\eeq
where
\beq \label{Dkappadef}
\Delta_\kappa := \{\tfrac{iQ}2 \pm\theta_t+imb+ilb^{-1} \}_{m,l=0}^\infty \cup \{-\tfrac{iQ}2+\theta_t-imb-ilb^{-1} \}_{m,l=0}^\infty.
\eeq
Moreover, $\mathcal{W}$ is an analytic function of $(\kappa,\omega) \in D_\mathcal{W}$ and admits the following integral representation:
\beq\label{defW}
\mathcal{W}(b,\theta_t,\kappa,\omega) = P_\mathcal{W}(\kappa,\omega) \displaystyle \int_{\mathsf{W}} dx \; I_\mathcal{W}(x,\kappa,\omega) \qquad \text{for $(\kappa,\omega) \in D_\mathcal{W}$},
\eeq
where 
\begin{align}
\label{PW} & P_\mathcal{W}(\kappa,\omega) = s_b(\tfrac{iQ}2+2\theta_t) s_b(\kappa-\theta_t), \\
\label{IW} & I_\mathcal{W}(x,\kappa,\omega) = e^{\frac{i \pi  x^2}{2}} e^{i \pi  x (\theta_t+\kappa +2 \omega )} \frac{s_b(x+\theta_t-\kappa)}{s_b\left(x+\frac{i Q}{2}\right) s_b\left(x+2 \theta_t +\frac{i Q}{2}\right)},
\end{align}
and the contour $\mathsf{W}$ is any curve from $-\infty$ to $+\infty$ which separates the decreasing sequence of poles from the two increasing ones, with the requirement that its right tail satisfies
\beq
\im x + \tfrac{Q}2 + \im \kappa + \im \omega > \delta \qquad \text{for all $x \in \mathsf{W}$ with $\re x$ sufficiently large},
\eeq
for some $\delta>0$. If $(\kappa,\omega)\in \mathbb R^2$, then $\mathsf{W}$ can be any curve from $-\infty$ to $+\infty$ lying within the strip $\im x \in (-Q/2+\delta,0)$.
\end{theorem}

Furthermore,  thanks to the symmetry \eqref{sbinverse},  $\mathcal W$ satisfies
\beq \label{Winverse}
\mathcal{W}(b^{-1},\theta_t,\kappa,\omega) = \mathcal{W}(b,\theta_t,\kappa,\omega).
\eeq

\subsection{Difference equations}
By taking the confluent limit \eqref{fromLtoW} of the difference equations \eqref{diffeqL} and \eqref{diffeqtildeL} satisfied by $\mathcal{L}$, we obtain the next two propositions which show that $\mathcal{W}$ is a joint eigenfunction of four difference operators, two acting on $\kappa$ and the other two on $\omega$.

The difference equations will first be derived as equalities between meromorphic functions of $\kappa \in \mathbb C$ and $\omega$ with $\im \omega < Q/2$. We will then use the difference equations in $\omega$ to show that (i) the limit in \eqref{fromLtoW} exists for all $\omega$ in the complex plane away from a discrete subset, (ii) $\mathcal{W}$ is in fact a meromorphic function of $(\kappa,\omega)$ in all of $\mathbb C^2$, and (iii) the four difference equations hold as equalities between meromorphic functions on $\mathbb C^2$, see Proposition \ref{Wextensionprop}.

\subsubsection{First pair of difference equations}

Define the difference operator $H_{\mathcal{W}}(b,\kappa)$ such that
\beq\label{HW}
H_{\mathcal{W}}(b,\kappa) = H^+_{\mathcal{W}}(b,\kappa) e^{ib \partial_\kappa} + H^-_{\mathcal{W}}(b,\kappa) e^{-ib \partial_\kappa} + H^0_{\mathcal{W}}(b,\kappa),
\eeq
where
\begin{align}
\label{HWpm} & H^\pm_{\mathcal{W}}(b,\kappa) = -2 e^{3\pi b \kappa} e^{\pm \pi b(\frac{ib}2-\theta_t)} \cosh \left(\pi  b \left(\tfrac{i b}{2}+\theta_t\pm \kappa \right)\right), \\
\label{HW0} & H^0_{\mathcal{W}}(b,\kappa) = - H^+_{\mathcal{W}}(b,\kappa) - H^-_{\mathcal{W}}(b,\kappa).
\end{align}
\begin{proposition}
For $\kappa \in \mathbb C$ and $\im \omega <Q/2$, the function $\mathcal{W}$ satisfies the following pair of difference equations:
\begin{subequations}\label{diffeqW}\begin{align}
\label{diffeqW1} & H_{\mathcal{W}}(b,\kappa)\; \mathcal{W}(b,\theta_t,\kappa,\omega) = e^{-2\pi b \omega}\; \mathcal{W}(b,\theta_t,\kappa,\omega), \\
\label{diffeqW2} & H_{\mathcal{W}}(b^{-1},\kappa)\; \mathcal{W}(b,\theta_t,\kappa,\omega) = e^{-2\pi b^{-1} \omega}\; \mathcal{W}(b,\theta_t,\kappa,\omega).
\end{align}\end{subequations}
\end{proposition}

\subsubsection{Second pair of difference equations}

Introduce the difference operator $\tilde{H}_{\mathcal{W}}(b,\omega)$ by
\beq\label{HtildeW}
\tilde{H}_{\mathcal{W}}(b,\omega) = \tilde{H}_{\mathcal{W}}^+(b,\omega) e^{ib \partial_\omega} + \tilde{H}_{\mathcal{W}}^-(b,\omega) e^{-ib \partial_\omega} + \tilde{H}_{\mathcal{W}}^0(b,\omega),
\eeq
where
\begin{align}
\label{HtildeWplus} & \tilde{H}_{\mathcal{W}}^+(b,\omega) = -2 e^{\pi  b \left(\omega-\frac{i b}{2}-2 \theta_t \right)} \cosh \left(\pi  b \left(\omega +\tfrac{i b}{2}\right)\right), \\
\label{HtildeWminus} & \tilde{H}_{\mathcal{W}}^-(b,\omega) = -e^{2 \pi  b (\theta_t+\omega )}, \\
\label{HtildeW0} & \tilde{H}_{\mathcal{W}}^0(b,\omega) = 2 e^{2 \pi  b \omega } \cosh (2 \pi  b \theta_t).
\end{align}
\begin{proposition}
For $\kappa \in \mathbb C$ and $\im(\omega+ib^{\pm 1}) <Q/2$, the function $\mathcal{W}$ satisfies the following pair of difference equations:
\begin{subequations}\label{diffeqtildeW}\begin{align}
\label{diffeqtildeW1} & \tilde{H}_{\mathcal{W}}(b,\omega)\;\mathcal{W}(b,\theta_t,\kappa,\omega) = e^{-2\pi b \kappa}\; \mathcal{W}(b,\theta_t,\kappa,\omega), \\
\label{diffeqtildeW2} & \tilde{H}_{\mathcal{W}}(b^{-1},\omega)\;\mathcal{W}(b,\theta_t,\kappa,\omega) = e^{-2\pi b^{-1} \kappa}\;\mathcal{W}(b,\theta_t,\kappa,\omega).
\end{align}\end{subequations}
\end{proposition}

The next proposition is stated without proof since it is similar to that of Proposition \ref{Hextensionprop}.

\begin{proposition} \label{Wextensionprop}
Let $(b,\theta_t) \in (0,\infty) \times \mathbb R$ and $\kappa \in \mathbb C \backslash \Delta_\kappa$. There is a discrete subset $\Delta \subset \mathbb C$ such that the limit in \eqref{fromLtoW} exists for all $\omega \in \mathbb C \backslash \Delta$. Moreover, the function $\mathcal W$ defined by \eqref{fromLtoW} is a meromorphic function of $(\kappa,\omega) \in \mathbb C^2$ and the four difference equations \eqref{diffeqW} and \eqref{diffeqtildeW} hold as equalities between meromorphic functions of $(\kappa,\omega) \in \mathbb C^2$.
\end{proposition}

\subsection{Polynomial limit}
Our next theorem shows that $\mathcal{W}$ reduces to the little $q$-Laguerre polynomials when $\kappa$ is suitably discretized. We omit the proof which is similar to that of Theorem \ref{thhahn}.

\begin{assumption} \label{assumptionW}
Assume that $b>0$ is such that $b^2$ is irrational and 
\beq 
\theta_t \neq 0.
\eeq
\end{assumption}
Assumption \ref{assumptionW} implies that all the poles of the integrand $I_\mathcal{W}$ are distinct and simple.

\begin{theorem}[From $\mathcal{W}$ to the little $q$-Laguerre polynomials] \label{fromWtoWn}
Let $\omega \in \mathbb{C}$ be such that $\im \omega < Q/2$ and suppose that Assumptions \ref{assumption} and \ref{assumptionW} are satisfied. Define $\{\kappa_n\}_{n=0}^\infty \subset \mathbb{C}$ by
\beq
\kappa_n = \theta_t + \tfrac{i Q}2 + ibn.
\eeq
Under the parameter correspondence
\beq\label{paramWn}
\alpha_{W} = e^{4\pi b \theta_t},  \qquad x_{W}=e^{-2 \pi  b \left(\frac{i Q}{2}+\omega\right)},\qquad q=e^{2i\pi b^2},
\eeq
the function $\mathcal{W}$ satisfies, for each $n\geq 0$,
\beq\label{limitWn}
\lim\limits_{\kappa \to \kappa_n} \mathcal{W}(b,\theta_t,\kappa,\omega) = W_n(x_{W};\alpha_{W},q),
\eeq
where $W_n$ are the little $q$-Laguerre polynomials defined in \eqref{Wn}.
\end{theorem}

\section{The function $\mathcal{M}$}\label{Msec}

In this section, we define the function $\mathcal{M}(b,\zeta,\omega)$ which generalizes the little $q$-Laguerre polynomials \eqref{Wn} evaluated at $\alpha=0$. It lies at the fifth and lowest level of the non-polynomial scheme and is defined as a limit of the function $\mathcal{W}$. We show that $\mathcal M$ is a joint eigenfunction of four difference operators, two acting on $\zeta$ and the other two on $\omega$. Finally, we show that $\mathcal M$ reduces to the little $q$-Laguerre polynomials evaluated at $\alpha=0$ when $\zeta$ is suitably discretized.

\subsection{Definition and integral representation}
Introduce a new parameter $\zeta$ such that $\kappa = \theta_t+\zeta$.

\begin{definition}
The function $\mathcal{M}(b,\zeta,\omega)$ is defined by
\beq \label{fromWtoM}
\mathcal{M}(b,\zeta,\omega) = \lim\limits_{\theta_t \to -\infty} \mathcal{W}(b,\theta_t,\theta_t + \zeta,\omega),
\eeq
where the function $\mathcal{W}$ is defined in \eqref{fromLtoW}.
\end{definition}
The next theorem shows that, for each choice of $b \in (0,\infty)$, $\mathcal M$ is a well-defined and analytic function of
\beq \label{zetaomegacondition}
(\zeta,\omega) \in \{\im \omega<Q/2,\; \im(\zeta+\omega) >0 \} \setminus (\Delta_\zeta \times \mathbb C),
\eeq
where $\Delta_\zeta \subset \mathbb C$ is a discrete set of points at which $\mathcal M$ may have poles.  More precisely, $\Delta_\zeta$ is defined by
\beq \label{Dzetadef}
\Delta_{\zeta} := \{\tfrac{iQ}2+ibm+ilb^{-1}\}_{m,l=0}^\infty \cup \{-\tfrac{iQ}2-imb-ilb^{-1} \}_{m,l=0}^\infty.
\eeq
The theorem also provides an integral representation for $\mathcal M$ for $(\zeta,\omega)$ satisfying \eqref{zetaomegacondition}.  The restrictions in \eqref{zetaomegacondition} are needed to ensure convergence of the integral in the integral representation.  Nevertheless, it will be shown later, with the help of the difference equations satisfied by $\mathcal M$, that $\mathcal M$ extends to a meromorphic function of $(\zeta,\omega) \in \mathbb C^2$. 

\begin{theorem} \label{thmforM} 
Suppose that Assumption \ref{assumption} holds. Let $\Delta_\zeta \subset \mathbb C$ be the discrete subset defined in \eqref{Dzetadef}. Then the limit \eqref{fromWtoM} exists uniformly for $(\zeta,\omega)$ in compact subsets of 
\beq \label{DMdef}
D_\mathcal{M} := \{(\zeta, \omega) \in \mathbb C^2 \,|\, \im \omega<Q/2,\; \im(\zeta+\omega) >0 \} \setminus (\Delta_\zeta \times \mathbb C).
\eeq
Moreover, $\mathcal M$ is an analytic function of $(\zeta,\omega) \in D_\mathcal{M}$ and admits the following integral representation:
\beq \label{defM}
\mathcal{M}(b,\zeta,\omega) = P_\mathcal{M}(\zeta,\omega) \displaystyle \int_\mathsf{M} dx \; I_\mathcal{M}(x,\zeta,\omega) \qquad \text{for $(\zeta,\omega) \in D_\mathcal{M}$},
\eeq
where
\begin{align}
\label{PM} & P_\mathcal{M}(\zeta,\omega) = s_b(\zeta ),  \\
\label{IM} & I_\mathcal{M}(x,\zeta,\omega) = e^{i \pi  x \left(\zeta -\frac{i Q}{2}+2 \omega \right)} \frac{s_b(x-\zeta ) }{s_b\big(x+\frac{i Q}{2}\big)},
\end{align}
and the contour $\mathsf{M}$ is any curve from $-\infty$ to $+\infty$ which separates the increasing from the decreasing sequence of poles. If $\im \omega \in (0,Q/2)$ and $\zeta \in \mathbb R$, then the contour $\mathsf M$ can be any curve from $-\infty$ to $+\infty$ lying within the strip $\im x \in (-Q/2,0)$.
\end{theorem}

Thanks to the identity \eqref{sbinverse}, $\mathcal{M}$ satisfies
\beq \label{Minverse}
\mathcal{M}(b^{-1},\zeta,\omega) = \mathcal{M}(b,\zeta,\omega).
\eeq

\subsection{Difference equations}
The two pairs of difference equations \eqref{diffeqW} and \eqref{diffeqtildeW} satisfied by the function $\mathcal{W}$ survive in the confluent limit \eqref{fromWtoM}. This implies that $\mathcal{M}$ is a joint eigenfunction of four difference operators, two acting on $\zeta$ and two acting on $\omega$.

We know from Theorem \ref{thmforM} that $\mathcal M$ is a well-defined holomorphic function of $\omega$ for $\im \omega <Q/2$ and is meromorphic in $\zeta$ for $\im(\zeta+\omega) >0$. The difference equations will first be derived as equalities between meromorphic functions defined on this limited domain and then extended to equalities between meromorphic functions on $\mathbb C^2$, see Proposition \ref{Mextensionprop}.

\subsubsection{First pair of difference equations}

Consider the difference operator $H_\mathcal{M}(b,\omega)$ defined by
\beq \label{HM}
H_\mathcal{M}(b,\zeta) = e^{2\pi b \zeta} \lb 1 - e^{ib\partial_\zeta} \rb.
\eeq
\begin{proposition} \label{diffeqMprop}
For $\im \omega<Q/2$ and $\im(\zeta+\omega)>0$, the function $\mathcal{M}$ satisfies the difference equations
\begin{subequations}\label{diffeqM}\begin{align}
\label{diffeqM1} & H_\mathcal{M}(b,\zeta)\;\mathcal{M}(b,\zeta,\omega) = e^{-2\pi b \omega}\; \mathcal{M}(b,\zeta,\omega), \\
\label{diffeqM2} & H_\mathcal{M}(b^{-1},\zeta)\;\mathcal{M}(b,\zeta,\omega) = e^{-2\pi b^{-1} \omega}\; \mathcal{M}(b,\zeta,\omega).
\end{align}\end{subequations}
\end{proposition}
\begin{proof}
The proof consists of taking the confluent limit \eqref{fromWtoM} of the difference equation \eqref{diffeqW1}. It is straightforward to verify that 
\beq
\lim\limits_{\theta_t \to -\infty} H_\mathcal{W}^+(b,\theta_t+\zeta) = -e^{2\pi b \zeta}, \qquad \lim\limits_{\theta_t \to -\infty} H_\mathcal{W}^-(b,\theta_t+\zeta) = 0, 
\eeq
where $H_\mathcal{W}^\pm$ is defined in \eqref{HWpm}. From \eqref{HW0}, this implies
\beq
\lim\limits_{\theta_t \to -\infty} H_\mathcal{W}^0(b,\theta_t+\zeta) = e^{2\pi b \zeta}.
\eeq
Therefore we obtain
\beq \label{fromHWtoHM}
\lim\limits_{\theta_t \to -\infty} H_\mathcal{W}(b,\theta_t+\zeta) = H_\mathcal{M}(b,\zeta),
\eeq
where $H_\mathcal{M}$ is given in \eqref{HM}. By Theorem \ref{thmforM}, the limit in \eqref{fromWtoM} exists whenever $(\zeta,\omega) \in D_\mathcal{M}$. Thus, the difference equation \eqref{diffeqM1} follows after utilizing \eqref{fromHWtoHM} and the definition \eqref{fromWtoM} of $\mathcal M$. Finally,  \eqref{diffeqM2} follows from \eqref{diffeqM1} and the symmetry \eqref{Minverse} of $\mathcal{M}$.
\end{proof}

\subsubsection{Second pair of difference equations}
Define the dual difference operator $\tilde{H}_\mathcal{M}(b,\omega)$ by
\beq\label{HtildeM}
\tilde{H}_\mathcal{M}(b,\omega) = e^{2\pi b \omega} - 2 e^{\pi b (\omega-\frac{ib}2)} \operatorname{cosh}{(\pi b(\tfrac{ib}2+\omega))} e^{ib \partial_\omega}.
\eeq

\begin{proposition} \label{diffeqMtildeprop}
For $\im(\omega+ib^{\pm 1}) < Q/2$ and $\im(\zeta+\omega)>0$, the function $\mathcal{M}$ satisfies the following pair of difference equations:
\begin{subequations}\label{diffeqtildeM}\begin{align}
\label{diffeqtildeM1} & \tilde{H}_\mathcal{M}(b,\omega)\;\mathcal{M}(b,\zeta,\omega) = e^{-2\pi b \zeta}\; \mathcal{M}(b,\zeta,\omega), \\
\label{diffeqtildeM2} & \tilde{H}_\mathcal{M}(b^{-1},\omega)\;\mathcal{M}(b,\zeta,\omega) = e^{2\pi b^{-1} \zeta}\; \mathcal{M}(b,\zeta,\omega).
\end{align}\end{subequations}
\end{proposition}
\begin{proof}
It is straightforward to show that the following limit holds:
\beq \label{fromHWtildetoHMtilde}
e^{2\pi b \theta_t} \tilde{H}_\mathcal{W}(b,\omega) = \tilde{H}_\mathcal{M}(b,\omega),
\eeq
where $\tilde{H}_\mathcal{W}$ and $\tilde{H}_\mathcal{M}$ are defined by \eqref{HtildeW} and \eqref{HtildeM}, respectively. By Theorem \ref{thmforM}, the limit in \eqref{fromWtoM} exists whenever $(\zeta,\omega) \in D_\mathcal{M}$. Thus, the difference equation \eqref{diffeqtildeM1} follows after multiplying \eqref{diffeqtildeW1} by $e^{2\pi b \theta_t}$ and utilizing \eqref{fromHWtildetoHMtilde} and the definition \eqref{fromWtoM} of $\mathcal M$. Finally, \eqref{diffeqtildeM2} follows from \eqref{diffeqtildeM1} thanks to the symmetry \eqref{Minverse} of $\mathcal M$.
\end{proof}

\begin{proposition} \label{Mextensionprop}
Let $b \in (0,\infty)$. Then there exist discrete subsets $\Delta,\Delta' \subset \mathbb C$ such that the limit in \eqref{fromWtoM} exists for all $(\zeta,\omega) \in (\mathbb C \setminus \Delta) \times (\mathbb C \setminus \Delta')$. Moreover, the function $\mathcal M$ defined by \eqref{defM} is a meromorphic function of $(\zeta,\omega) \in \mathbb C^2$ and the four difference equations \eqref{diffeqM} and \eqref{diffeqtildeM} hold as equalities between meromorphic functions of $(\zeta,\omega) \in \mathbb C^2$.
\end{proposition}
\begin{proof}
The proof utilizes the difference equations in $\zeta$ and $\omega$ and is similar to the proof of Proposition \ref{Hextensionprop}.
\end{proof}

\subsection{Polynomial limit}
We now show that the function $\mathcal{M}$ reduces to the little $q$-Laguerre polynomials with $\alpha=0$ when $\zeta$ is suitably discretized.

\begin{theorem}[From $\mathcal{M}$ to the little $q$-Laguerre polynomials with $\alpha=0$] \label{fromMtoMn}
Assume that $b>0$ is such that $b^2$ is irrational. Suppose $\omega \in \mathbb{C}$ satisfies $-Q/2+\delta < \im \omega < Q/2$ for some $\delta > 0$.
Define $\{\zeta_n\}_{n=0}^\infty \subset \mathbb{C}$ by
\beq
\zeta_n = \frac{iQ}2+ibn.
\eeq
For each $n\geq 0$, the function $\mathcal M$ satisfies
\beq \label{limMtoMn}
\lim\limits_{\zeta\to\zeta_n} \mathcal{M}(b,\zeta,\omega) = W_n(x_{W_n};0,q),
\eeq
where $W_n$ are the little $q$-Laguerre polynomials defined in \eqref{Wn} and where $x_{W_n}, q$ are given in \eqref{paramWn}.
\end{theorem}

\begin{proof}
We prove \eqref{limMtoMn} by computing the limit $\zeta\to \zeta_n$ of the representation \eqref{defM} for each $n\geq 0$. Let $m,l\geq 0$ be integers and define $x_{m,l} \in \mathbb{C}$ by
\beq \label{xml}
x_{m,l} = \zeta-\frac{iQ}2-imb-\frac{il}b.
\eeq
The function $s_b(x-\zeta)$ in \eqref{IM} has a simple pole located at $x=x_{m,l}$ for any integers $m,l \geq 0$. In the limit $\zeta \to \zeta_n$, the pole $x_{n,0}$ collides with the pole of $s_b(x+\tfrac{iQ}2)$ located at $x=0$, and the contour $\mathsf{M}$ is squeezed between the colliding poles. Hence, before taking the limit $\zeta\to\zeta_n$, we deform $\mathsf{M}$ into a contour $\mathsf{M}'$ which passes below $x_{n,0}$, thus picking up residue contributions from all the poles $x=x_{m,l}$ which satisfy $\im x_{m,l} \geq \im x_{n,0}$, i.e., from all the poles $x_{m,l}$ such that $mb+\tfrac{l}b \leq nb$.  This leads to
\begin{align} \label{deformM}
\mathcal{M}(b,\zeta,\omega) = -2i\pi P_\mathcal{M}(\zeta,\omega) \displaystyle \sum_{\begin{matrix} \substack{m,l \geq 0 \\ mb+\tfrac{l}b \leq nb} \end{matrix}}  \underset{x=x_{m,l}}{\text{Res}} \lb I_\mathcal{M}(x,\zeta,\omega) \rb + P_\mathcal{M}(\zeta,\omega) \displaystyle \int_{\mathsf{M}'} dx \; I_\mathcal{M}(x,\zeta,\omega).
\end{align}
Utilizing the generalized difference equation \eqref{differencepochsb} satisfied by the function $s_b$ and the residue formula \eqref{ressb}, straightforward calculations show that
\begin{align} 
\nonumber  -2i\pi \underset{x=x_{m,l}}{\text{Res}} \lb I_\mathcal{M}(x,\zeta,\omega) \rb =&\; e^{-i \pi  \left(\frac{b^2 m^2}{2}+b m (Q+i (\zeta +2 \omega ))+\left(\zeta -\frac{i Q}{2}\right) \left(\frac{i Q}{2}-\zeta-2 \omega \right)\right)} e^{-i\pi b^{-2}(\frac{l^2}{2}+i b l (\zeta-i b m -i Q+2 \omega))} 
	\\\label{resIMml}
& \times \frac{1}{\big(e^{-\frac{2 i l \pi }{b^2}};e^{\frac{2 i \pi }{b^2}}\big){}_l \left(e^{-2 i b^2 m \pi };e^{2 i b^2 \pi }\right){}_m } \frac{1}{s_b\left(\zeta-\frac{i l}{b}-i b m \right)}.
\end{align}
Because of the factor $s_b(\zeta-\frac{i l}{b}-i b m )^{-1}$, we deduce from \eqref{polesb}  that the right-hand side of \eqref{resIMml} has a simple pole at $\zeta=\zeta_n$ if the pair $(m,l)$ satisfies $m \in [0,n]$ and $l=0$, but is regular at $\zeta=\zeta_n$ for all other choices of $m\geq 0$ and $l \geq 0$.  On the other hand, the prefactor $P_\mathcal{M}$ in \eqref{PM} has a simple zero at $\zeta=\zeta_n$. Therefore, in the limit $\zeta\to \zeta_n$ the second term on the right-hand side of \eqref{deformM} vanishes, and the first term is nonzero only if $m \in [0,n]$ and $l=0$.  We conclude that
\beq
\lim\limits_{\zeta\to \zeta_n} \mathcal{M}(b,\zeta,\omega) = \mathcal{M}(b,\zeta_n,\omega) = -2i\pi \lim\limits_{\zeta\to \zeta_n} P_\mathcal{M}(\zeta,\omega)  \displaystyle \sum_{m=0}^n  \underset{x=x_{m,0}}{\text{Res}} \lb I_\mathcal{M}(x,\zeta,\omega) \rb,
\eeq
or, more explicitly,
\beq\label{firstsum}
\mathcal{M}(b,\zeta_n,\omega) = \displaystyle \sum_{m=0}^n e^{-i \pi b^2\lb \frac{m^2}{2}-mn+n^2 \rb}e^{2\pi b\lb \omega(m-n)-\frac{imQ}{4}\rb} \frac{1}{\left(e^{-2 i b^2 m \pi };e^{2 i b^2 \pi }\right){}_m} \frac{s_b(ibn+\frac{iQ}2)}{s_b(ib(n-m)+\frac{iQ}2)}.
\eeq
Using \eqref{differencepochsb}, we obtain
\beq\label{secondsum}
\mathcal{M}(b,\zeta_n,\omega) = \displaystyle \sum_{m=0}^n e^{\pi  b \left(2(m-n) \omega-i b n^2-i m Q\right)} \frac{\left(q^{1+n-m};q\right){}_m}{\left(q^{-m};q\right){}_m}.
\eeq
We now apply the q-Pochhammer identity
\beq \label{qpochidentity}
(\alpha;q)_n = (-\alpha)^n q^{\frac{n(n-1)}2} \lb \alpha^{-1} q^{1-n};q \rb_n
\eeq
with $\alpha=q^{1+n-m}$ and with $\alpha=q^{-m}$ to find 
\beq \label{sum}
\mathcal{M}(b,\zeta_n,\omega) = \displaystyle \sum_{m=0}^n  e^{\pi  b (2 m-n) (2 \omega +i b n)} e^{i \pi  m \left(b^2+2 i b \omega -1\right)} \frac{\big(q^{-n};q\big){}_m}{\left(q;q\right){}_m}.
\eeq
After replacing $m \to n-m$, in the sum and using the parameter correspondence \eqref{paramWn}, we obtain
\beq
\mathcal{M}(b,\zeta_n,\omega) = \displaystyle \sum_{m=0}^n e^{i\pi(2m-n)} e^{-i\pi b^2(2m-n)(n+1)} \frac{\left(q^{-n};q\right){}_{n-m}}{\left(q;q\right){}_{n-m}} (q x_W)^m.
\eeq
Using the q-Pochhammer identity 
\beq 
\frac{(\alpha;q)_{n-m}}{(\beta;q)_{n-m}} = \frac{(\alpha;q)_n}{(\beta;q)_n} \frac{(\beta^{-1} q^{1-n};q)_m}{(\alpha^{-1} q^{1-n};q)_m} \lb \frac{\beta}\alpha \rb^m, \qquad \alpha,\beta \neq 0,  \quad m=0,1,\dots,n,
\eeq
with $\alpha=q^{-n}$ and $\beta=q$, we arrive at
\beq
\mathcal{M}(b,\zeta_n,\omega) = \displaystyle e^{i \pi  n \left(b^2 (n+1)-1\right)} \frac{\left(q^{-n};q\right){}_n}{\left(q;q\right){}_n} \sum_{m=0}^n    \frac{\left(q^{-n};q\right){}_m}{\left(q;q\right){}_m} (q x_W)^m.
\eeq
Finally, utilizing the q-Pochhammer identity \eqref{qpochidentity} with $\alpha=q^{-n}$, we find
\beq
e^{i \pi  n \left(b^2 (n+1)-1\right)} \frac{\left(q^{-n};q\right){}_n}{\left(q;q\right){}_n} = 1.
\eeq
Therefore we conclude that
\beq
\mathcal{M}(b,\zeta_n,\omega) = {}_2\phi_1\left( \left. \begin{split} \begin{matrix} q^{-n}, 0 \\ 0 \end{matrix} \end{split} \right|q;q x_W\right) = W_n(x_{W_n};0,q).
\eeq
This concludes the proof of \eqref{limMtoMn}.
\end{proof}

\section{A simple application: Duality formulas}\label{dualitysec}
As a simple application of the constructions presented above, we show how the non-polynomial scheme can be used to easily obtain various duality formulas which relate  members of the $q$-Askey scheme. 

\subsection{Duality formula for the Askey--Wilson polynomials}
The duality formula \eqref{dualityRn} for the Askey--Wilson polynomials is an easy consequence of Theorem \ref{thmAW} and the self-duality property \eqref{selfdual} of the function $\mathcal{R}$ defined in \eqref{defR}. Indeed, suppose Assumptions \ref{assumption} and \ref{assumptionAW} are satisfied and define $\{\sigma_t^{(n)}\}_{n=0}^\infty \subset \mathbb{C}$  by
\beq
  \sigma_t^{(n)} = \tfrac{iQ}2+\theta_1+\theta_t+ibn.
\eeq
Under the parameter correspondence
$$
\tilde{\alpha}_R = e^{2 \pi  b \left(\frac{i Q}{2}+\theta_0+\theta_t\right)}, \quad \tilde{\beta}_R =e^{2 \pi  b \left(\frac{i Q}{2}+\theta_1-\theta_\infty\right)}, \quad \tilde{\gamma}_R=e^{2 \pi  b \left(\frac{i Q}{2}-\theta_0+\theta_t\right)}, \quad \tilde{\delta}_R=e^{2 \pi  b \left(\frac{i Q}{2}+\theta_1+\theta_\infty\right)}, \quad q=e^{2i\pi b^2},
$$
the function $\mathcal{R}$ satisfies, for each integer $n \geq 0$,
\beq\label{secondlimitAW}
\lim\limits_{\sigma_t \to \sigma_t^{(n)}} \mathcal{R} \left[\substack{\theta_1\;\;\;\theta_t\vspace{0.1cm}\\ \theta_{\infty}\;\;\theta_0};\substack{\sigma_s \vspace{0.15cm} \\  \sigma_t},b\right] 
= R_n(e^{2\pi b \sigma_s};\tilde{\alpha}_R,\tilde{\beta}_R,\tilde{\gamma}_R,\tilde{\delta}_R,q),
\eeq
where $R_n$ are the Askey--Wilson polynomials defined in (\ref{AW}).
Moreover, evaluating $\mathcal{R}$ at $\sigma_s = \sigma_s^{(n)}$ and $\sigma_t = \sigma_t^{(m)}$ with $n,m \in \mathbb{Z}_{\geq 0}$ and using \eqref{limitAW} and \eqref{secondlimitAW}, we obtain
\beq\label{dualityAW}
\mathcal{R} \left[\substack{\theta_1\;\;\;\theta_t\vspace{0.1cm}\\ \theta_{\infty}\;\;\theta_0};\substack{\sigma_s^{(n)} \vspace{0.15cm} \\  \sigma_t^{(m)}},b\right] = R_n(e^{2\pi b \sigma_t^{(m)}};\alpha_R,\beta_R,\gamma_R,\delta_R,q) = R_m(e^{2\pi b \sigma_s^{(n)}};\tilde{\alpha}_R,\tilde{\beta}_R,\tilde{\gamma}_R,\tilde{\delta}_R,q).
\eeq
Employing the symmetry $R_n(z;\alpha,\beta,\gamma,\delta,q) =R_n(z^{-1};\alpha,\beta,\gamma,\delta,q)$ and observing that the parameters in \eqref{paramAW} satisfy
\beq
\alpha_R^2 = \frac{\tilde{\alpha}_R \tilde{\beta}_R \tilde{\gamma}_R \tilde{\delta}_R}{q}, \qquad \beta_R = \frac{\tilde{\alpha}_R \tilde{\beta}_R}{\alpha_R}, \qquad \gamma_R = \frac{\tilde{\alpha}_R \tilde{\gamma}_R}{\alpha_R}, \qquad \delta_R = \frac{\tilde{\alpha}_R \tilde{\delta}_R}{\alpha_R},
\eeq
and that
\beq
e^{-2\pi b\sigma_s^{(n)}} = \alpha_R^{-1} q^{-n}, \qquad e^{-2\pi b\sigma_t^{(m)}} = \tilde{\alpha}_R^{-1} q^{-m},
\eeq
we recover the duality formula \eqref{dualityRn} for the Askey--Wilson polynomials.

\subsection{Duality formula relating the continuous dual $q$-Hahn and the big $q$-Jacobi polynomials}
Suppose that Assumptions \ref{assumption} and \ref{assumptionhahnjacobi} are satisfied.  Evaluating the function $\mathcal{H}$ at $\nu=\nu_n$ and $\sigma_s = \sigma_s^{(m)}$ with $n,m \in \mathbb{Z}_{\geq 0}$ and utilizing the polynomial limits \eqref{limHtoHn} and \eqref{limHtoJn}, we obtain a duality formula relating the polynomials $H_n$ and $J_n$:
\beq
\mathcal{H}(b,\theta_0,\theta_t,\theta_*,\sigma_s^{(m)},\nu_n) = H_n(e^{2\pi b \sigma_s^{(m)}};\alpha_H,\beta_{H},\gamma_{H},q) = J_m(q^{-n};\alpha_J,\beta_J,\gamma_J;q).
\eeq
Employing the symmetry $H_n(z; \alpha,\beta,\gamma,q)=H_n(z^{-1}; \alpha,\beta,\gamma,q)$ and observing that the parameters in \eqref{paramhahn} and \eqref{paramjacobi} are related by
\beq
\alpha_J = \frac{\alpha_H \beta_H}{q}, \qquad \beta_J = \frac{\alpha_H}{\beta_H}, \qquad \gamma_J = \frac{\alpha_H \gamma_H}{q}, \qquad e^{-2\pi b \sigma_s^{(m)}} = \alpha_H^{-1} q^{-m}, 
\eeq
we recognize the duality formula \eqref{dualityHnJn}.

\subsection{Duality formula relating the little $q$-Jacobi and the Al-Salam Chihara polynomials}
Suppose that Assumptions \ref{assumption} and \ref{assumptionS} are satisfied. Evaluating $\mathcal{S}(b,\theta_0,\theta_t,\sigma_s,\rho)$ at $\rho=\rho_n$ and $\sigma_s=\sigma_s^{(m)}$ with $n,m \in \mathbb{Z}_{\geq0}$ and utilizing the limits \eqref{limStoSn} and \eqref{limStoYn}, we obtain a duality formula relating the polynomials $S_n$ and $Y_n$:
\beq
S_n(e^{2\pi b \sigma_s^{(m)}};\alpha_S,\beta_S,q) = Y_m(q^{-n}; \alpha_Y,\beta_Y,q), \qquad n,m=0,1,2,\dots.
\eeq
Employing the symmetry $S_n(z;\alpha_S,\beta_S,q)=S_n(z^{-1};\alpha_S,\beta_S,q)$ and observing that
\beq
\alpha_Y = \frac{\alpha_S \beta_S}{q}, \qquad \beta_Y = \frac{\alpha_S}{\beta_S}, \qquad e^{-2\pi b \sigma_s^{(m)}} = \alpha_S^{-1}q^{-m},
\eeq
we recover the duality formula \eqref{dualitySnYn}.

\section{Perspectives}\label{perspectivessec}

\subsection{Relations to other areas}

We discuss some relations that deserve to be studied in more detail in the future.

\subsubsection{Relation to quantum relativistic integrable systems}

The function $\mathcal{R}$ was introduced in \cite{R1994} in the context of relativistic systems of Calogero-Moser type, and studied in greater detail in \cite{R1999,R2003,R2003bis}. In particular, after a proper change of parameters, the difference operator $H_\mathcal{R}$ defined in \eqref{HR} corresponds to the rank one hyperbolic van Diejen Hamiltonian. Strong coupling (or Toda) limits of $H_\mathcal{R}$ and of its higher rank generalizations were considered in \cite{V95}. In this paper, we have shown that each element in the non-polynomial scheme is a joint eigenfunction of four difference operators, which, by construction, are confluent limits of $H_\mathcal{R}$.  Thus it would be desirable to understand if each confluent limit considered in the present article  corresponds to a Toda-type limit. In fact, a renormalized version of the function $\mathcal{Q}$ (see \eqref{defQ}), which is one of the two elements at the fifth level of the non-polynomial scheme, was studied in \cite{KLS2002,R2011} and was interpreted as the eigenfunction of a $q$-Toda type Hamiltonian. 

\subsubsection{Relation to two-dimensional CFTs}
As mentioned in the introduction, Ruijsenaars' hypergeometric function is essentially equal to the Virasoro fusion kernel \cite{R2021}. The latter plays a key role in the conformal bootstrap approach to Liouville conformal field theory on punctured Riemann spheres. We believe that the various confluent limits considered in the present article correspond to collisions of punctures in this theory.  Therefore, we expect the other elements in the non-polynomial scheme to play a role in Liouville conformal field theory on different Riemann surfaces with punctures and cusps \cite{GT}.

\subsubsection{Relation to quantum groups}

Motivated by the application of quantum groups in two-dimensional conformal field theories, Faddeev introduced in \cite{F1999} the modular double of $U_q(sl_2(\mathbb{R}))$, denoted $U_{q\tilde{q}}(sl_2(\mathbb{R}))$.  Both the functions $\mathcal R$ and $\mathcal Q$ can be obtained from the viewpoint of harmonic analysis on $U_{q\tilde{q}}(sl_2(\mathbb{R}))$ \cite{PT1,PT2,B2006,KLS2002}. More generally, we expect that all the elements of the non-polynomial scheme can be obtained from the viewpoint of harmonic analysis on $U_{q\tilde{q}}(g)$, where $g$ is the Lie algebra of a non-compact Lie group. Examples of such modular doubles were studied in \cite {I2015,I2015bis,I2020}. 

\subsection{Outlook}

It was shown in \cite[Theorem 1.1]{R2003bis} that a renormalized version of the function $\mathcal{R}$ has a hidden $D_4$ symmetry in its four external parameters. It would be interesting to determine what this $D_4$ invariance becomes after taking the various confluent limits presented in this article.

In a certain region of the external parameter space,  Ruijsenaars constructed in \cite{R2003} a unitary Hilbert space transform admitting the function $\mathcal R$ as a kernel. An important project is to handle the Hilbert space theory of the other difference equations studied in this article. This project would be of particular interest, since the theory of linear (analytic) difference equations is much less understood than the theory of linear discrete difference, or linear differential equations. In particular, it is not known if the existence of joint eigenfunctions implies that they can be promoted to kernels of unitary Hilbert space transforms. 

Finally, the Askey-Wilson polynomials $R_n(z;\alpha,\beta,\gamma,\delta,q)$ are symmetric under the exchange $z \to z^{-1}$. The non-symmetric Askey-Wilson polynomials, introduced in \cite{NS2004,S1999}, are Laurent polynomials satisfying a difference/reflection equation. A symmetrization procedure described in \cite{NS2004} can be used to recover the symmetric Askey-Wilson polynomials from their non-symmetric counterparts. More generally, a non-symmetric extension of the $q$-Askey scheme was initiated in \cite{M2014,KM2018}. To the best of our knowledge, non-polynomial generalizations of such Laurent polynomials have not been studied in the literature. It would therefore be of interest to explore whether there exists a non-symmetric version of the non-polynomial scheme presented in this paper.

\appendix

\section{$q$-hypergeometric series}\label{appendixA}
The $q$-hypergeometric series $_{s+1}\phi_s$ is defined by
\begin{align}\label{hypergeometricphidef}
_{s+1}\phi_s\left[ \begin{matrix} a_1, \dots, a_{s+1} \\ b_1, \dots, b_s \end{matrix} ;q,z\right] = \sum_{k=0}^\infty \frac{(a_1, \dots,a_{s+1};q)_k}{(b_1, \dots,b_s,q;q)_k}z^k,
\end{align}
where the $q$-Pochammer symbols $(a;q)_n$ and $(a_1,a_2,\dots,a_m;q)_n$ are given by
\begin{align}\label{qpochhammerdef}
(a;q)_n = \prod_{k=0}^{n-1} (1-a q^k) \quad \text{and} \quad (a_1,a_2,\dots,a_m;q)_n = \prod_{j=1}^m (a_j;q)_n.
\end{align}
Note that the sum in \eqref{hypergeometricphidef} contains only finitely many terms if one of the $a_i$ in the numerator is equal to $q^{-n}$ for some integer $n\geq 1$. Otherwise, the sum converges for $|z|<1$.

\section{The $q$-Askey scheme}\label{AppendixB}

\subsection{Askey--Wilson polynomials}
The Askey--Wilson polynomials $R_n$, defined by
\beq\label{AW}
R_n(z;\alpha,\beta,\gamma,\delta,q) = {}_4 \phi_3\lb \left. \begin{split}\begin{matrix}  q^{-n}, \alpha\beta\gamma\delta q^{n-1}, \alpha z, \alpha z^{-1} \\ \alpha \beta, \alpha\gamma, \alpha\delta \end{matrix} \end{split}\right| q;q \rb,
\eeq
are the most general polynomials of the $q$-Askey scheme. 
The normalization in $\eqref{AW}$ for $R_n$ is related to the standard normalization of \cite[Eq. (14.1.1)]{KLS2010} by
\beq\label{1p2}
p_n\lb\tfrac{z+z^{-1}}2;\alpha,\beta,\gamma,\delta,q\rb = \alpha^{-n} (\alpha\beta,\alpha\gamma,\alpha\delta;q)_n ~ R_n(z;\alpha,\beta,\gamma,\delta,q).
\eeq
Since $p_n(x;\alpha,\beta,\gamma,\delta,q)$ is a polynomial of order $n$ in $x$, $R_n$ is a polynomial of order $n$ in $z + z^{-1}$.
The polynomials $R_n$ satisfy the three-term recurrence relation
\beq\label{recurrenceAW}
(R_{R_n} R_n)(z;\alpha,\beta,\gamma,\delta,q) = (z+z^{-1}) R_n(z;\alpha,\beta,\gamma,\delta,q),
\eeq
where the operator $R_{R_n}$ is given by
\beq\label{Mn}
R_{R_n} = a^{+}_n T_{n+1} + (\alpha+\alpha^{-1}-a^{+}_n-a^{-}_n) + a^{-}_n T_{n-1},
\eeq
with $T_{n \pm 1}p_n(x) := p_{n \pm 1}(x)$ and
\beq \label{bndef}
\begin{split}
& a^{+}_n = \frac{\lb 1-\alpha\beta q^n\rb\lb 1-\alpha\gamma q^n\rb\lb 1-\alpha\delta q^n\rb\lb 1-\alpha\beta\gamma\delta q^{n-1}\rb}{\alpha \lb 1-\alpha\beta\gamma\delta q^{2n-1}\rb\lb 1-\alpha\beta\gamma\delta q^{2n}\rb}, \\
& a^{-}_n = \frac{\alpha\lb 1-q^n\rb\lb 1-\beta \gamma q^{n-1}\rb\lb 1-\beta\delta q^{n-1}\rb\lb 1-\gamma\delta q^{n-1}\rb}{\lb 1-\alpha\beta\gamma\delta q^{2n-2}\rb\lb 1-\alpha\beta\gamma\delta q^{2n-1}\rb}.
\end{split}\eeq
The Askey--Wilson polynomials also possess a symmetry exchanging the parameters $n$ and $z$ \cite{KM2018}. More precisely, define dual parameters $\tilde{\alpha},\tilde{\beta},\tilde{\gamma},\tilde{\delta}$ such that
\beq
\alpha^2 = q^{-1} \tilde{\alpha} \tilde{\beta} \tilde{\gamma} \tilde{\delta}, \qquad \beta = \frac{\tilde{\alpha} \tilde{\beta}}{\alpha}, \qquad \gamma = \frac{\tilde{\alpha} \tilde{\gamma}}{\alpha}, \qquad \delta=\frac{\tilde{\alpha} \tilde{\delta}}{\alpha}.
\eeq
Then, by \cite[Eq. (27)]{KM2018},
\beq \label{dualityRn}
R_n(\alpha^{-1} q^{-m};\alpha,\beta,\gamma,\delta,q) = R_m(\tilde{\alpha}^{-1} q^{-n};\tilde{\alpha},\tilde{\beta},\tilde{\gamma},\tilde{\delta},q),\qquad n,m=0,1,2,\dots.
\eeq
\subsection{Continuous dual $q$-Hahn polynomials} 
The continuous dual $q$-Hahn polynomials $H_n(z;\alpha,\beta,\gamma,q)$ are defined by
\beq\label{qhahn}
H_n(z;\alpha,\beta,\gamma,q) = R_n(z;\alpha,\beta,\gamma,0,q)= {}_3\phi_2\left(\left. \begin{matrix} q^{-n},\alpha z ,\alpha z^{-1} \\ \alpha\beta, \alpha\gamma \end{matrix}\right|q;q\right).
\eeq
They satisfy the three-term recurrence relation 
\beq\label{recurrenceHn}
\lb R_{H_n} H_n\rb(z;\alpha,\beta,\gamma,q)=(z+z^{-1})~H_n(z;\alpha,\beta,\gamma,q),
\eeq
where the operator $R_{H_n}$ is defined by
\beq \label{recurrenceoperatorHn}
R_{H_n} = b^{+}_n T_{n+1}+\lb \alpha+\alpha^{-1}-b^{+}_n-b^{-}_n\rb +b^{-}_n T_{n-1},
\eeq
with 
\beq\label{cndef}
b^{+}_n=\alpha^{-1}(1-\alpha \beta q^n)(1-\alpha \gamma q^n), \qquad 
b^{-}_n=\alpha(1-q^n)(1-\beta \gamma q^{n-1}). 
\eeq
 
\subsection{Big $q$-Jacobi polynomials}
The big $q$-Jacobi polynomials $J_n(x;\alpha,\beta,\gamma;q)$ are defined by
\beq\label{Jn}
J_n(x;\alpha,\beta,\gamma;q) = \lim\limits_{\lambda \to 0} R_n\lb\frac{x}{\lambda};\lambda,\frac{\alpha q}\lambda,\frac{\gamma q}\lambda,\frac{\lambda \beta}\gamma,q\rb
= {}_3\phi_2\left( \left. \begin{split} \begin{matrix} q^{-n},\alpha \beta q^{n+1}, x \\ \alpha q, \gamma q\end{matrix} \end{split} \right|q;q\right).
\eeq
They satisfy the three-term recurrence relation
\beq\label{recurrenceJn}
R_{J_n} J_n(x;\alpha,\beta,\gamma;q)=x J_n(x;\alpha,\beta,\gamma;q),
\eeq
where the operator $R_{J_n}$ is defined by
\beq\label{recurrenceoperatorJn}
R_{J_n} = c^{+}_n T_{n+1} + \lb 1-c^{+}_n-c^{-}_n\rb + c^{-}_n T_{n-1},
\eeq
with
\beq\label{cnpm}
\left\{\begin{split} 
& c^{+}_n = \frac{\lb 1-\alpha q^{n+1}\rb\lb 1-\alpha \beta q^{n+1}\rb\lb 1-\gamma q^{n+1}\rb}{\lb 1-\alpha \beta q^{2n+1}\rb\lb 1-\alpha \beta q^{2n+2}\rb}, \\
& c^{-}_n = -\alpha \gamma q^{n+1} \frac{\lb 1-q^n\rb\lb 1-\alpha \beta \gamma^{-1} q^n\rb\lb 1-\beta q^n\rb}{\lb 1-\alpha \beta q^{2n}\rb\lb 1-\alpha \beta q^{2n+1}\rb}.
\end{split}\right.
\eeq
There exists a duality between the big $q$-Jacobi and the continuous dual $q$-Hahn polynomials which is inherited from the duality \eqref{dualityRn} of the Askey--Wilson polynomials. More precisely, by \cite[Eq. (44)]{KM2018},
\beq \label{dualityHnJn}
H_n\lb \alpha^{-1} q^{-m}; \alpha,\beta,\gamma,q\rb = J_m\lb q^{-n}; q^{-1} \alpha \beta, \alpha \beta^{-1},  q^{-1} \alpha \gamma; q\rb, \qquad n,m=0,1,2,\dots.
\eeq

\subsection{Al-Salam-Chihara polynomials}

The Al-Salam Chihara polynomials $S_n(z;\alpha,\beta,q)$ are defined by
\beq\label{Sn}
S_n(z;\alpha,\beta,q) = H_n(z;\alpha,\beta,0,q) = \frac{\alpha^n}{\lb \alpha\beta;q \rb_n} \lb \alpha z;q \rb_n z^{-n} {}_2\phi_1\left( \left. \begin{split} \begin{matrix} q^{-n}, \beta z^{-1} \\ \alpha^{-1} q^{1-n} z^{-1} \end{matrix} \end{split} \right|q;\alpha^{-1} q z \right),
\eeq
where $H_n$ are the continuous dual $q$-Hahn polynomials defined in \eqref{qhahn}.  In \eqref{Sn}, we use the normalization of \cite[Eq. (59)]{KM2018}.
The polynomials $S_n$ satisfy the three-term recurrence relation
\beq\label{recurrenceSn}
R_{S_n} S_n(z;\alpha,\beta,q) = \lb z+z^{-1}\rb S_n(z;\alpha,\beta,q),
\eeq
where the operator $R_{S_n}$ is defined by
\beq\label{recurrenceoperatorSn}
R_{S_n} = \lb\alpha^{-1}-\beta q^n\rb T_{n+1} +\lb \alpha+\beta \rb q^n+ \alpha\lb 1-q^n \rb T_{n-1}.
\eeq

\subsection{Little $q$-Jacobi polynomials}
The Little q-Jacobi polynomials $p_n(x;\alpha,\beta,q)$ are defined by (see \cite[Eq. (66)]{KM2018})
\beq
p_n(x;\alpha,\beta,q) = {}_2\phi_1\left( \left. \begin{split} \begin{matrix} q^{-n}, \alpha \beta q^{n+1} \\ \alpha q \end{matrix} \end{split} \right|q;qx\right),
\eeq
or, equivalently, by \cite[Eq. (3.38)]{Koo94}
\beq
p_n(x;\alpha,\beta,q) = (-q\beta)^{-n} q^{-\frac{n(n-1)}{2}} \frac{(q\beta;q)_n}{(q \alpha;q)_n} {}_3\phi_2\left( \left. \begin{split} \begin{matrix} q^{-n},  q^{n+1}\alpha \beta, q \beta x \\ q \beta, 0 \end{matrix} \end{split} \right|q;q\right).
\eeq
They arise as limits of the big $q$-Jacobi polynomials in two ways. First, we have (see \cite[Eq. (68)]{KM2018})
\beq\label{limJnpn}
p_n(x;\alpha,\beta,q) = \lim\limits_{\gamma\to\infty} J_n(\gamma q x;\alpha,\beta,\gamma;q).
\eeq
Second, we have
\beq\label{limJnYn}
p_n(x;\alpha,\beta,q) = (-q \beta)^{-n} q^{-\frac{n(n-1)}2} \frac{(q\beta;q)_n}{(q \alpha;q)_n} Y_n(q \beta x; \beta, \alpha, q),
\eeq
where
\beq\label{Yn}
Y_n(x;\alpha, \beta, q) = \lim\limits_{\gamma \to 0} J_n(x;\alpha,\beta,\gamma,q) = {}_3\phi_2\left( \left. \begin{split} \begin{matrix} q^{-n},  q^{n+1}\alpha \beta, x \\ \alpha q , 0 \end{matrix} \end{split} \right|q;q\right).
\eeq
The polynomials $Y_n$ satisfy the three-term recurrence relation (see \cite[Eqs. (73)--(74)]{KM2018})
\beq\label{recurrenceYn}
R_{Y_n} Y_n(x;\alpha,\beta,q) = x Y_n(x;\alpha,\beta,q),
\eeq
where the operator $R_{Y_n}$ is defined by
\beq\label{recurrenceoperatorYn}
R_{Y_n} = y^+_n T_{n+1} + 1-y^+_n-y^-_n + y^-_n T_{n-1},
\eeq
with
\begin{align}\label{defyn}
& y^+_n = \frac{\lb 1-\alpha q^{n+1}\rb \lb 1-\alpha \beta q^{n+1} \rb}{\lb 1-\alpha \beta q^{2n+1} \rb \lb 1-\alpha \beta q^{2n+2} \rb}, \qquad  y^-_n =  q^{2n+1} \alpha^2 \beta \frac{\lb 1-q^n \rb \lb 1-\beta q^n \rb}{\lb 1-\alpha \beta q^{2n} \rb \lb 1-\alpha \beta q^{2n+1} \rb}.
\end{align}
Finally, there exists a duality between the little $q$-Jacobi polynomials and the Al-Salam Chihara polynomials which is inherited from \eqref{dualityHnJn}. More precisely, from \cite[Eq. (75)]{KM2018} we have
\beq \label{dualitySnYn}
S_n\lb \alpha^{-1} q^{-m}; \alpha,\beta,q \rb = Y_m \lb q^{-n}; q^{-n} \alpha \beta, \alpha \beta^{-1},  q\rb, \qquad n,m=0,1,2,\dots.
\eeq
\subsection{Big $q$-Laguerre polynomials}

The big $q$-Laguerre polynomials $L_n(x;\alpha,\beta;q)$ are defined by setting $\beta=0$ in the big $q$-Jacobi polynomials (see \cite[Eq. (14.11.1)]{KLS2010}):
\beq\label{Ln}
L_n(x;\alpha,\beta;q)=J_n(x;\alpha,0,\beta;q) =\frac{1}{\lb \beta^{-1} q^{-n};q \rb_n}{}_2\phi_1\left( \left. \begin{split} \begin{matrix} q^{-n}, \alpha q x^{-1} \\ \alpha q \end{matrix} \end{split} \right|q;x \beta^{-1}\right).
\eeq
The polynomials $L_n$ satisfy the three-term recurrence relation
\beq\label{recurrenceLn}
R_{L_n} L_n(x;\alpha,\beta,q) = x \;  L_n(x;\alpha,\beta,q),
\eeq
where the operator $R_{L_n}$ is defined by
\beq\label{recurrenceoperatorLn}
R_{L_n} = l^+_n T_{n+1} +(1-l^+_n-l^-_n) + l^-_n T_{n-1},
\eeq
with
\begin{align}
& l_n^+ = \lb 1-\alpha q^{n+1} \rb \lb 1-\beta q^{n+1} \rb, \qquad l_n^- = -\alpha \beta q^{n+1} \lb 1-q^n\rb.
\end{align}

\subsection{Little $q$-Laguerre polynomials}
The little $q$-Laguerre (or Wall) polynomials $W_n(x;\alpha,q)$ are obtained from the big $q$-Laguerre polynomials as follows:
\beq\label{Wn}
W_n(x;\alpha,q) = \lim\limits_{\beta \to -\infty} L_n(\beta q x;\alpha,\beta;q) = {}_2\phi_1\left( \left. \begin{split} \begin{matrix} q^{-n}, 0 \\ \alpha q \end{matrix} \end{split} \right|q;q x\right).
\eeq
They satisfy the recurrence relation
\beq\label{recurrenceWn}
R_{W_n} W_n(x;\alpha,q) = -x W_n(x;\alpha,q),
\eeq
where the operator $R_{W_n}$ is defined by
\beq \label{recurrenceoperatorWn}
R_{W_n} = w^+_n T_{n+1} -(w^+_n+w^-_n) + w^-_n T_{n-1},
\eeq
with
\begin{align}
& w_n^+ = q^n \lb 1-\alpha q^{n+1} \rb , \qquad w_n^- = \alpha q^n \lb 1-q^n \rb.
\end{align}
\subsection{Continuous big $q$-Hermite polynomials}
The continuous big $q$-Hermite polynomials $X_n(z;\alpha,q)$ are defined by 
\beq\label{Xn}
X_n(z;\alpha,q)=\lim\limits_{\beta \to 0} S_n(z;\alpha,\beta,q) = \alpha^n z^n {}_2\phi_0\left( \left. \begin{split} \begin{matrix} q^{-n}, \alpha z \\ - \end{matrix} \end{split} \right|q;q^n z^{-2} \right),
\eeq
where $S_n$ are the Al-Salam-Chihara polynomials defined in \eqref{Sn}.  The normalization in \eqref{Xn} is obtained by multiplying $H_n(\tfrac12(z+z^{-1}),\alpha|q)$ in \cite[Eq. (14.18.1)]{KLS2010}  by $\alpha^n$. They satisfy the recurrence relation
\beq\label{recurrenceXn}
R_{X_n} X_n(z;\alpha,q) = \lb z+z^{-1}\rb X_n(z;\alpha,q),
\eeq
where the operator $R_{X_n}$ is defined by
\beq \label{recurrenceoperatorXn}
R_{X_n} = \alpha^{-1} T_{n+1} + \alpha q^n + \alpha \lb 1-q^n \rb T_{n-1}.
\eeq

\subsection{Continuous $q$-Hermite polynomials}
The continuous $q$-Hermite polynomials $Q_n(z;q)$ are defined by
\beq\label{Qn}
Q_n(z;q) = \lim\limits_{\alpha \to 0} \alpha^{-n} X_n(z;\alpha,q)= z^n {}_2\phi_0\left( \left. \begin{split} \begin{matrix} q^{-n}, 0 \\ - \end{matrix} \end{split} \right|q;q^n z^{-2} \right).
\eeq
They satisfy the recurrence relation
\beq\label{recurrenceQn}
R_{Q_n} Q_n(z;q) = \lb z+z^{-1}\rb Q_n(z;q),
\eeq
where the operator $Q_n$ is defined by
\beq
R_{Q_n} = T_{n+1}+\lb 1-q^n \rb T_{n-1}.
\eeq


\begin{thebibliography}{10}

\bibitem{AAK} M.N. Atakishiyev, N.M. Atakishiyev, and A.U.  Klimyk, Big $q$-Laguerre and $q$-Meixner polynomials and representations of the quantum algebra $U_q(su_{1,1})$, \href{https://iopscience.iop.org/article/10.1088/0305-4470/36/41/006}{J. Phys. A: Math. Gen. \textbf{36} 10335}.

\bibitem{AW1985} R. Askey and J. Wilson,  Some basic hypergeometric orthogonal polynomials that generalize Jacobi polynomials, \href{http://www.ams.org/books/memo/0319/}{Mem. Amer. Math. Soc. 54 (1985)}. 

\bibitem{BMZ} P. Baseilhac, X. Martin, and A.S. Zhedanov, Little and big $q$-Jacobi polynomials and the Askey–Wilson algebra, \href{https://doi.org/10.1007/s11139-018-0080-1}{The Ramanujan Journal, \textbf{51}, 629–648(2020)}.

\bibitem{B2006} F.J. van de Bult, Ruijsenaars' hypergeometric function and the modular double of $U_q(sl(2,C))$, \href{https://doi.org/10.1016/j.aim.2005.05.023}{Advances in Mathematics 204 (2006): 53971}.

\bibitem{BRS2007} F.J. van de Bult, E.M. Rains, and J.V. Stokman, Properties of generalized univariate hypergeometric functions, \href{https://link.springer.com/article/10.1007/s00220-007-0289-0}{Commun. Math. Phys. \textbf{275}, 3795 (2007)}.

\bibitem{CBMP} S. Collier, Y. Gobeil, H. Maxfield, and E. Perlmutter, Quantum Regge trajectories and the Virasoro analytic bootstrap, \href{https://doi.org/10.1007/JHEP05(2019)212}{J. High Energ. Phys. \textbf{2019}, 212 (2019)}.

\bibitem{CK} S. Corteel and L.K. Williams, Tableaux combinatorics for the asymmetric exclusion process and Askey-Wilson polynomials, Duke Mathematical Journal, \textbf{159}(3) 385-415.

\bibitem{V95} J.F.  van Diejen, Difference Calogero-Moser systems and finite Toda chains, \href{https://doi.org/10.1063/1.531122}{J. Math. Phys.  \textbf{36}, 1299 (1995)}.

\bibitem{F1999} L. Faddeev, Modular Double of a Quantum Group, \href{https://arxiv.org/abs/math/9912078}{arXiv:9912078}[math.QA].

\bibitem{FK1994} L. Faddeev and R. Kashaev, Quantum dilogarithm, Mod. Phys. Lett. \textbf{9} (1994), 265--282.

\bibitem{FLK} R. Floreanini, J. LeTourneux, and L. Vinet, An algebraic interpretation of the continuous big q‐Hermite polynomials, \href{https://doi.org/10.1063/1.531216}{J. Math. Phys. \textbf{36}, 5091 (1995)}.

 \bibitem{GT}
     D. Gaiotto and J. Teschner, Irregular singularities in Liouville theory and Argyres-Douglas type
     gauge theories, I, \href{https://link.springer.com/article/10.1007/JHEP12(2012)050}{J. High Energ. Phys. 2012, 50 (2012)}; arXiv:1203.1052 [hep-th].

\bibitem{I2015} I.C.H.  Ip, Positive Representations of Non-simply-laced Split Real Quantum Groups,  Journal of Algebra, \textbf{425} (2015), 245--276.

\bibitem{I2015bis} I.C.H.  Ip, Positive Representations of Split Real Quantum Groups: The Universal R Operator, \href{https://doi.org/10.1093/imrn/rnt198}{ International Mathematics Research Notices, Volume 2015, Issue 1, 2015, Pages 240–287}.
     
\bibitem{I2020} I.C.H.  Ip, Positive Representations of Split Real Simply-Laced Quantum Groups,  Publ. Res. Inst. Math. Sci. \textbf{56} (2020), 603--646.
     
\bibitem{KLS2002} S. Kharchev, D. Lebedev, and M. Semenov-Tian-Shansky, Unitary representations of $U_q(sl(2,\mathbb{R}))$, the modular double and the multiparticle $q$-deformed Toda chains, \href{https://doi.org/10.1007/s002200100592}{Commun. Math. Phys. \textbf{225}, 573--609 (2002).}

\bibitem{KLS2010} R. Koekoek, P.A. Lesky, and R.F. Swarttouw, Hypergeometric orthogonal polynomials and their $q$-analogues, Springer Monographs in Mathematics. Springer-Verlag, Berlin, 2010. 

\bibitem{K96bis} H.T. Koelink, 8 Lectures on quantum groups and q-special functions, \href{https://arxiv.org/abs/q-alg/9608018}{arXiv:q-alg/9608018}[math.QA].

\bibitem{K96} H.T. Koelink, Askey-Wilson polynomials and the quantum SU(2) group: Survey and applications, \href{https://link.springer.com/article/10.1007/BF00047396}{Acta Applicandae Mathematica \textbf{44}, 295–352 (1996)}.

\bibitem{KS2000} E. Koelink and J.V. Stokman, The Askey--Wilson Function Transform Scheme, \href{https://link.springer.com/chapter/10.1007/978-94-010-0818-1_9}{Special Functions 2000: Current Perspective and Future Directions. NATO Science Series (Series II: Mathematics, Physics and Chemistry), vol 30. Springer, Dordrecht.}

\bibitem{Koo93} T.H. Koornwinder, Representations of the twisted SU(2) quantum group and some q- hypergeometric orthogonal polynomials, \href{https://doi.org/10.1016/S1385-7258(89)80020-4}{Proc. Kon. Ned. Akad. van Wetensch. Indag. Math., 51:97–117, 1989}.

\bibitem{Koo93bis} T.H. Koornwinder, Askey-Wilson polynomials as zonal spherical functions on the SU(2) quantum group, \href{https://doi.org/10.1137/0524049}{SIAM J. Math. Anal., 24:795–813, 1993}.

\bibitem{Koo94} T.H. Koornwinder, Compact quantum groups and q-special functions, Representations of Lie Groups and Quantum Groups (V. Baldoni and M. A. Picardello, Eds.), Pitman Research Notes in Mathematics Series 311, pp. 46–128, Longman Scientific $\&$ Technical, New York, 1994, \href{https://arxiv.org/abs/math/9403216}{arXiv:math/9403216}[math.CA].

\bibitem{KM2018} T.H. Koornwinder and M. Mazzocco, Dualities in the $q$-Askey scheme and degenerate DAHA, \href{https://doi.org/10.1111/sapm.12229}{Stud. Appl. Math. \textbf{141} (2018), 424--473.}

\bibitem{K91} N. Kurokawa, Multiple sine functions and Selberg zeta functions,  \href{https://doi.org/10.3792/pjaa.67.61}{Proc. Japan Acad. Ser. A Math. Sci. \textbf{67} 61-64 (1991)}.

\bibitem{LR1} J. Lenells and J. Roussillon, Confluent conformal blocks of the second kind, \href{https://doi.org/10.1007/JHEP06(2020)133}{J. High Energ. Phys. {\bf{2020}}, 133 (2020).}

\bibitem{LR2} J. Lenells and J. Roussillon,  The family of confluent Virasoro fusion kernels and a non-polynomial $q$-Askey scheme, \href{https://arxiv.org/abs/2011.07877}{arXiv:2011.07877}[math-ph].

\bibitem{M2014} M. Mazzocco, Non-Symmetric Basic Hypergeometric Polynomials and Representation Theory for Confluent Cherednik Algebras, \href{https://doi.org/10.3842/SIGMA.2014.116}{SIGMA \textbf{10} (2014), 116, (2014)}.

\bibitem{M2016} M. Mazzocco,  Confluences of the Painlev\'e equations, Cherednik algebras and q-Askey scheme,  Nonlinearity \textbf{29} (9),  2565--2608.

\bibitem{NM1} M. Noumi and K. Mimachi, Quantum 2-spheres and big $q$-Jacobi polynomials, \href{https://doi.org/10.1007/BF02096871}{Commun.Math. Phys. \textbf{128}, 521–531 (1990)}.

\bibitem{NM2} M. Noumi and K. Mimachi, Askey-Wilson polynomials and the quantum group $SU_q(2)$, \href{https://doi.org/10.3792/pjaa.66.146}{Proceedings of the Japan Academy, Series A, Mathematical Sciences, 66(6) 146-149 1990}.

  
\bibitem{NS2004} M. Noumi and J.V. Stokman, Askey--Wilson polynomials: an affine Hecke algebra approach, Laredo Lectures on Orthogonal Polynomials and Special Functions, Adv. Theory Spec. Funct. Orthogonal Polynomials, Nova Sci. Publ., Hauppauge, NY,
2004, 111--144.

\bibitem{PT1} B. Ponsot and J. Teschner, Liouville bootstrap via harmonic analysis on a noncompact quantum group, \href{https://arxiv.org/abs/hep-th/9911110}{arXiv:hep-th/9911110}[hep-th].
   	
\bibitem{PT2} B. Ponsot and J. Teschner, Clebsch--Gordan and Racah--Wigner coefficients for a continuous series of representations of $\mathcal{U}_q(\mathfrak{sl}(2,\mathbb R))$, \href{https://doi.org/10.1007/PL00005590}{Comm. Math. Phys.~\textbf{224} (2001), 613--655.}

\bibitem{Ribault} S. Ribault, Conformal field theory on the plane,  \href{https://arxiv.org/abs/1406.4290}{arXiv:1406.4290v5}[hep-th]. 

\bibitem{R2021} J. Roussillon, The Virasoro fusion kernel and Ruijsenaars' hypergeometric function, \href{https://doi.org/10.1007/s11005-020-01351-4}{Lett. Math. Phys.  \textbf{111}:7 (2021)}.

\bibitem{R1994} S. Ruijsenaars, Systems of Calogero-Moser type, \href{https://link.springer.com/chapter/10.1007/978-1-4612-1410-6_7}{Particles and fields, 251-352}.

\bibitem{R1996} S. Ruijsenaars, First order analytic difference equations and integrable quantum systems, \href{https://doi.org/10.1063/1.531809}{J. Math. Phys. \textbf{38}, 1069 (1997)}.

\bibitem{R1999} S. Ruijsenaars,  A Generalized Hypergeometric Function Satisfying Four Analytic Difference Equations of Askey--Wilson Type, \href{https://doi.org/10.1007/PL00005522}{Commun. Math. Phys. \textbf{206} (1999), 639--690}.

\bibitem{R2003bis} S. Ruijsenaars,  A Generalized Hypergeometric Function II. Asymptotics and D4 Symmetry, \href{https://link.springer.com/article/10.1007/s00220-003-0969-3}{Commun. Math. Phys. \textbf{243}, 389412 (2003)}.

\bibitem{R2003} S. Ruijsenaars, A Generalized Hypergeometric Function III. Associated Hilbert Space Transform, \href{https://doi.org/10.1007/s00220-003-0970-x}{Commun. Math. Phys. \textbf{243}, 413448 (2003)}.

\bibitem{R2011} S. Ruijsenaars, A Relativistic Conical Function and its Whittaker Limits, \href{https://doi.org/10.3842/SIGMA.2011.101}{SIGMA \textbf{7} (2011), 101, 54 pages}.

\bibitem{S1999} S. Sahi, Nonsymmetric Koornwinder Polynomials and Duality,  \href{https://www.jstor.org/stable/121102?seq=1}{Annals of Mathematics, vol. 150, no. 1, 1999, pp. 267--282}.

\bibitem{Sasa} T. Sasamoto, One-dimensional partially asymmetric simple exclusion process with open boundaries: orthogonal polynomials approach, \href{https://doi.org/10.1088/0305-4470/32/41/306}{J. Phys. A: Math. Gen. \textbf{32} 7109}.

\bibitem{U07} M. Uchiyama, Two-Species Asymmetric Simple Exclusion Process with Open Boundaries, \href{https://doi.org/10.1016/j.chaos.2006.05.013}{Chaos, Solitons $\&$ Fractals,  \textbf{35} 2 (2008)}.

\bibitem{USW} M. Uchiyama, T. Sasamoto, and M. Wadati, Asymmetric Simple Exclusion Process with Open Boundaries and Askey-Wilson Polynomials, \href{https://doi.org/10.1088/0305-4470/37/18/006}{J. Phys. A: Math. Gen. \textbf{37} 4985}.

\bibitem{W2000} S.L. Woronowicz, Quantum Exponential Function, \href{https://doi.org/10.1142/S0129055X00000344}{Rev. Math. Phys. \textbf{12} (2000) 873--920}.

\bibitem{Z92} A.S. Zhedanov,  “Hidden symmetry” of Askey-Wilson polynomials, \href{https://doi.org/10.1007/BF01015906}{Theor.Math.Phys., \textbf{89}, 1146–1157 (1991)}.

\end{thebibliography}
\end{document}